\documentclass[12pt]{amsart}

\usepackage[left=2cm,tmargin=2.5cm,bmargin=2.5cm,right=2cm]{geometry}

\usepackage[backref=page]{hyperref}
\usepackage[utf8]{inputenc}
\usepackage[initials,alphabetic]{amsrefs}
\usepackage{bm,amsmath,amssymb,amsthm}
\usepackage{enumerate}
\usepackage{amsfonts}
\usepackage{mathrsfs}
\usepackage[all]{xy} \usepackage{stmaryrd}

\usepackage{siunitx}
\makeatletter
\def\@xocreysp{ }
\makeatother

\usepackage{xcolor,color,url}
\definecolor{labelkey}{gray}{.25}
\definecolor{refkey}{gray}{.9}
\reversemarginpar
\usepackage{longtable}
\allowdisplaybreaks[1]

\theoremstyle{plain}
\newtheorem{theorem}{Theorem}[section]

\newtheorem{corollary}[theorem]{Corollary}
\newtheorem{proposition}[theorem]{Proposition}
\newtheorem{lemma}[theorem]{Lemma}

\newtheorem*{proposition*}{Proposition}
\newtheorem*{lemma*}{Lemma}
\newtheorem*{corollary*}{Corollary}

\theoremstyle{definition}
\newtheorem{example}[theorem]{Example}
\newtheorem*{examples*}{Examples}
\newtheorem*{example*}{Example}

\theoremstyle{remark}
\newtheorem{remark}[theorem]{Remark}

\newtheorem*{remarks*}{Remarks}
\newtheorem*{remark*}{Remark}

\usepackage{comment}

\numberwithin{equation}{section}

\usepackage{mathtools}

\usepackage{constants}
\let\Cst\C
\renewcommand{\Cl}[2][normal]{\Cst[#1]\labelconstant{#2}{\string #1}}
\newconstantfamily{prf}{
symbol=a,
format=\arabic,
reset={theorem},
}
\newconstantfamily{c}{
symbol=c,
format=\arabic,
}

\newcommand\ProvideMathOperator[2]{\ifdefined#1 \typeout{warning: DeclareMathOperator #1 already defined}\else\DeclareMathOperator{#1}{#2}\fi}

\usepackage{dsfont}  %\matds
\usepackage{mathrsfs} %\mathscr
\RequirePackage{xspace} % \newcommand{}{text\xspace}
\providecommand{\SB}{\backslash}
\renewcommand{\le}{\leqslant}
\renewcommand{\ge}{\geqslant}

\providecommand{\Mun}{\mathds{1}}

\providecommand{\mdemi}{\tfrac{1}{2}}
\providecommand{\eps}{\varepsilon}

\ProvideMathOperator {\GL} {GL}
\ProvideMathOperator {\PGL} {PGL}
\ProvideMathOperator {\SL} {SL}
\ProvideMathOperator {\PSL} {PSL}
\ProvideMathOperator {\SO} {SO}
\ProvideMathOperator {\TO} {O}
\ProvideMathOperator {\SU} {SU}
\ProvideMathOperator {\U} {U}
\ProvideMathOperator {\USp} {USp}
\ProvideMathOperator {\Sp} {Sp}
\ProvideMathOperator {\GSp} {GSp}
\ProvideMathOperator {\GO} {GO}
\ProvideMathOperator {\Spin} {Spin}

\ProvideMathOperator {\TLie} {Lie}
\ProvideMathOperator {\TAd} {Ad}
\ProvideMathOperator {\Res}  {Res}

\ProvideMathOperator {\Mchar}{char}
\ProvideMathOperator {\Mdiag}{diag}
\ProvideMathOperator {\MId}{Id}
\ProvideMathOperator {\vol}  {vol}
\ProvideMathOperator {\Mid}  {Id}
\ProvideMathOperator {\Mvol} {vol}
\ProvideMathOperator {\MTr}  {Tr}
\ProvideMathOperator {\Mtr}  {tr}

\providecommand{\A}{\mathbb{A}}

\renewcommand{\C}{\mathbb{C}}

\providecommand{\F}{\mathbb{F}}
\providecommand{\G}{\mathbb{G}}

\providecommand{\N}{\mathbb{N}}

\providecommand{\Q}{\mathbb{Q}}
\providecommand{\R}{\mathbb{R}}

\providecommand{\Z}{\mathbb{Z}}

\providecommand{\cB}{\mathcal{B}}
\providecommand{\cC}{\mathcal{C}}
\providecommand{\cD}{\mathcal{D}}

\providecommand{\cF}{\mathcal{F}}

\providecommand{\cL}{\mathcal{L}}

\providecommand{\cU}{\mathcal{U}}

\providecommand{\CmA}{\mathcal{A}}

\providecommand{\CmE}{\mathcal{E}}
\providecommand{\CmF}{\mathcal{F}}

\providecommand{\CmH}{\mathcal{H}}

\providecommand{\CmL}{\mathcal{L}}

\providecommand{\CmO}{\mathcal{O}}
\providecommand{\CmP}{\mathcal{P}}

\providecommand{\CmT}{\mathcal{T}}
\providecommand{\CmU}{\mathcal{U}}

\providecommand{\kL}{\mathfrak{L}}

\providecommand{\kR}{\mathfrak{R}}

\providecommand{\ka}{\mathfrak{a}}
\providecommand{\kb}{\mathfrak{b}}

\providecommand{\kk}{\mathfrak{k}}
\providecommand{\kl}{\mathfrak{l}}
\providecommand{\km}{\mathfrak{m}}

\providecommand{\kp}{\mathfrak{p}}
\providecommand{\kq}{\mathfrak{q}}

\providecommand{\FmF}{\mathfrak{F}}

\providecommand{\FmM}{\mathfrak{M}}

\providecommand{\Fma}{\mathfrak{a}}

\providecommand{\Fmf}{\mathfrak{f}}
\providecommand{\Fmg}{\mathfrak{g}}

\providecommand{\Fmk}{\mathfrak{k}}

\providecommand{\Fmm}{\mathfrak{m}}

\providecommand{\Fmo}{\mathfrak{o}}

\providecommand{\Tmc}{\mathbf{c}}

\renewcommand{\Im}{\operatorname{Im}}

\RequirePackage{xspace}
\renewcommand{\km}{\mathfrak{m}}
\newcommand{\tO}{\mathrm{O}}
\newcommand{\Jgeom}{J_{\mathrm{geom}}}
\ProvideMathOperator{\sgn}{sgn}
\ProvideMathOperator {\PO} {PO}
\newcommand{\oi}{\mathrm{oi}}
\newcommand{\glob}{\mathrm{glob}}

\begin{document}
\title[]{Sato-Tate equidistribution for families of Hecke--Maass forms on $\SL(n,\R) / \SO(n)$}
\author[]{Jasmin Matz}
\address{Department of Mathematical Science, Universitetsparken 5, 2100 Copenhagen, Denmark}
\email{matz@math.ku.dk}
\author[]{Nicolas Templier}
\address{Department of Mathematics, Cornell University, Malott Hall, Ithaca NY 14853-4201}
\email{templier@math.cornell.edu}
\date{\today}

\begin{abstract}
We establish the Sato-Tate equidistribution of Hecke eigenvalues of the family of
Hecke--Maass cusp forms on $\SL(n,\Z)\backslash \SL(n,\R) / \SO(n)$.
As part of the proof, we establish a uniform upper-bound for spherical
functions on semisimple Lie groups which is of independent interest.
For each of the principal, symmetric square and exterior square $L$-functions, we deduce
the level distribution with restricted
support of the low-lying zeros. We also deduce average estimates toward 
Ramanujan, including an improvement on the previous 
literature in the case $n=2$.
\end{abstract}

\maketitle
{\parskip=0pt \small \tableofcontents}

\section{Introduction}

Hecke--Maass cusp forms are certain eigenfunctions of the Laplace operator on the locally
symmetric space $\SL(n,\Z) \backslash \SL(n,\R)/
\SO(n)$. Beyond
the existence of such forms and structure theory, we want to
study spectral properties such as the Weyl's law, the distribution of Hecke eigenvalues,
temperedness, and average behavior in families.
Major difficulties in the analysis of the trace formula arise when the space is
not
compact, and when the test function is not of compact support. In this paper we deal with
both difficulties together and solve a long-standing equidistribution problem that
generalizes some classical results of Selberg for $n=2$.

Selberg~\cite{Selb56} introduced the trace formula and derived the Weyl's law for
$\SL(2,\Z)$, which is an asymptotic count for the family of Hecke--Maass cusp forms on
$\SL(2,\Z) \backslash \SL(2,\R)/
\SO(2)$, ordered by the size of their eigenvalues.
Much later, Sarnak~\cite{Sarnak87} observed that a variant of the same argument can be
used to establish a much more precise result concerning averages
of Hecke eigenvalues, namely the Sato-Tate
equidistribution for the same family of Hecke--Maass cusp forms. This
entails to
inserting test functions of varying support in the Selberg trace formula, such as a
Hecke operator $T_p$ with $p$ growing arbitrary large, and estimating
the geometric side in a concrete way.
Sarnak and Piatetskii-Shapiro~\cite{Sarnak87}*{\S4} then raised the problem of
generalizing this to the family
of Hecke--Maass cusp forms on $\SL(n,\Z) \backslash \SL(n,\R)/
\SO(n)$ for an arbitrary $n\ge 2$.
The present paper solves this problem.

Arthur generalized~\cite{Arthur:tr-I} the Selberg trace formula to general reductive
groups, by decomposing the geometric side into coarse equivalence classes
and constructing truncation operators. He further introduced in~\cite{Ar81} the
weighted orbital integrals, which enabled him to develop in~\cite{Ar86} a fine
expansion of the geometric side. Especially important for our purposes are
the splitting formulas established in~\cite{Ar81}*{\S6}, and the weighted measures
constructed in~\cite{Art88a}.

In addition to these fundamental results of Arthur, we use the recursive analysis by
Lapid--M\"uller~\cite{LM09} of the spectral side of the trace formula for
$\GL(n)$, the method and results by Shin and the second-named
author~\cite{ST11cf} on uniform estimates of certain orbital integrals, and the estimate
by the first-named author~\cite{Ma:coeff} of Arthur's global coefficients.
Then there are two important novelties. We develop in Part~\ref{partI} some uniform germ
estimates for orbital integrals of certain bi-$\SO(n)$-invariant unbounded functions.
This difficulty arises because the test function is not of compact support, and our
method is of independent interest, notably in view of the new estimates on Harish-Chandra
spherical
functions.
In Part~\ref{part2}, we establish uniform bounds for all the terms that appear in
Arthur's fine geometric expansion for $\GL(n)$.
This paper is also the first to establish a remainder term in the Weyl's law for
$\SL(n,\Z)$ when $n\ge 3$.

\subsection{Main result}
Let $G=\GL(n)$ and $\A=\R\times\A_f$ the ring of adeles of $\Q$. Let $K=K_{\infty}\cdot K_f$ the usual maximal compact subgroup of
\[
G(\A)^1:=\{g\in G(\A),\ |\det g|_{\A}=1\}
\]
given by $K_{\infty}=\TO(n)\subseteq G(\R)$ and $K_f=G(\hat{\Z})\subseteq G(\A_f)$. Let $K_{\infty}^{\circ}=\SO(n)\subseteq K_{\infty}$ be the identity component of $K_{\infty}$.
Let $\Pi_{\text{cusp}}(G(\A)^1)$ denote the set of irreducible unitary representations $\pi$ occurring  in the cuspidal part of $L^2(G(\Q)\backslash G(\A)^1)$.
Such $\pi$ can be uniquely extended to a cuspidal automorphic representation of $G(\A)$ whose central character has finite order, and conversely.
We say that $\pi$ is spherical (resp. unramified) if $\pi^{K_\infty}$ (resp. $\pi^{K_f}$) is non-zero. Unramified representations $\pi$ in $\Pi_{\text{cusp}}(G(\A)^1)$ correspond to unramified cuspidal automorphic representation of $G(\A)$ with trivial central character.

For $\pi\in \Pi_{\text{cusp}}(G(\A)^1)$ let $\lambda_{\pi}\in \ka^*_\C/ W$ denote the
infinitesimal character of  the archimedean component $\pi_{\infty}$.  Here $\ka$ is the
Lie algebra of the subgroup  $A\subset G(\R)^1$ of diagonal matrices with positive
entries and $W\simeq \mathfrak{S_n}$ is the Weyl group. For $t>0$ and a bounded
open
subset $\Omega\subseteq i\ka^*$ let
\[
\Lambda_\Omega(t):=2 \vol(G(\Q)\backslash G(\A)^1/K_f)|W|^{-1} \int_{t\Omega} 
\left|
\frac{\mathbf{c}(\rho)}
{\mathbf{c}(\lambda)}\right|^2 d\lambda
\]
where $\mathbf{c}$ denotes Harish-Chandra's $c$-function for 
$G(\R)^1:=\{g\in G(\R)\mid |\det g|_{\R}=1\}$, and $\rho$ is the half-sum of positive 
roots.
It is of order $t^d$ as $t\rightarrow\infty$  where $d=\dim_{\R} (G(\R)^1/K_{\infty}) =\frac{n(n+1)}{2}-1$.

We define characters $\chi_{\pm}: \TO(n)/\{\pm\MId\}=\PO(n)\longrightarrow \{\pm1\}$  as follows: $\chi_+$ is the trivial character, that is, $\chi_+(k)=1$ for all $k\in \PO(n)$, and $\chi_-(k)=\det k$ if $n$ is even and is the trivial character if $n$ is odd. Note that the group $\PO(n)$ is disconnected if $n$ is even in which case $ \{\det(\pm\MId)\}=\{1\}$ while $\PO(n)$ is connected if $n$ is odd in which case  $\{\det(\pm\MId)\}=\{\pm 1\}$.
 We  view $\chi_{\pm}$ as unitary characters on $K_{\infty}=\TO(n)$ which are invariant under $\TO(n)\cap Z(\R) = \{ \pm \MId\}$, and which are both trivial if $n$ is odd. Here $Z$ denotes the center of $G$.

If $\chi\in\{\chi_+,\chi_-\}$  then $\pi_{\infty}\otimes\chi$ defines another element in the unitary dual of $G(\R)$, and the set of fixed vectors $(\pi_{\infty}\otimes\chi)^{K_{\infty}}$ under $K_{\infty}$ is non-empty if and only if $\pi_{\infty}$ has $K_{\infty}$-type $\chi$, that is, if one vector in $\pi_\infty$ is $\chi$-invariant with respect to the $K_\infty$-action. If $\chi=\chi_+$, $(\pi\otimes\chi_+)^{K_{\infty}}=\pi^{K_{\infty}}\neq 0$ means that $\pi_{\infty}$ is spherical.

\begin{theorem}\label{thm:weyl}
For any integer $n\ge3$ there is an effective constant $A>0$ depending only on $n$, 
and for any
  non-empty $W$-invariant bounded open subset $\Omega\subseteq i\ka^*$ with
piecewise
$C^2$-boundary there is a constant $c_1>0$ such that the following holds. Let  $\chi\in\{\chi_+, \chi_-\}$  and let $\tau:G(\A_f)\longrightarrow\C$
be the characteristic function of a compact bi-$K_f$-invariant subset. Then
\[
\lim_{t\rightarrow\infty}\Lambda_\Omega(t)^{-1}\sum_{\substack{\pi\in\Pi_{\text{cusp}}(G(\A)^1), \\ \lambda_{\pi}\in t\Omega}} \dim (\pi_{\infty}\otimes \chi)^{K_{\infty}} \Mtr\pi_f(\tau)
= \sum_{\gamma\in Z(\Q)/\{\pm 1\}} \tau(\gamma),
\]
where $Z$ is the center of $G$.
Moreover, for all $t\ge1$ we have
\[
\left|\sum_{\substack{\pi\in\Pi_{\text{cusp}}(G(\A)^1), \\ \lambda_{\pi}\in t\Omega}} \dim (\pi_{\infty}\otimes \chi)^{K_{\infty}} \Mtr\pi_f(\tau)
-  \Lambda_\Omega(t) \sum_{\gamma\in Z(\Q)/\{\pm 1\}} \tau(\gamma) \right|
 \le c_1 \|\tau\|_{L^1(G(\A_f))}^{A} t^{d-1/2}.
\]
\end{theorem}

\begin{example}
If $\tau=\tau_0$ is the characteristic function of $K_f$, then Theorem~\ref{thm:weyl} 
is a
Weyl's law with remainder term,
\begin{equation}\label{weyllaw}
\left|
\{
\pi\in \Pi_{\text{cusp}}(G(\A)^1),\
\lambda_\pi\in t\Omega,\
(\pi_\infty\otimes \chi)^{K_\infty} \neq 0,\
\pi_f^{K_f}\neq 0
\}\right|
=
\Lambda_\Omega(t) + O(t^{d-1/2}),
\end{equation}
which was also established in an unpublished manuscript of the second-named author.
This is new already for $n=3$. The asymptotic Weyl's law, i.e., without remainder term,
was established for $\SL_3(\Z)$ by Miller~\cite{Miller01}, for $\SL_n(\Z)$, $n\ge 3$,
by M\"uller~\cite{Muller07}, and for quasi-split reductive groups
by~Lindenstrauss--Venkatesh~\cite{LV07}.
\end{example}

\begin{remark}
Conditional on the assumption that the lattice $G(\Q)\cap K_f$ is neat, which is false in
our case since $\PSL(n,\Z)$ is not torsion free, a stronger version of the
Weyl's law~\eqref{weyllaw} with remainder $t^{d-1}(\log t)^{\max(3,n)}$ is due to
Lapid--M\"uller~\cite{LM09}.
\end{remark}

\subsection{Hecke--Maass forms}\label{s:intro:maass}
We can restate the result classically in terms of Hecke--Maass cusp forms which are smooth functions $f$ on
	\[
	\SL_n(\Z) \backslash \SL_n(\R) / \SO_n(\R) \simeq G(\Q) \backslash G(\A)^1 / K_{\infty}^{\circ}K_f
	\]
	that are eigenfunctions of the Laplace operator, the Hecke operators and are cuspidal. Hecke--Maass cusp forms can be divided into \emph{even} and \emph{odd} forms. Let $W$ denote the Hecke operator corresponding to the double coset $\SL_n(\Z)\Mdiag(-1,1,\ldots,1)\SL_n(\Z)=\SL_n(\Z)\Mdiag(-1,1,\ldots,1) $.  Then $f$ is called even if $Wf=f$ and odd if $Wf=-f$. If $n$ is odd, there are no odd Hecke--Maass cusp forms so that all Hecke--Maass cusp forms are even. If $n$ is even, asymptotically half of all Hecke--Maass cusp forms are even and half are odd as follows from Theorem~\ref{thm:weyl} and also from~\cite{Muller07}.

An even Hecke--Maass cusp form $f$ generates a spherical unramified representation $\pi$ in $\Pi_{\text{cusp}}(G(\A)^1)$ and conversely, if $\pi$ is spherical unramified, $\pi^K$ is one-dimensional and $f$ is a non-zero element in $\pi^K$.  If $n$ is even, then the odd Hecke--Maass cusp forms $f$ generate unramified representations $\pi$ in $\Pi_{\text{cusp}}(G(\A)^1)$ of $K_{\infty}$-type $\chi_-$.  If conversely $\pi$ is unramified with $K_{\infty}$-type $\chi_-$, the subspace of $\pi^{K_f}$ transforming under $K_{\infty}$ according to $\chi_-$ is one-dimensional, and $f$ is a non-zero element in this subspace.

	For every prime $p$ we can attach the Satake parameter $\alpha_f(p)\in 
	{\C^{\times}}^{n}/\mathfrak{S}_n$ to $f$ which we denote in coordinates as 
	$\alpha^{(j)}_f(p)$. There is a Satake isomorphism between the algebra of 
	symmetric Laurent polynomials $\C[x^{\pm}_1,\cdots,x^{\pm}_n]^{\mathfrak{S}_n}$ and 
	the algebra of bi-$G(\Z_p)$-invariant functions on $G(\Q_p)$. If the polynomial 
	$\phi$ corresponds to $\tau_p$, then
\[
\phi(\alpha^{(1)}_f(p),\cdots,  \alpha^{(n)}_f(p)) = \phi(\alpha_f(p)) = \Mtr \pi_f(\tau_p 1_{K_f^{(p)}}),
\]
where $1_{K_f^{(p)}}$ denotes the characteristic function of $K_f$ away from $p$.

Since the central character is trivial, $\alpha_f^{(1)}(p)\cdots  \alpha_f^{(n)}(p)=1$. Let $\mu_p$ be the unramified Plancherel measure of $\PGL_n(\Q_p)$. It is supported on the elements $\alpha\in {S^1}^n/\mathfrak{S}_n$  such that $\alpha^{(1)}\cdots \alpha^{(n)}=1$ and for any corresponding pair $\phi \leftrightarrow \tau_p$,
\[
\int_{  {S^1}^n/\mathfrak{S}_n} \phi \mu_p = \vol(\Z_p)^{-1} \int_{Z(\Q_p)} \tau_p(z)dz.
\]
An exact formula for $\mu_p$ is given by Macdonald~\cite{book:macd71}.
Our main theorem in classical terms is:
\begin{theorem}\label{th:heckemaass} For $n\ge3$, and any $\phi \in \C[x^{\pm}_1,\cdots,x^{\pm 
}_n]^{\mathfrak{S}_n}$ with coefficients less than one, any prime $p$ and any $t\ge 1$,
\begin{equation*}
	\left|
	\sum_{f:~\lambda_f\in t\Omega}
	 \phi(\alpha_f(p))
	 - \Lambda_\Omega(t) \int_{  {S^1}^n/\mathfrak{S}_n} \phi \mu_p
	\right|
	\le c_1 p^{A\deg(\phi)} t^{d-1/2}
\end{equation*}
where $f$ runs through either even or odd Hecke--Maass cusp forms if $n$ is even, 
and through all Hecke--Maass cusp forms if $n$ is odd. Here $\deg(\phi)$ satisfies 
$\deg(x_1\cdots x_n)=0$ and $\deg(e_i)=1$ for all the other non-constant elementary 
symmetric polynomials $1\le i\le n-1$.
\end{theorem}
\begin{proof}
	The first term agrees with that of Theorem~\ref{thm:weyl}. For the second term let $\tau_p$ correspond to $\phi$ under the Satake correspondence. Then
\[
\sum_{\gamma\in Z(\Q)/\{\pm 1\}}
\tau_p(\gamma) 1_{K_f^{(p)}}(\gamma)
=
\sum_{z\in Z(\Q_p)/Z(\Z_p)} \tau_p(z)
=  \vol(\Z_p)^{-1} \int_{Z(\Q_p)} \tau_p(z)dz.
\]

 For any $\xi=(\xi_1,\ldots,\xi_n)\in \Z^n$ denote by $e_{p,\xi}$ the polynomial that correspond under the Satake correspondence to the indicator function $\tau_{p,\xi}$ on the double coset
\[
K_pp^{\xi} K_p:=G(\Z_p)\Mdiag(p^{\xi_1}, \ldots, p^{\xi_n})G(\Z_p),
\]
 see Section~\ref{sec:arthur-selberg}. The polynomials $\{e_{p,\xi}\}$ form a basis of the symmetric polynomial algebra. We have
\[
\| \tau_{p,\xi} \|_{L^1(G(\Q_p))} \asymp p^{\langle \xi, 2\rho \rangle}
\]
which follows from \cite{Gross:satake}*{Prop.7.4}
where $\rho$ is half-sum of positive roots. Then $\langle \xi, 2\rho \rangle \le (n-1)^2 (\max\xi-\min\xi)$.
On the other hand we have $\deg(e_{p,\xi})=\max \xi - \min \xi$, since $e_{p,\xi}$ is a linear combination of monomial symmetric polynomials, which concludes the proof.
\end{proof}

We now turn to Fourier coefficients which occur often in the study of Hecke--Maass cusp forms. To obtain analogous results for the distribution of Fourier coefficients we shall simply insert Schur polynomials for $\phi$ in Theorem~\ref{th:heckemaass} as we now explain.

Every Hecke--Maass cusp form $f$ is generic. We denote the normalized Fourier
coefficients as $a_f$ with $a_f(1)=1$.  We say that $\nu=(\nu_1,\ldots,\nu_n)\in \Z^n$ is
dominant if $\nu_1\ge \cdots \ge \nu_n\ge 0$. For each dominant $\nu$ there is a Schur
polynomial $s_\nu  \in \C[x_1,\ldots,x_n]^{\mathfrak{S}_n}$ and the
Shintani/Casselman-Shalika formula reads $a_f(p^\nu)=s_\nu(\alpha_f(p))$. Precisely, for
any $\nu \in \Z^n$
\[
a_f(p^{\nu_1}, \ldots, p^{\nu_n})
=
\begin{cases}
	s_\nu(\alpha^{(1)}_f(p), \ldots,\alpha^{(n)}_f(p) ) &
	\text{if $\nu$ is dominant,}\\
	0&
	\text{otherwise.}
\end{cases}
\]
The Schur polynomials form a basis of the algebra of symmetric polynomials.

It is traditional to consider the coefficients $A_f$ which are directly related to the $a_f$ by
\[
A_f(m_1,\ldots,m_{n-1}) = a_f(m_1m_2\cdots m_{n-1}, \ldots,m_{1},1)
\]
for all $m_1,\ldots,m_{n-1}\in \Z_{\ge 1}$.
\begin{theorem}\label{th:fouriermaass} For any integers $m_1,\ldots,m_{n-1}\in \Z_{\ge 1}$ and any $t\ge 1$,
\[
\sum_{\substack{
	\text{$f$ Hecke-Maass}\\
	\lambda_f\in t\Omega}
	}
	A_f(m_1,\ldots, m_{n-1})
	= \Lambda_\Omega(t) \gamma(m_1, \ldots, m_{n-1})
	+ O((m_1 \cdots m_{n-1})^{A}t^{d-1/2}).
\]
Here $\gamma(m_1,\ldots,m_{n-1})$ is multiplicative in each of the $n-1$ variables.
Moreover for any prime $p$ and any $\nu\in \Z^n$ such that $\nu_1\ge \ldots \ge \nu_n=0$,
\[
 \gamma(p^{\nu_{n-1}},p^{\nu_{n-2}-\nu_{n-1}}, \ldots, p^{\nu_1-\nu_2}) =
p^{-\langle\nu,\rho\rangle} P_{0,\nu}(p)
\]
where $P_{0,\nu}$ is the Kazhdan-Lusztig polynomial with parameters $0,\nu$ in $\Z^n$ viewed inside the affine Weyl group of type $A_n$.
\end{theorem}

\begin{example}
If $\nu$ is a fundamental weight then $\int s_\nu \mu_p=0$, see~\cite{Gross:satake} where it is furthermore explained that the conceptual reason for this vanishing is that all the fundamental representations of $\GL(n)$ are minuscule.  Thus $\gamma(m_1,\ldots,m_{n-1})$ is zero if $m_1\cdots m_{n-1}$ is square-free and not equal to one. For example if $n=3$, then  $\gamma(1,p)=0$ which corresponds to the average of the coefficients
\[
A_f(1,p)=a_f(p,1,1)
= s_{(1,0,0)}(\alpha_f(p))
= \alpha^{(1)}_f(p) + \alpha_f^{(2)}(p) +  \alpha^{(3)}_f(p).
\]
More information on the polynomial $P_{0,\nu}$ can be found in the discussion
following~\cite{FGKV:geometric-Whittaker}*{Prop.6.3}.
\end{example}
\begin{proof}[Proof of Theorem \ref{th:fouriermaass}]
	We only provide the details when the $m$'s are powers of a prime $p$, the general case being similar. Since
\[
a_f(p^{\nu_1},\ldots,p^{\nu_{n-1}},1 )
 =  s_\nu(\alpha^{(1)}_f(p), \ldots,\alpha^{(n)}_f(p) ),
\]
and $\deg(s_\nu) = \nu_1$, Theorem~\ref{th:heckemaass} yields
\[
\sum_{\substack{
	\text{$f$ Hecke-Maass}\\
	\lambda_f\in t\Omega}
	}
A_f(p^{\nu_{n-1}},p^{\nu_{n-2}-\nu_{n-1}}, \ldots, p^{\nu_1-\nu_2})
	= \Lambda_\Omega(t) \int_{{S^1}^n/ \mathfrak{S}_n} s_\nu \mu_p
+ O(p^{A \nu_1 }t^{d-1/2}).
\]
The integral against the Plancherel measure is equal to $\gamma(p^{\nu_{n-1}},p^{\nu_{n-2}-\nu_{n-1}}, \ldots, p^{\nu_1-\nu_2})$, so the second assertion of the theorem follows from the formula
\[
\int_{{S^1}^n/ \mathfrak{S}_n} s_\nu \mu_p =  p^{-\langle \nu,\rho\rangle} P_{0,\nu}(p)
\]
which can be found in the work of S.-I.~Kato~\cite{Kato82}, see
also~\cite{Gross:satake}*{Prop.4.4}.
\end{proof}

We note that it is not difficult to deduce a similar result for a product of Fourier coefficients $A_f$. We insert a product of Schur polynomials in Theorem~\ref{th:heckemaass} in which case the main term can be computed in terms of Littlewood-Richardson coefficients.

\begin{example}\label{ex:n=2-thm1}
	If $n=2$, then Theorem~\ref{th:fouriermaass} is established by Sarnak~\cite{Sarnak87}
	and the analogous results for holomorphic modular forms is established by Serre.
 For $m\in \Z_{\ge 1}$   the Fourier coefficients $A_f(m)=a_f(m,1)$  coincide with the eigenvalues of the Hecke operator $T_m$ because $f$ is unramified, thus a newform.  We have $\ka^*\simeq \R$ and without loss of generality we may choose $\Omega=(-1,1)$. The condition $\lambda_f\in t\Omega$ means that the Laplace eigenvalue of $f$ is greater than $\frac14$ and less than $\frac14+ t^2$. We have $\Lambda_{\Omega}(t)\sim t^2/12$ and for $m=1$ the result reduces to the Weyl's law established by Selberg. For general $m$ the main term involves
\[
\gamma(m)= |m|^{-\frac{1}{2}} \delta_{m=\square}
\]
 where $\delta_{m=\square}$ is one if $m$ is a perfect square and zero otherwise.
\end{example}

For $n=2$, related results with the added factor $L(1,\mathrm{sym}^2
f)^{-1}$ are obtained via the Kuznetsov trace formula by Bruggeman,
Deshouillers--Iwaniec. For $n=3$, there are results by Goldfeld--Kontorovich and
Blomer~\cite{Blomer:Kuznetsov} via a generalization of the Kuznetsov trace formula, again
with the addition of the arithmetic weights
$L(1,\mathrm{Ad} f)^{-1}$.

The method by Luo~\cite{Luo:nonvanishing-Weyl} for removing
weights is based on large sieve inequalities, and has been employed
in~\cite{Lau-Wang:Sato-Tate-gl2} to establish Theorem~\ref{thm:weyl} for $n=2$.
For $n=3$, after a version of the present paper was posted on
arXiv, Buttcane--Zhou~\cite{Buttcane-Zhou:weight} succeeded in
establishing Theorem~\ref{thm:weyl} in a similar way, using a new
version of the Kuznetsov trace formula for $\SL(3,\Z)$ due to Blomer--Buttcane.

\subsection{Average bounds towards Ramanujan}
The Plancherel measure $\mu_p$ on the unitary dual of $\PGL_n(\Q_p)$ is supported on 
the tempered spectrum. As a consequence of the Sato-Tate equidistribution 
Theorem~\ref{thm:weyl} for families we can deduce quantitative bounds towards 
Ramanujan.
\begin{corollary}\label{cor:ramanujan} (i) There is an effective constant $c>0$ 
(depending only on $n$)
such that for any $t\ge 1$, any $\theta\ge 0$, and any prime $p$, 
	\begin{equation}\label{plain-ramanujan}
	\left|
	\{
f,\
\|\lambda_f\|\le t
,\ \max_{1\le j\le n} \log_p |\alpha^{(j)}_f(p)|  > \theta
\}
\right|
\ll p^{2\theta} t^{d-c\theta},
	\end{equation}
	where $f$ runs through Hecke--Maass cusp forms on $\SL(n,\Z)\backslash \SL(n,\R)/
	\SO(n)$, and the implied constant depends on $n$ only.

(ii) Moreover, the following stronger bound holds
\begin{equation}\label{strong-ramanujan}
\left|
	\{
f,\
\|\lambda_f\|\le t
,\ \max_{1\le j\le n} \log_p |\alpha^{(j)}_f(p)|  > \theta
\}
\right|
\ll t^{d - x\left\lfloor \frac{c}{x}\right\rfloor \theta},
\end{equation}
where $x:= \dfrac{2 \log p}{\log t}$ and $\left\lfloor \frac{c}{x}\right\rfloor$ denotes the largest 
integer less or equal than $\frac cx$.
\end{corollary}

In words the corollary says that exceptions to Ramanujan for Hecke-Maass forms 
are sparse. 
This generalizes a result obtained by Sarnak~\cite{Sarnak87}*{Thm.1.1} for $n=2$ 
and $p$ 
fixed and
Blomer--Buttcane--Raulf~\cite{Blomer-Buttcane-Raulf}*{Thm.3} for $n=3$. The LHS is 
zero
for $\theta > \frac{1}{2} - \frac{1}{n^2+1}$ (due to Luo--Rudnick--Sarnak and Serre)
and conjecturally for $\theta=0$ (generalized Ramanujan conjecture).

If $p$ is fixed, then the term $p^{2\theta}$ can be dropped in~\eqref{plain-ramanujan} 
and the bound becomes $\ll t^{d-c\theta}$.
The bound~\eqref{plain-ramanujan} is a simple version of the 
bound~\eqref{strong-ramanujan}, using that $x \left\lfloor \frac{c}{x}\right\rfloor\ge c-x$.
The function $\R_{\ge 0}\to [0,c]$, $x\mapsto x \left\lfloor \frac{c}{x}\right\rfloor$ 
is piecewise linear, with discontinuous jumps for every $x$ that is an 
integer fraction of $c$. It is equal to $c$ at $x=0$.
It is equal to $0$ for every $x> c$, which means that the 
bound~\eqref{strong-ramanujan} is trivial for $p^2 > t^c$.

\begin{remark*}
Independently, Lau-Wang~\cite{Lau-Wang:L1}*{Thm.7.3} have obtained the same 
result as Corollary~\ref{cor:ramanujan} with a different proof.
Note that their bound is stated as $\ll t^{d-c\theta}$, which is incorrect since the 
term 
$p^{2\theta}$ cannot be dropped in our bound~\eqref{plain-ramanujan}
and the product 
$x\left\lfloor
\frac{c}{x}
\right\rfloor$ cannot be replaced by $c$ in our bound~\eqref{strong-ramanujan} because 
Theorem~\ref{thm:weyl} is trivial when $p$ is large. The source of the 
error is that 
the 
bound $||\tau_h||\le T^{1/4A}$ on the last line of page 788 of \cite{Lau-Wang:L1} 
doesn't hold for $p>T^{\delta_0}$ because it implies $h=1$, whereas $||\tau_1||$ is 
unbounded as the prime $p$ grows.
\end{remark*}

We offer the following additional corollary which is non-trivial even in the 
region 
$p>t^{\frac c2}$.
Neither Corollary~\ref{cor:ramanujan} nor Corollary~\ref{cor:largep} implies the 
other.
Both corollaries are obtained from applying Theorem~\ref{th:heckemaass} in a 
situation where the main term and the remainder term are balanced (that is 
we optimize the choice of degree of our test polynomials so that the two terms 
are close). For Corollary~\ref{cor:ramanujan} the main term is greater 
than the remainder term, whereas for Corollary~\ref{cor:largep} the main term 
is smaller than the bound for the remainder term.

\begin{corollary}\label{cor:largep} (i) For any $t\ge 1$, any $\theta\ge 0$, and any 
prime $p > t^{\frac c2}$, 
	\begin{equation}\label{largep-ramanujan}
	\left|
	\{
f,\
\|\lambda_f\|\le t
,\ \max_{1\le j\le n} \log_p |\alpha^{(j)}_f(p)|  > \theta
\}
\right|
\ll t^{d - \frac{x}{2} (\frac{1}{c} - \frac{1}{x} - 2 \theta)}.
	\end{equation}

(ii) More generally, for any $t\ge 1$, any $\theta\ge 0$ 
and any prime $p$,
\begin{equation}\label{plusone-ramanujan}
\left|
	\{
f,\
\|\lambda_f\|\le t
,\ \max_{1\le j\le n} \log_p |\alpha^{(j)}_f(p)|  > \theta
\}
\right|
\ll t^{d - \frac12 + x (\left\lfloor \frac{c}{x}\right\rfloor+1) ( \frac{1}{2c} - \theta ) },
\end{equation}
where again the implied constant depends on $n$ only.
\end{corollary}

The bound~\eqref{largep-ramanujan} is a version of the 
bound~\eqref{plusone-ramanujan} in the restricted range $p > t^{\frac c2}$, using 
that $\left\lfloor \frac{c}{x}\right\rfloor = 0$ for $x>c$.
Corollary~\ref{cor:largep} is worse than the trivial bound $t^d$ for $p^2 > 
t^{c/(1 - 2c\theta)}$.

The proofs of Corollary~\ref{cor:ramanujan} and Corollary~\ref{cor:largep} are 
given in \S\ref{sub:pf-ram}, where we obtain $c=1/(2A)$.
We shall construct test polynomials $\phi\in 
\C[x^\pm_1,\ldots,x^\pm_n]^{\mathfrak S_n}$ with uniformly bounded coefficients which are 
optimal 
for establishing average 
bounds towards Ramanujan. Namely $\phi(\alpha) \ge |\alpha|_\infty^{\deg(\phi)}$ when evaluated 
on Satake 
parameters $\alpha \in (\C^\times)^n/\mathfrak S_n$ which are unitary in the sense that 
$\alpha=1/\overline{\alpha}$. These polynomials are of independent interest, and the 
present paper seems to be the first to give an optimal construction in higher 
rank.

\subsection{Revisiting the case n=2}
The Kuznetsov trace formula for $\SL_2(\Z)$ yields the best known average bounds 
toward Ramanujan in the case $n=2$.
We have that 
\[
A_f(p) = a_f(p,1) = s_{(1,0)}(\alpha^{(1)}_f(p), \alpha^{(2)}_f(p))
=\alpha^{(1)}_f(p) + \alpha^{(2)}_f(p)
\] 
is the normalized eigenvalue of $f$ for 
the $p$th Hecke operator.
Blomer--Buttcane--Raulf~\cite{Blomer-Buttcane-Raulf}*{p.902} have established the 
following.
For any $\epsilon,\theta >0$, any $t\ge 1$ and any prime $p$,
\begin{equation}\label{BBR}
 \left|
	\{
f,\
\|\lambda_f\|\le t
,\ |A_f(p)| > 2p^\theta
\}
\right|
\ll_{\epsilon} t^{2-
\frac{2 \log p}{\log t} \left\lfloor
\frac{4\log t}{\log p}\right\rfloor \theta + \epsilon},
\end{equation}
where the multiplicative constant depends on $\epsilon$ only.
In comparing notation with \cite{Blomer-Buttcane-Raulf}, write 
$\alpha=2p^\theta$, which is equivalent to $\theta = \frac{\log \alpha/2}{\log p}$.
As before, since $\left\lfloor
\frac{4\log t}{\log p}\right\rfloor \le \frac{4\log t}{\log p}$, a consequence of~\eqref{BBR} 
is that
\begin{equation*}
 \left|
	\{
f,\
\|\lambda_f\|\le t
,\ |A_f(p)| > 2p^\theta
\}
\right|
\ll_{\epsilon} p^{2\theta}
t^{2-8\theta + \epsilon}.
\end{equation*}

We shall  make some improvements to the above bound~\eqref{BBR}. Since 
\[
|A_f(p)| > p^{\theta} + p^{-\theta}  \iff \max_{j=1,2} \log_p 
|\alpha^{(j)}_f(p)|  > \theta,
\]
we deduce that
\[
  \left|
	\{
f,\
\|\lambda_f\|\le t
,\ |A_f(p)| > 2p^\theta
\}
\right|
\le 
\left|
	\{
f,\
\|\lambda_f\|\le t
,\ \max_{j=1,2} \log_p |\alpha^{(j)}_f(p)|  > \theta
\}
\right|.
\]
We shall establish in \eqref{SL2-8x} that the bound~\eqref{BBR} holds for the latter 
quantity, by 
modifying the combinatorics of Hecke eigenvalues 
in~\cite{Blomer-Buttcane-Raulf}*{\S2}.
Again, a consequence is that for any $\epsilon,\theta >0$, any $t\ge 1$ and any 
prime $p$,
\begin{equation*}\label{t8theta}
  \left|
	\{
f,\
\|\lambda_f\|\le t
,\ \max_{j=1,2} \log_p |\alpha^{(j)}_f(p)|  > \theta
\}
\right|
\ll_{\epsilon} 
p^{2\theta}
t^{2-8\theta + \epsilon}.
\end{equation*}
In particular, if $p$ is fixed then the term $p^{2\theta}$ can be dropped and the 
bound becomes $\ll_\epsilon t^{2-8\theta + \epsilon}$.

Additionally, we offer the complementary bound~\eqref{SL2-plusone} which is 
obtained
by increasing the degree of the test polynomial:
The bound \eqref{SL2-8x} is obtained by rounding down the ratio $\frac{4\log 
t}{\log p}$ to 
the nearest integer, whereas the bound~\eqref{SL2-plusone} is obtained by 
\emph{rounding it up}, see \S\ref{sub:proof-SL2}.
Neither \eqref{SL2-8x} nor \eqref{SL2-plusone} implies the other.
\begin{theorem}\label{t:SL2-Ramanujan}
Suppose $n=2$. For any $\epsilon,\theta >0$, any $t\ge 1$ and any prime $p$,
\begin{equation}\label{SL2-8x}
  \left|
	\{
f,\
\|\lambda_f\|\le t
,\ \max_{j=1,2} \log_p |\alpha^{(j)}_f(p)|  > \theta
\}
\right|
\ll_{\epsilon} t^{2- 2 \frac{\log p}{\log t} \left\lfloor \frac{4 \log t}{\log p} \right\rfloor \theta
+ \epsilon},
\end{equation}
\begin{equation}\label{SL2-plusone}
  \left|
	\{
f,\
\|\lambda_f\|\le t
,\ \max_{j=1,2} \log_p |\alpha^{(j)}_f(p)|  > \theta
\}
\right|
\ll_{\epsilon} p^{
( \left\lfloor \frac{4 \log t}{\log p} \right\rfloor +1) \cdot (\frac12-2\theta)
+ \epsilon},
\end{equation}
where the multiplicative constants depend on $\epsilon$ only.
\end{theorem}

An interesting feature of~\eqref{SL2-plusone} is that it is non-trivial even in 
the region $p> t^4$, where it becomes $\ll_\epsilon p^{\frac12-2\theta
+ \epsilon}$.
In fact, it becomes trivial in the region $p > t^{\frac{4}{1-4\theta+2\epsilon}}$.

In the region $p\le t^4$, we may take the minimum of~\eqref{SL2-8x} and 
\eqref{SL2-plusone} to deduce the following interesting bound.
\begin{corollary}\label{cor:ranges}
For any $\epsilon,\theta >0$, any integer $m\ge 1$, any $t\ge 1$ and any prime $p$ with $p\le 
t^{4/m}$, 
\begin{equation}\label{ranges}
   \left|
	\{
f,\
\|\lambda_f\|\le t
,\ \max_{j=1,2} \log_p |\alpha^{(j)}_f(p)|  > \theta
\}
\right|
\ll_{\epsilon} 
t^{2(1+m)(1-4\theta)/(1+m-4\theta)+\epsilon},
\end{equation}
whrer the multiplicative constant depends on $\epsilon$ only.
\end{corollary}

\begin{examples*} For $m=1,2,4$, the corollary specializes to the following 
bounds:

(i) For any $\epsilon,\theta >0$, and any $p\le t^4$, 
\begin{equation*}\label{p4let}
   \left|
	\{
f,\
\|\lambda_f\|\le t
,\ \max_{j=1,2} \log_p |\alpha^{(j)}_f(p)|  > \theta
\}
\right|
\ll_{\epsilon} 
t^{2(1-4\theta)/(1-2\theta)+\epsilon}.
\end{equation*}

(ii) For any $\epsilon,\theta >0$, and any $p\le t^2$, 
\begin{equation*}\label{p2let}
   \left|
	\{
f,\
\|\lambda_f\|\le t
,\ \max_{j=1,2} \log_p |\alpha^{(j)}_f(p)|  > \theta
\}
\right|
\ll_{\epsilon} 
t^{6(1-4\theta)/(3-4\theta)+\epsilon}.
\end{equation*}

(iii) For any $\epsilon,\theta >0$, and any $p\le t$, 
\begin{equation}\label{corrected-BBR}
   \left|
	\{
f,\
\|\lambda_f\|\le t
,\ \max_{j=1,2} \log_p |\alpha^{(j)}_f(p)|  > \theta
\}
\right|
\ll_{\epsilon} 
t^{10(1-4\theta)/(5-4\theta)+\epsilon}.
\end{equation}
\end{examples*}

\begin{remark*}
The third region $p\le t$ had also been considered 
in~\cite{Blomer-Buttcane-Raulf}*{Prop.1}. 
However there is an inaccuracy 
due to not taking into account that $\left\lfloor \frac{4\log t}{\log p} \right\rfloor$ is a 
priori smaller than $\frac{4\log t}{\log p}$
in the exponent of~\eqref{BBR} (the same also 
affects~\cite{Blomer-Buttcane-Raulf}*{Thm.1}). 
A remedy is to calculate that the infimum of $\frac{\log p}{\log t} \left\lfloor 
\frac{4 \log t}{\log p} \right\rfloor$ for $\frac{\log p}{\log t}\in 
[0,1]$ is equal to $\frac{16}{5}$ and attained as $\frac{\log p}{\log t}$ tends to 
$\frac{4}{5}^-$.
This yields a bound $\ll_\epsilon t^{2-\frac{32}{5}\theta + \epsilon}$, which is 
weaker than our bound~\eqref{corrected-BBR}.
\end{remark*}

\subsection{Main ideas for the proof of Theorem \ref{thm:weyl}} \label{sub:ideas}
The main tool to prove Theorem \ref{thm:weyl} will be Arthur--Selberg's trace formula for $\GL(n)$ in which we insert a suitable family of test functions.
We are facing the two difficulties that the lattice subgroup $\SL(n,\Z)$ is not
cocompact, and that the test functions are not of uniform compact support.
Functions that are not of compact support occur frequently for $\GL(2)$ and are more
recent
in higher rank~\cite{Blomer:Kuznetsov,ST11cf}.

Since $\SL(n,\Z)$ is not cocompact, there is a continuous spectrum which complicates the
analysis of the cuspidal spectrum. A lot of work has been done on this problem starting
from the pioneering works of Selberg and Langlands on Eisenstein series. Thanks to the
description of the
discrete spectrum of $\GL(n)$ by Moeglin--Waldspurger, a satisfactory grasp of the spectral side of the trace formula is achieved in~\cite{Muller07,LM09}.

Our work then happens on the geometric side of the trace formula.
As we shall explain now,
the approach is similar to that of~\cite{ST11cf}, with several important additions.
In the remainder of this introduction we shall focus only on the trivial 
$K_{\infty}$-type $\chi=\chi_+$. The global test functions have the form $(f \cdot 
\tau)_{|G(\A)^1}$ where $\tau$ is as in Theorem~\ref{thm:weyl} (a Hecke operator) and 
$f$ is a smooth bi-$K_\infty$-invariant function on $G(\R)$ compactly supported mod 
center. The support of the global test function is not uniformly bounded 
because $\tau$ is varying. To still make use of the Arthur--Selberg trace formula, 
this demands a good understanding of the behavior of the orbital integrals over 
$G(\A)$-conjugacy classes of varying elements $\gamma\in G(\Q)$.

Arthur's fine geometric expansion and splitting formula for $(G,M)$-families 
yield a decomposition of global orbital integrals as a sum  over certain Levi 
subgroups $M, L_1, L_2$ containing $M$ of products of three terms
\[
a^M(\gamma,S) J_M^{L_1}(\gamma,f^{L_1})J_M^{L_2}(\gamma,\tau_S^{L_2}),
\]
where $a^M(\gamma, S)$ are certain global coefficients, $f^{L_1}$ is a function on $L_1(\R)$ constructed from $f$, $\tau_S^{L_2}$ is a function on $L_2(\Q_S)$ constructed from $\tau$, $J_M^{L_1}(\gamma, f^{L_1})$ is a weighted orbital integral on the $L_1(\R)$-orbit of $\gamma$, and $J_M^{L_2}(\gamma, \tau_S^{L_2})$ is a weighted orbital integral on the $L_2(\Q_S)$-orbit of $\gamma$; see Section~\ref{sec:fine}. Here $S$ is a finite set of primes such that $\tau$ equals the unit element in the Hecke algebra at the primes outside of $S$. Our estimates need to be polynomial in $S$ and $\tau$, i.e., the remainder term should be at most a power of $\|\tau\|_{L^1(G(\Q_S))}$.

We are going to estimate these three terms separately:

\noindent (i) The archimedean orbital integrals $J_M^{L_1}(\gamma,f^{L_1})$ are the 
subject of Part~\ref{partI}. We establish an estimate that is polynomial in $\gamma$ 
with a specific dependence on the function $f$,  see the summary in the next 
\S\ref{sub:intro:germ}.
A similar estimate was obtained in an unpublished manuscript of the 
second-named author on the Weyl's law with remainder term for $\SL(n,\Z)$, however at 
the time without the polynomial dependence on $\gamma$, which we now achieve in this 
paper, notably thanks to Proposition~\ref{prop:spher:fct:bound}.

\noindent (ii) The first-named author~\cite{Ma:coeff} has established an upper-bound for Arthur's global coefficients $a^M(\gamma,S)$ that is polynomial in $\gamma$ and $S$.
Recent works of Chaudouard--Laumon~\cite{Laumon-Chaudouard} and Chaudouard~\cite{Chaudouard:comptage,Chaudouard:unipotent} provide exact formulas, and logarithmic upper-bounds, however these cover only a limited number of cases which is not sufficient for our purpose.

\noindent (iii) We establish in Part~\ref{part2} uniform bounds for the non-archimedean
orbital integrals $J_M^G(\gamma, \tau_S)$ that are polynomial in $\gamma$ and $\tau$,
i.e., a power of $D^{G}(\gamma_s)$ (see below for a definition) and
$\|\tau_S\|_{L^1(G(\Q_S))}$.
These bounds originate from the work of Shin and the second-named 
author~\cite{ST11cf}*{\S7}.
We provide in this paper a complete treatment, which produces effective constants and is 
independent of
motivic integration methods~\cite{ST11cf}*{App.~B}.

\subsection{Germ estimates for certain unbounded functions}\label{sub:intro:germ}
Let $G=\GL_n(\R)^1$ in this subsection. Let $T_0\subseteq G$ be the split maximal torus of diagonal elements. A Levi subgroup $M\subseteq G$ is called semi-standard if $T_0\subseteq M$.

We want to estimate the weighted orbital integral $J^G_M(\gamma,f)$ in a uniform way in 
both $\gamma$ and $f$. The uniformity in $\gamma$ is closely related with germ expansions. 
Germ expansions occur for example when $\gamma$ is regular semisimple and approaches a 
singular element and have been first studied in the work of Harish-Chandra~\cite{HC57}. 
The descent formulas apply to a fixed $f$, $M=G$ and $\gamma$ a varying semisimple element.
An important result of Harish-Chandra, that we shall generalize, is that for any $f\in \cC^\infty_c(G)$, there is a constant $C(f)>0$ such that for every semisimple element $\gamma\in G$,  $|J^G_G(\gamma,f)|\le C(f)$.
The dependence in $f$ is studied by 
Duistermaat--Kolk--Varadarajan~\cite{DKV83}, also in the case $M=G$, but for $\gamma$ 
a fixed semisimple element.

Our main result is an estimate where both $\gamma$ and $f$ vary. The uniformity in 
$\gamma$ is needed because we consider varying Hecke operators $\tau$ in 
Theorem~\ref{thm:weyl}. The uniformity in $f$ is needed to take the limit $t\to \infty$ 
in Theorem~\ref{thm:weyl}, because as is standard in the study~\cite{DKV83} of the spectra 
of locally symmetric spaces, $f$ depends on $t$.

The dependence on $\gamma$ is quantified by the Weyl discriminant
\[
D^G(\gamma_s)
= \left|\det(1- \mathrm{Ad}(\gamma_s) | \mathfrak{g}/\mathfrak{g}_{\gamma_s})\right|_\C
=  \prod_{\substack{1\le i<j\le n \\ \rho_i\neq\rho_j}}
 |1-\rho_i\rho_j^{-1}|_{\C}
\]
for $\rho_1, \ldots, \rho_n\in\C$ the eigenvalues of $\gamma_s\in G$ acting on $\R^n$. It is locally bounded and never vanishes, however it becomes arbitrary small if $\gamma_s$ is close to an irregular element (and is discontinuous at these points).

Next recall that $\ka$ is the Lie algebra of $A\subset T_0$, the connected component of the identity. We identify $\ka$ with the space of vectors $(X_1,\ldots,X_n)\in \R^n$ with $X_1+\cdots + X_n =0$.  Let $\ka^+=\{(X_1, \ldots, X_n)\in\ka\mid X_1\ge\ldots\ge X_n\}$ be the positive Weyl chamber of $\ka$. Let $K=\TO(n)$ be the maximal compact subgroup of $G$. We then have a map
\[
X: G\longrightarrow \ka^+
\]
given by the Cartan decomposition, namely, for $g\in G$ the element $X(g)\in \ka^+$ is the unique element such that $g\in Ke^{X(g)} K$.

There are different ways to approach the dependence in $f$. We 
follow~\cite{DKV83} in using the spherical Paley-Wiener theorem, however we then quickly 
differ from~\cite{DKV83}, because we establish cancellations by integrating a different 
variable, see Section~\ref{sec:spherfcts}. Our approach naturally leads to consider orbital 
integrals of functions of the form
\[
g\mapsto f(g) \| X(g) \|^{-\eta},\ g\in G,
\]
where $f\in C_c^\infty(G)$ is fixed and $\eta>0$.  Note that the function is unbounded in a neighborhood of $K_\infty$. In fact $X(g)=0$ iff $g\in K_\infty$.

Our first main result of Part~\ref{partI} is the following:
\begin{theorem}\label{thm:bound:orb:int:real}
There exist constants $\eta>0$ and $B< \infty$ depending only on $n$ such that the following holds.
For any $f\in C_c^\infty(G)$ there is a constant $C(f)>0$ depending only on $n$ and $f$ such 
that for any pair $(M,\gamma)$ consisting of a semi-standard Levi subgroup $M\subseteq G$ and 
an element $\gamma\in M$ with $(M,\gamma)\neq(G,\pm1)$,
\begin{equation*}
|J_M^G(\gamma, f \| X(\cdot)\|^{-\eta})|
\le C(f)D^G(\gamma)^{-B}.
\end{equation*}
\end{theorem}

The theorem generalizes several previous results. Our starting point is Arthur~\cite{Ar88a} 
who showed that for any $\gamma$, the weighted orbital integrals $f\mapsto J_M^G(\gamma,f)$ 
define smooth Radon measures on $G$ supported on the conjugacy class $O_M^G(\gamma)$, 
the conjugacy class in $G$ which is induced from  $O^M(\gamma)=\{ x^{-1}\gamma x,\ x\in M\}$ 
in $M$ (the unweighted case $M=G$ is due to Rao).

Duistermaat--Kolk--Varadarajan~\cite{DKV83} studied in great depth the unweighted case 
$M=G$, with $\gamma$  a fixed semisimple element, and $f$ the zonal spherical
function of spectral parameter $\mu$ multiplied by some characteristic function of compact support. Via the stationary phase method and the study of singularity
of the phase functions they are able to produce an asymptotic for large frequencies 
($\mu\to \infty$). Lapid-M\"uller~\cite{LM09} treated the weighted case $M\neq G$, with $\gamma=1$, 
in a way similar to~\cite{DKV83}.
This was extended by the first-named author~\cite{Matz}*{\S12} to the case of \emph{split} $\gamma$, using parabolic descent.
Our present approach is completely independent, even in the split case, and in fact we can recover the main results of~\cites{LM09,Matz} from Theorem~\ref{thm:bound:orb:int:real} and Proposition~\ref{prop:spher:fct:bound}.

The other direction is if $\gamma$ varies. Our formulation allows a direct
comparison with classical germ expansions of Harish-Chandra~\cite{HC57} and
Arthur~\cite{Ar88a}*{\S13} which correspond to $\eta=0$ (then the test function $f$
is smooth and bounded). The most recent result in this direction is
Arthur~\cite{ArthurGerm} who has generalized the descent formulas and germ
expansions of Harish-Chandra to the weighted case. If $\gamma$ is regular
semisimple and $\eta=0$, then it is shown in~\cite{ArthurGerm}*{\S3} that the bound
 $J_{M}^G(\gamma,f)\le C_B(f)D^G(\gamma)^{-B}$ holds for any fixed $B>0$ and that the 
 constant $C_B(f)$ can be
taken as a continuous semi-norm on $C_c^{\infty}(G)$ that extends to the
Harish-Chandra Schwartz space.

The relative position of $K$ and $O(\gamma)$ as submanifolds of $G$ play a role in finding 
good bounds for the weighted orbital integrals since $f\mapsto J_M^G(\gamma,f)$ is a 
distribution supported in the orbit $O(\gamma)$ and the test function is unbounded in a 
neighborhood of $K$. If $(M,\gamma)=(G, \pm1)$, then $O(\pm 1)=\{\pm 1\}\subset K$, and the test 
function $f\|\cdot \|^{-\eta}$ is not defined on that point.
The theorem says that conversely the condition $(M,\gamma)\neq (G,\pm 1)$ is sufficient to 
obtain cancellations. Our estimate is soft in the sense that the proof doesn't require 
hard analysis estimates at the 
cost of poor exponents (besides Arthur's theory of weighted orbital integrals, our main 
analytic input is a multidimensional van der Corput estimate with derivatives of order $2$).
From a representation-theoretic perspective, the analysis in~\cite{Sarnak87} for $n=2$
relies on the fact that Fourier transforms of local weighted orbital integrals are
explicitly known for the group $\SL(2,\R)$.

\begin{remark}
	It is natural to ask whether the exponent $1/2$ in 
	Theorem~\ref{thm:bound:orb:int:real} and Theorem~\ref{thm:weyl} can be doubled 
	to match the bound of Selberg for $\SL(2)$. We will see in 
	\S\ref{sec:spherfcts} below that the saving by $1/2$ comes from our uniform 
	estimate for zonal spherical functions, and as such it is sharp. There is 
	an additional saving by $1/2$ to be gained in the orbital integral, via a 
	geometric analysis of critical manifolds. The idea is to combine our 
	method in Part~\ref{partI} with the final sections of~\cite{DKV83}.
\end{remark}

To estimate the archimedean orbital integrals of (i) in \S\ref{sub:ideas}, we  eventually reduce with the help of Theorem~\ref{thm:bound:orb:int:real} to obtaining an estimate for zonal spherical functions $\phi_{\lambda}(g)$ that is uniform in both $\lambda\in i\Fma^*$ and $g\in G$. This is
achieved in Proposition~\ref{prop:spher:fct:bound} which is our second main result in Part~\ref{partI}. It shows that the zonal spherical function $\phi_\lambda(g)$ is uniformly small as soon as $g$ is away from the identity at distance greater than the frequency $\|\lambda\|^{-1}$.
Our proof is to apply a multidimensional van der Corput estimate in combination
with~\cite{DKV83}. Independently Blomer--Pohl~\cite{BP:sup-norm-Siegel} have obtained the
same estimate.

To gain further intuition of the role of the test functions $f\|X(\cdot)\|^{-\eta}$ in the Weyl's law it is helpful to draw the analogy with Fourier analysis on $\R$. Essentially the test function is the absolute value of the $\mathrm{sinc}$ function whose Fourier transform is a rectangular pulse (the indicator function in frequency of an interval $[-t,t]$). Similarly $f \|X(\cdot)\|^{-\eta}$ approximates the test functions whose spherical transform capture the automorphic spectrum of Laplace eigenvalue less than $t$ and this is how they appear in the proof of Theorem \ref{thm:weyl}. See Section~\ref{sec:spherfcts} for the exact formulas.

\subsection{Convention}
Throughout this paper the multiplicative constants in $\ll$, $\gg$, $\asymp$, $O()$ should be understood to depend on $n$ and could in principle be made explicit. Although we don't pursue this direction here, it would be interesting to understand the trace formula on $\GL(n)$ in the limit as $n\to \infty$, see e.g.~\cite{Miller:highest}.

\section{Symmetry type of families and low-lying zeros}

We fix a $W$-invariant non-empty bounded subset $\Omega\subset i\ka^*$ with piecewise
$C^2$-boundary. For any $n\ge 2$, we define a family of even Hecke--Maass cusp forms
$\FmF_{even}$, consisting of unramified spherical representations with spectral parameter
in the open cone $\R_{>0} \Omega$. Thus we let for all $t\ge 1$
\[
\FmF_{\mathrm{even}}(t)
:=
\{
\pi \in \Pi_{\mathrm{cusp}}(G(\A)^1),\ \lambda_\pi\in t\Omega,\ \pi^{K_f}\neq 0 \text{ and $\pi_\infty$ spherical}
\}.
\]
If $n$ is even we define similarly
\[
\FmF_{\mathrm{odd}}(t)
:=
\{
\pi \in \Pi_{\mathrm{cusp}}(G(\A)^1),\ \lambda_\pi\in t\Omega,\ \pi^{K_f}\neq 0 \text{ and $\pi_\infty$ with $K_\infty$-type $\chi_-$}
\}.
\]
In the sequel we let $\FmF$ be either $\FmF_{\mathrm{even}}$ or $\FmF_{\mathrm{odd}}$. The Weyl's law as in Theorem~\ref{thm:weyl} and~\cite{Muller07} shows that $|\FmF(t)|\sim \Lambda_\Omega(t)$ as $t\to \infty$.

\subsection{Principal $L$-functions} We attach to every representation $\pi$ the
principal $L$-function $L(s,\pi,\mathrm{std})$. We denote by $C(t)$ the average analytic
conductor for $\pi \in\FmF(t)$. We have $C(t) \asymp t^n$ as $t \to \infty$. The zeros
$\Lambda(\rho,\pi,\mathrm{std})=0$ are inside the critical strip, that is
$0<\mathrm{Re}\, \rho<1$.
\begin{theorem}\label{th:k-level}
	Let $k\ge 1$ and $\Phi_1,\ldots, \Phi_k$ be entire functions whose Fourier
	transforms are smooth and have small enough support. The average $k$-level density of
	low-lying zeros
	\begin{equation}\label{limit}
\frac{1}{|\FmF(t)|}
\sum_{\pi\in \FmF(t)}
\sum_{\substack{\rho_j=\mdemi + i\gamma_j\\ j=1\ldots k}}
\Phi_1\left(\frac{\gamma_1}{2\pi} \log C(t)\right)
\cdots
\Phi_k\left(\frac{\gamma_k}{2\pi}\log C(t)\right)
\end{equation}
  where the second sum is over $k$-tuples of zeros $\Lambda(\rho_j,\pi,\mathrm{std})=0$,
  converges as $t\to \infty$. The limit coincides with the $k$-level density of the
  eigenvalues of the $U(\infty)$ ensemble if $n\ge 3$. If $n=2$, the limit coincides with
  the $k$-level density of the eigenvalues of the $\SO(\mathrm{even})$ ensemble for
  $\FmF=\FmF_{\mathrm{even}}$ and the $k$-level density of the eigenvalues of the
  $\SO(\mathrm{odd})$ ensemble for $\FmF=\FmF_{\mathrm{odd}}$.
\end{theorem}

The sum~\eqref{limit}   encodes deep information about the correlation of low-lying zeros
of $L(s,\pi,\mathrm{std})$. The theorem is a partial confirmation of the Katz-Sarnak
heuristics~\cite{book:KS,SST} for this family. We emphasize that the result is entirely
unconditional (and similarly for Theorem~\ref{th:sym-ext}), for example we do not 
need to
assume the GRH because the functions $\Phi_j$ are entire.

If $n=2$, that is for classical Hecke-Maass forms on $\SL(2)$, the same result recently appeared in the work of Alpoge--Miller~\cite{Alpoge-Miller}, and is also to be compared with~\cite{ILS00} in the holomorphic case.

\begin{example}
	 If $k=1$, the limit of~\eqref{limit} is $\int_{-\infty}^{+\infty} \Phi_1(x) dx$. If $k=2$, the limit is
\[
\int_{\R^2} \Phi_1(x_1) \Phi_2(x_2)
\left[
1- \frac{\sin \pi(x_1-x_2)}{\pi(x_1-x_2)}
\right]^2
dx_1dx_2.
\]
In general the $k$-level density of the $\U(\infty)$ ensemble is given by the determinant of the Dyson kernel~\cite{book:KS}.
\end{example}

\subsection{Functorial lifts}
Next we want to consider more general $L$-functions. Since every $\pi\in \FmF$ has trivial central character, the $L$-group is $\SL(n,\C)$.  The symmetric square $L$-function $L(s,\pi,\mathrm{sym}^2)$ comes from the representation of $\SL(n,\C)$ on $\mathrm{Sym}^2 \C^n$. It has degree $n(n+1)/2$.  The exterior square $L$-function $L(s,\pi,\wedge^2)$ comes from the representation of $\SL(n,\C)$ on $\wedge^2 \C^n$. It has degree $n(n-1)/2$.  The adjoint $L$-function $L(s,\pi,\mathrm{Ad})$ comes from the adjoint representation	of $\SL(n,\C)$ on $\mathfrak{sl}(n,\C)$. It has degree $n^2-1$.  It is useful to note that $L(s,\pi\times \pi)= L(s,\pi,\mathrm{sym}^2)  L(s,\pi,\wedge^2)$ and $L(s,\pi\times \widetilde \pi)=\zeta(s) L(s,\pi,\mathrm{Ad})$. All these $L$-series converge absolutely for $\mathrm{Re}\, s>1$.

\begin{example}
For $n=2$, $L(s,\pi,\wedge^2)=\zeta(s)$ because $\pi$ has trivial central character, and also $L(s,\pi,\mathrm{Ad})=L(s,\pi,\mathrm{sym}^2)$. For $n=3$, $L(s,\pi,\wedge^2)=L(s,\widetilde\pi,\mathrm{std})$. There are no other relations between these $L$-functions for general representations $\pi$.
\end{example}

The analytic continuation and functional equation of the exterior square are known from
either the Langlands-Shahidi method, or the Jacquet-Shalika integral representation,
see~\cite{Cogdell-Matringe} and the references there. For the symmetric square it is
known from either the Langlands-Shahidi method or the Bump-Ginzburg~\cite{Bump-Ginzburg}
integral representation. For our purpose the choice of the construction is irrelevant
since it doesn't change the location of the zeros inside the critical strip.\footnote{In
both cases the local $L$ and $\gamma$-factors are conjectured to agree at ramified places
with those obtained by local Langlands correspondence but we shall not need this.}  The
meromorphic continuation and functional equation of the adjoint $L$-function follows from
Rankin-Selberg theory for $\Lambda(s,\pi\times \widetilde \pi)$ and by quotienting by
$\zeta(s)$. Thus, for each of the above $L$-functions, we can speak of its zeros $\rho_j
= \tfrac 12 + i\gamma_j$ and form the associated density statistics.

In each case we denote by $C_{\mathrm{sym}^2}(t)$,  $C_{\wedge^2}(t)$ and  $C_{\mathrm{Ad}}(t)$ the average analytic conductor for $\pi \in \FmF(t)$. We have $C_{\mathrm{sym}^2}(t) \asymp t^{\frac{n(n+1)}{2}}$, $C_{\wedge^2}(t) \asymp t^{\frac{n(n-1)}{2}}$ and $C_{\mathrm{Ad}}(t) \asymp t^{n^2-1}$ respectively as $t\to \infty$.
\begin{theorem}\label{th:sym-ext}   The average $k$-level density of the low-lying zeros
$\rho_j$ of the symmetric square $L$-functions $\Lambda(\rho_j,\pi,\mathrm{sym}^2)=0$
(resp. exterior square $L$-functions $\Lambda(\rho_j,\pi,\wedge^2)$ if $n\ge 3$, resp.
adjoint $L$-function $\Lambda(\rho_j,\pi,\mathrm{Ad})$) converges as $t\to \infty$ if the
Fourier transforms of $\Phi_1,\ldots, \Phi_k$ are smooth and have small enough support.
 The limit coincides with the $k$-level density of the eigenvalues of the $U(\infty)$ ensemble for the symmetric square and exterior square if $n\ge 3$, and of the $\Sp(\infty)$ ensemble for the adjoint.
\end{theorem}

\begin{example}
For $k=1$ and for the zeros of the adjoint $L$-function, the limit of~\eqref{limit} is
\[
\int_{-\infty}^\infty \Phi_1(x) \left[
1- \frac{\sin 2\pi x}{2\pi x}
\right] dx.
\]
In general the $k$-level density is given by the determinant of the Dyson kernel for $\Sp(\infty)$~\cite{book:KS}.
\end{example}

\subsection{Essential cuspidality}\label{sub:cuspidal}
In establishing the above density statistics of the zeros, one needs to control the
poles.
All of the representations $\mathrm{std}$, $\mathrm{sym}^2$, $\wedge^2$ and $\mathrm{Ad}$ of
$\SL(n,\C)$ are irreducible. Following~\cite{SST} we say that each of the associated families are essentially cuspidal.
Essentially cuspidal families of $L$-functions are expected to have negligible number of poles on average which we shall now verify for each family in turn.

The completed $L$-functions $\Lambda(s,\pi,\mathrm{std})$ are entire so there is nothing to verify for the family of standard $L$-functions.

The completed Rankin-Selberg $L$-functions $\Lambda(s,\pi \times \widetilde \pi)$ have a simple pole at $s=1$ and therefore $\Lambda(s,\pi,\mathrm{Ad})$ is holomorphic at $s=1$. The other possible\footnote{It is conjectured that $\Lambda(s,\pi,\mathrm{Ad})$ is entire, i.e., $\zeta(s)$ divides $L(s,\pi\times \widetilde \pi)$. This is known for $n=2,3,4,5$ by the works of Shimura, Ginzburg, Bump--Ginzburg and Ginzburg--Hundley respectively.} poles of $\Lambda(s,\pi,\mathrm{Ad})$ are the zeros of $\zeta(s)$ inside the critical strip. Since these potential poles are fixed, thus independent of $\pi\in \FmF(t)$, they are negligible in the limit $t\to \infty$ of the average $k$-level density of the low-lying zeros. Indeed the explicit formula for $\Lambda(s,\pi \times \widetilde \pi)$  will capture the zeros of $\Lambda(s,\pi,\mathrm{Ad})$ while the extra zeros of $\zeta(s)$ are negligible in the limit since $\Phi_1, \ldots, \Phi_k$ are of rapid decay.

For the symmetric square and exterior square $L$-functions we shall need the following
result:
\[
\lim_{t\to \infty}
\frac{
 |\{\pi\in \FmF(t),\ \Lambda(s,\pi,\mathrm{sym}^2) \text{ is entire}\}
|
}
{ |\FmF(t)|} = 1,
\]
 and similarly for $\Lambda(s,\pi,\wedge^2)$ for $n\ge 3$.
Since $L(s,\pi \times \pi)=L(s,\pi,\mathrm{sym}^2)L(s,\pi,\wedge^2)$, the representation
$\pi$ is self-dual if and only if the symmetric square or exterior square $L$-function
has a pole.
Thus it is equivalent to establish that for $n\ge 3$, the number of self-dual automorphic
representations in $\FmF(t)$ is
negligible as $t\to \infty$. The precise statement is as follows:
\begin{proposition}[(Kala~\cite{Kala:phd})]\label{p:Kala}
For every $n\ge 3$, the following upper bounds hold
\[
|\{\pi\in \FmF(t),\ \Lambda(s,\pi,\wedge^2) \text{ has a pole}\}| \ll
t^{\frac{n^2+2n}{4}},
\]
\[
|\{\pi\in \FmF(t),\ \Lambda(s,\pi,\mathrm{sym}^2) \text{ has a pole}\}| \ll
\begin{cases}
t^{\frac{n^2-1}{4}} & \text{if $n$ is odd,} \\
t^{\frac{n^2}{4}} & \text{if $n$ is even.}
\end{cases}
\]
\end{proposition}

\begin{proof}
The assumptions in~\cite{Kala:phd} are somewhat different, so
we repeat the argument for the sake of completeness.
If $\pi\in \FmF(t)$ is such that $L(s,\pi,\wedge^2)$ has a pole, then
Arthur's classification says that $n$ is even, and that $\pi$ descends to a cuspidal
automorphic
representation $\sigma$ of the split
orthogonal group $SO_{n+1}$,
see~\cite{Cogdell-Kim-PS-Shahidi:functoriality}*{Thm.7.1}
and~\cite{Ginzburg-Rallis-Soudry:descent}. If $\pi
\in \FmF(t)$ is such that $L(s,\pi,\mathrm{sym}^2)$
has a
pole and $n$ is odd, then Arthur's classification says that $\pi$ descends to a cuspidal
automorphic representation $\sigma$ of the split
symplectic group $\Sp_{n-1}$,
see~\cite{Cogdell-Kim-PS-Shahidi:functoriality}*{Thm.7.2}
and~\cite{Ginzburg-Rallis-Soudry:descent}.
If $\pi\in \FmF(t)$ is such that $L(s,\pi,\mathrm{sym}^2)$ has a pole and $n$ is even,
then since the central character of $\pi$ is trivial, because $\pi$ is spherical and
unramified, Arthur's classification says that $\pi$ descends to a cuspidal automorphic
representation
$\sigma$ of the split orthogonal group $SO_n$,
see~\cite{Cogdell-Kim-PS-Shahidi:functoriality}*{Thm.7.2}
and~\cite{Ginzburg-Rallis-Soudry:descent}.
 (If the central character of $\pi$ were non-trivial, then the descent would be to a
 quasi-split orthogonal group.)

The condition $\pi^{K_f}\neq 0$, equivalently that $\pi_p$ is unramified for every prime
$p$ implies that $\sigma_p$ is also unramified. This is established
in~\cite{Cogdell-Kim-PS-Shahidi:functoriality}*{Thm.9.2} via a case-by-case
analysis.

There is a linear relation between the infinitesimal character
$\lambda_\pi$ of $\pi_\infty$, and the infinitesimal character $\lambda_\sigma$ of
$\sigma_\infty$, see~\cite{Kala:phd}*{Thm.3.3.2}.
In particular, the condition $\lambda_\pi\in t\Omega$ implies $\lambda_\sigma \in
t \Omega'$ for some open bounded $\Omega'$.

The upper bound of Donnelly~\cite{Donn82} towards Weyl's law gives an estimate for the
number of automorphic representations $\sigma$ that satisfy the above conditions (that
is, cuspidal,
unramified, and spherical with $\lambda_\sigma \in t
\Omega'$). Namely, the number of such representations is at most
$\ll t^{d'}$,
where $d'$ is the dimension of the symmetric space of the corresponding (split) classical
group,
that is $d'=\frac{n^2+2n}{4}$, resp. $d'=\frac{n^2-1}{4}$, and $d'=\frac{n^2}{4}$ in the
respective three cases. This concludes the proof of the proposition.
\end{proof}

\subsection{Homogeneity type}\label{sub:homogeneity}
As $p\to \infty$ the Plancherel measure $\mu_p$ converges to the Sato-Tate measure on
${S^1}^n/\mathfrak{S}_n$ attached to the Haar measure on $SU(n)$
by~\cite{ST11cf}*{Prop.5.3}.
Theorem~\ref{thm:weyl} is the Sato-Tate equidistribution for the family $\FmF$.

Following the terminology
in~\cite{SST} we can identify the respective homogeneity types by computing the
Frobenius--Schur indicators.
The representations $\mathrm{std}$, $\mathrm{sym}^2$ and $\wedge^2$ are non self-dual,
with the exception of $\mathrm{std}$ for $n=2$ which is self-dual symplectic. (This case
is already handled in~\cite{Alpoge-Miller}.)
The representation $\mathrm{Ad}$ is self-dual orthogonal because it preserves the Killing form on $\mathfrak{sl}(n,\C)$ which is bilinear symmetric and non-degenerate.

\subsection{Proof of Theorems~\ref{th:k-level} and \ref{th:sym-ext}}
This is similar to~\cite{ST11cf}*{\S12}, and combines the following:
\begin{itemize}
\item the Sato-Tate equidistribution Theorem~\ref{thm:weyl} for the families
$\FmF_{\mathrm{even}}$ and $\FmF_{\mathrm{odd}}$, where the exponent $A$ determines the
size of the support of the Fourier transform of the test functions
$\Phi_1,\ldots,\Phi_k$, and the homogeneity type is determined in
\S\ref{sub:homogeneity};
\item both $\FmF_{\mathrm{even}}$ and $\FmF_{\mathrm{odd}}$ are essentially cuspidal as
explained in \S\ref{sub:cuspidal};
\item both $\FmF_{\mathrm{even}}$ and $\FmF_{\mathrm{odd}}$ have rank zero in the sense
of~\cite{SST} because  $\int_{  {S^1}^n/\mathfrak{S}_n} \phi  \mu_p=O(\frac{1}{p})$ for
each of the respective polynomial functions $\phi(x)=\Mtr(x)$, $\phi(x)=
\Mtr(\mathrm{sym}^2(x))$, $\phi(x)=\Mtr(\wedge^2(x))$ and $\phi(x)=\Mtr(\mathrm{Ad}(x))$.
This is also established in complete generality in~\cite{ST11cf}*{Lem.2.9} using
combinatorial results from~\cite{Kato82}.
\end{itemize}
We omit the details since they are rather standard and one of the purposes of~\cite{SST}
was to organize the properties of families in such a way that a formal verification
becomes straightforward.

\subsection{The average root number}
Let $\psi$ be the standard additive character on $\Q\backslash \A_\Q$. Since both $\pi$
and $\psi$ are unramified at all finite places we have
$\epsilon(\frac12,\pi)=\epsilon(\frac12,\pi_\infty,\psi_\infty)$.

If $\pi_\infty$ is spherical then $\epsilon(\frac12,\pi_\infty,\psi_\infty)=1$ while if $\pi_\infty$ has $K_\infty$-type $\chi_-$ then $\epsilon(\frac12,\pi_\infty,\psi_\infty)=-1$. Thus even (resp. odd) unramified Maass cusp forms have root number equal to $1$ (resp. $-1$), which in the classical language~\cite{Go}*{\S9} is related to the $W$-eigenvalue as in \S\ref{s:intro:maass} and equivalently to the relation $A_f(1,\ldots,1,-1)=\epsilon(\frac12,f)$.  If $n$ is odd then all Maass forms are even, while if $n$ is even we have seen that $|\FmF_{\mathrm{even}}(t)| \sim |\FmF_{\mathrm{odd}}(t)|\sim \Lambda_\Omega(t)$ as $t\to \infty$, and therefore the root number is equidistributed between $\pm 1$.

\section{Average bounds towards Ramanujan}\label{s:pf:ram}

\subsection{Proof of Corollary~\ref{cor:ramanujan} and 
Corollary~\ref{cor:largep}}\label{sub:pf-ram}
	We shall apply the Sato-Tate equidistribution in the version of
	Theorem~\ref{th:heckemaass}.
	For $\alpha \in \C^n$ write $|\alpha|_\infty:=  \max\limits_{1\le j\le n}
	|\alpha^{(j)}|$.
Choose symmetric polynomials $\phi_0, \ldots,\phi_s$ as in Lemma~\ref{l:pol} below so 
that
	\[
	|\phi_0(\alpha)|^{2} + \cdots + |\phi_s(\alpha)|^{2} \ge |\alpha|^{2}_\infty	\]
	for all $\alpha \in  {\C}^n$. Form the conjugate Laurent polynomials
\[
\phi_j^\vee(x_1,\ldots , x_n) := \overline{
\phi_j({\overline{x_1}}^{-1},\ldots , {\overline{x_n}}^{-1})},
\quad
0\le j \le n,
	\]
and let $\tilde\phi := (\phi_0 \phi_0^{\vee}) + \cdots + (\phi_s \phi_s^\vee)$. In fact, 
Lemma~\ref{l:pol} shows that we may 
arrange that $\deg(\widetilde \phi)=2$, where $\deg(e_i)=1$ for $1\le i\le n-1$. 
For large $k\in\N$, to be chosen later, put $\phi(x_1,\ldots,x_n):=
\tilde\phi(x_1^k,\ldots,x_n^k)$ so that the degree of $\phi$ equals $k\deg(\tilde\phi)=2k$,
and the
coefficients of $\phi$ are the same as the coefficients of $\tilde \phi$, hence bounded
by some constant $a>0$ depending only on $n$.
Note that
$\int_{{S^1}^n/\mathfrak{S_n}} \phi \mu_p$ is bounded by
$\sup_{|\alpha|_{\infty}=1} |\tilde \phi(\alpha)| = \sup_{|\alpha|_{\infty}=1}
|\phi(\alpha)|=:b>0$,
which is also constant indepent of $k$.

For any unramified irreducible unitary representation $\pi$ of $\PGL_n(\Q_p)$ with Satake
parameter $\alpha \in (\C^{\times})^n/\mathfrak S_n$ we have
\[
\phi(\alpha) =
\Mtr \pi(\tau)
= |\phi_0(\alpha^k)|^{2} + \cdots + |\phi_s(\alpha^k)|^{2}\ge |\alpha|^{2k}_\infty,
\]
where $\alpha^k=((\alpha^{(1)})^k,\ldots, (\alpha^{(n)})^k)$.
We have that $\alpha$ is unitary in the sense that
\[
\{\alpha^{(1)},\ldots,\alpha^{(n)}\}
= \{
\frac{1}{\overline{\alpha^{(1)}}},\ldots,\frac{1}{\overline{\alpha^{(n)}}}
\}
.
\]
This implies $\phi_j^\vee(\alpha^k)=\overline{\phi_j(\alpha^k)}$.

Denote by $N$ the left-hand side of Corollary~\ref{cor:ramanujan}, that is, $N$ is the 
number of Hecke--Maass cusp forms $f$ with $||\lambda_f||\le t$ and $|\alpha_f(p)|_\infty = 
\max\limits_{1\le j \le n}|\alpha_f^{(j)}(p)| > p^{\theta}$.
By the above inequality for $\phi(\alpha)$, we have
\[
N p^{2k\theta}
\le
\sum_f
\phi(\alpha_f(p)),
\]
where now $f$ runs through all Hecke--Maass cusp forms with $||\lambda_f||\le t$.
We apply Theorem~\ref{th:heckemaass}, together with the above properties of the function
$\phi$ to conclude that this is bounded by
\[
\le b \Lambda_\Omega(t)
+ a c_1 p^{2Ak} t^{d-\frac12},
\]
where we recall that $\Lambda_{\Omega}(t)\sim b_1 t^d$ for some constant $b_1$.
We choose
$k:= 
\left\lfloor
\frac{\log t}{4A\log p}
\right\rfloor
$ to be the largest integer $\le \frac{\log t}{4A\log p}$. We obtain
\[
 N\le  p^{-2 \left\lfloor
\frac{\log t}{4A\log p}
\right\rfloor \theta
} 
(bb_1 t^d +a c_1 t^{\frac12} t^{d-\frac12})
\le 
 t^{-x \left\lfloor
\frac{c}{x}
\right\rfloor \theta} 
(bb_1 + a c_1)t^{d}
\]
with $c:=1/(2A)$, a constant depending on $n$, and $x=\frac{2\log p}{\log t}$.
 This concludes the proof of Corollary~\ref{cor:ramanujan}.

If instead we choose $k:= 
\left\lfloor
\frac{\log t}{4A\log p}
\right\rfloor +1
$, then we obtain
\[
 N\le  p^{-2\theta (\left\lfloor
\frac{c}{x}
\right\rfloor
+1)
}  (bb_1 + a c_1) p^{2 A (\left\lfloor
\frac{c}{x}
\right\rfloor
+1) }
t^{d-\frac12}.
\]
This concludes the proof of Corollary~\ref{cor:largep}.

\begin{lemma}\label{l:pol}
	 There exist symmetric polynomials $\phi_0, \ldots, \phi_s \in \C[x_1, \ldots, 
	 x_n]^{\mathfrak{S}_n}$ such that
	\[
\max\bigl(
	|\phi_0(\alpha) |, \ldots,  |\phi_s(\alpha)|\bigr) \ge |\alpha|_\infty	\]
	for all $\alpha \in  {\C}^n$.
	In fact, one can take $s=n$, and multiples the elementary symmetric 
	polynomials $e_0,e_1,\ldots, e_n$ for the polynomials $\phi_0,\ldots, \phi_{n}$.
\end{lemma}

\begin{proof}
 For an integer $m\in \left\{ 0, \ldots,n \right\}$, and $x=(x_1,\ldots,x_n)\in\C^n$ let
 \[
  \phi_m(x):=\frac{2^m}{m!\ (n-m)!} \sum_{\sigma\in\mathfrak{S}_n} x_{\sigma(1)}\cdot\ldots\cdot 
  x_{\sigma(m)}
 \]
be a multiple of the elementary symmetric polynomial of degree $m$ in $n$ 
variables. In particular $\phi_0=1$.
Let $x_{\max}\in \{x_1,\ldots,x_n\}$ be such that $|x_{\max}|=|x|_{\infty}=\max_{1\le j\le n}|x_j|$. Let $x^-\in\C^{n-1}$ be the vector obtained from $x$ by omitting the coordinate $x_{\max}$.
Then for every $1\le m\le n$
\[
 \phi_m(x)
 = 2 x_{\max}\phi_{m-1}(x^-) + \phi_m(x^-).
\]
Hence we have either
\begin{equation}\label{eq:low:bound:phi1}
 |\phi_m(x)|
 \ge |x_{\max}| |\phi_{m-1}(x^-)|,
\end{equation}
or
\begin{equation}\label{eq:low:bound:phi2}
 |\phi_m(x^-)|\ge |x_{\max}| |\phi_{m-1}(x^-)|.
\end{equation}

The inequality~\eqref{eq:low:bound:phi1} holds for $m=n$. Hence we can let $m_0$ be the smallest $m\in \left\{ 1, \ldots, n \right\}$ such that
\eqref{eq:low:bound:phi1} holds. For every $1\le m\le m_0-1$ the inequality \eqref{eq:low:bound:phi2} holds so that
\[
 |\phi_{m_0-1}(x^-)|
 \ge |x_{\max}| |\phi_{m_0-2}(x^-)|
 \ge \ldots
 \ge |x_{\max}|^{m_0-1}.
\]
Therefore,
\[
 |\phi_{m_0}(x)|
 \ge |x_{\max}| |\phi_{m_0-1}(x^-)|
 \ge |x_{\max}|^{m_0}.
\]
Hence the lemma follows with $s:=n+1$ and $\phi_j:=\phi_{j-1}$, for $j=1,\ldots,n+1$.
\end{proof}

We see that the constants in Corollary~\ref{cor:ramanujan} depend on the choice of polynomials in the above lemma which in turn depend on $n$ only.

If $\pi$ is an irreducible unitary representation of $\PGL_n(\Q_p)$ then, the associated 
Satake parameter $\alpha\in \C^n$ is unitary and $\alpha^{(1)} \cdots \alpha^{(n)}=1$. For the 
application to Corollary~\ref{cor:ramanujan} it would have been sufficient to establish 
Lemma~\ref{l:pol} with these two extra conditions on $\alpha \in \C^n$. This can be exploited 
for $n=2$ as shown below. For general $n$  we have decided to establish the lemma in this 
stronger form since the proof is essentially the same.

\begin{example}
	For $n=3$ consider the Schur polynomial $s_{(l,0,0)}$. 
	Blomer-Buttcane-Raulf~\cite{Blomer-Buttcane-Raulf} estimate the average of 
	$|A_f(p^l,1)|^{2k} = |s_{(l,0,0)}(\alpha)|^{2k}$ in the Kuznetsov trace formula.
The property used there is that $|s_{(l,0,0)}(\alpha)|$ is approximately greater 
than $|\alpha|_{\infty}^{l-1}(|\alpha|_\infty - 1)$ for all unitary $\alpha\in (\C^\times)^3/\mathfrak 
S_3$ such that 
$\alpha^{(1)}\alpha^{(2)}\alpha^{(3)}=1$. Since $|\alpha|_\infty$ may be arbitrary close to 
$1$, this forces~\cite{Blomer-Buttcane-Raulf} to choose $l$ large enough depending 
on $\theta$ which makes their final result less explicit than our
Corollary~\ref{cor:ramanujan}. The difference with our 
approach via Lemma~\ref{l:pol} is that our test polynomials $\phi$ constructed 
above have stronger approximation properties.
One could combine our construction with the 
strong estimates for the 
Kuznetsov trace formula for $\SL_3(\Z)$ 
in~\cite{Blomer-Buttcane-Raulf,Buttcane-Zhou:weight} to improve 
on the average bound towards Ramanujan in~\cite{Blomer-Buttcane-Raulf}*{Thm.2}. 
We haven't pursued the details here.
\end{example}

\subsection{Proof of Theorem~\ref{t:SL2-Ramanujan}}\label{sub:proof-SL2}
Assume for the remainder of this section that $n=2$. 
All unitary $\alpha\in (\C^\times)^2/\mathfrak S_2$ with 
$\alpha^{(1)}\alpha^{(2)}=1$ satisfy either that $|\alpha^{(1)}| = |\alpha^{(2)}|=1$ 
(tempered parameters), or
$\alpha^{(1)},\alpha^{(2)}\in \R^\times$ (non-tempered parameters).

Let $e_1(x_1,x_2)=x_1+x_2$. 
The average bound towards Ramanujan in~\cite{Sarnak87}, resp. in 
\cite{Blomer-Buttcane-Raulf}, is established by averaging $A_f(p)^{2k} = 
e_1(\alpha_f(p))^{2k}$ in Selberg trace formula for $\SL_2(\Z)$, resp. in Kuznetsov 
trace 
formula for $\SL_2(\Z)$.

Instead, following our construction above, we shall average 
\[
\left(\alpha^{(1)}_f(p)^{k} + \alpha^{(2)}_f(p)^k\right)^2
=
(A_f(p^k) - A_f(p^{k-2}))^2
=
 A_f(p^{k-2})^2 + A_f(p^k)^2 - 2 A_f(p^{k-2}) A_f(p^k),\quad k\ge 2
\]
in Kuznetsov trace formula.
If $k=1$, then we shall average $(\alpha^{(1)}_f(p) + \alpha^{(2)}_f(p))^2 = A_f(p)^2$.
If $k=0$, then we average $(1+1)^2=4$, which of course yields a trivial bound, 
though it is ok to include it (this corresponds to that~\eqref{SL2-8x} is trivial 
in the region $p^4>t$).

The approximation property we shall use is that $\left|{\alpha^{(1)}}^k + 
{\alpha^{(2)}}^k\right|
> |\alpha|^k_\infty$ for every unitary $\alpha \in (\C^\times)^2/\mathfrak S_2$ with 
$\alpha^{(1)}\alpha^{(2)}=1$ and every $k\ge 0$.
The rest of the proof is identical to~\cite{Blomer-Buttcane-Raulf}*{\S2}.
Note that $u_j$ in \cite{Blomer-Buttcane-Raulf} is our $f$, and $\lambda_j=\frac14 + 
t_j^2$ is our $||\lambda_f||^2$.
We obtain:
\[
Np^{2k\theta}
\le
 \sum_{||\lambda_f||\le t} \left(\alpha^{(1)}_f(p)^{k} + \alpha^{(2)}_f(p)^k\right)^2 \ll_\epsilon
(t^2+p^{\frac k2})^{1+\epsilon},
\]
where $N$ is the number of Hecke-Maass cusp forms $f$ with $||\lambda_f||\le t$ and 
$\max\limits_{j=1,2} |\alpha^{(j)}_f(p)|>p^\theta$.
Choosing $k=\left\lfloor
\frac{4 \log t}{\log p}
\right\rfloor$, we have $t^2 \ge p^{\frac k2}$, and  we deduce the estimate~\eqref{SL2-8x}.
Choosing $k=\left\lfloor
\frac{4\log t}{\log p}
\right\rfloor + 1$, we have $t^2 \le p^{\frac k2}$, and we deduce the 
estimate~\eqref{SL2-plusone}. \qed

\subsection{Proof of Corollary~\ref{cor:ranges}}
Fix $\theta \in [0,\frac14]$ and let $x=\frac{2 \log t}{\log p}$. 
Taking the minimum of~\eqref{SL2-8x} and~\eqref{SL2-plusone}, we need to study the 
function 
\[
r(x) =  
 \min\left(2 - x \left\lfloor
\frac{8}{x}
\right\rfloor \theta , x(  \left\lfloor
\frac{8}{x}
\right\rfloor + 1 )(\frac14 - \theta)\right).
\]
It is not difficult to verify that $r$ is function. We have
$r(\frac{8}{m})= 2-8\theta$ for every integer $m\ge 1$. 
Restricted to the subinterval $x\in (\frac{8}{m+1}, \frac{8}{m}]$, we have $\left\lfloor
\frac{8}{x}
\right\rfloor = m$, hence $r(x)$ is the minimum of two affine functions of $x$, which
can be seen to increase from its minimal value $r(\frac{8}{m+1})=2 - 8\theta$ to its 
maximal value
\[
r\left(\frac{8}{1+m-4\theta}\right) = 2 - \frac{8m\theta}{1+m - 4\theta} =  
\frac{2(1+m)(1-4\theta)}{1+m - 
4\theta},
\] 
and then decrease from there to its same minimum value $r(\frac{8}{m})=2 - 8\theta$.
Since the sequence $\frac{2(1+m)(1-4\theta)}{1+m - 
4\theta}$ is decreasing as the integer $m$ increases, we deduce
\[
 \max_{0\le x \le \frac{8}{m}} r(x) =
 \max_{\frac{8}{m+1}\le x \le \frac{8}{m}} r(x)
=
 \frac{2(1+m)(1-4\theta)}{1+m - 
4\theta},
\]
for every integer $m\ge 1$, which concludes the proof. \qed

\part{Local theory: Real orbital integrals}\label{partI}

General (weighted) orbital integrals were defined and studied by Arthur in a series of 
papers on establishing the trace formula for general reductive groups over number fields. 
In this Part~\ref{partI} we establish all the necessary estimates at the archimedean 
place for the group $\GL(n)$ over $\R$.

The properties of orbital integrals are rather mysterious even in the unweighted case 
and the weights introduce more difficulties. The literature contains some versions of germ 
expansions, and descent formulas but they are not directly applicable. It seems that a 
direct approach only exists in the unweighted and regular semisimple case which we 
present in Section~\ref{sec:HC}.
It would be interesting to refine the existing framework even further along the lines of~\cite{ArthurGerm}.
 For example the local trace formula at the archimedean place should come forth since it is implicit in what we are doing.

In the end we develop the material from the outset and shall rely in an essential way on Harish-Chandra's and Arthur's theorems and on  analytic techniques such as the multidimensional van der Corput inequality.
As mentioned in the introduction, we hope to return in a subsequent work to establishing 
sharp estimates
 where the idea will be to replace the van der Corput inequality by a combination of germ expansions and semiclassical estimates for Morse-Bott functions that vary in families.

The most important step of our approach is contained in Section~\ref{sec:spherfcts} with
a uniform estimate on zonal spherical functions  which seems to have been missed despite
their rather comprehensive study since the 60's.\footnote{This estimate has also been
established by Blomer--Pohl~\cite{BP:sup-norm-Siegel}, independently of our work, and for
a different purpose.}

At a coarse level, Part~\ref{partI} contains the main ingredients to establish the
remainder in Weyl's law with respect to the spectral parameter, and
Part~\ref{part2} is concerned with the polynomial control of the geometric side.
We expect that
Part~\ref{partI} and therefore Theorem~\ref{thm:bound:orb:int:real} extends to general
groups.
Our argument is fundamentally different from~\cite{Matz} who treats the case of 
$\GL_n(\C)$.
In Part~\ref{part2}, there are obstacles to work with general groups, both of local
and global nature, such as the global Arthur coefficients, and bounds for the 
residual
spectrum.

\section{Preliminaries}\label{sec:notation}

\subsection{Notation}
We work with the group $G=A_G\SB \GL_n(\R)$, where $A_G\simeq \R_{>0}$ is the group of scalar diagonal matrices with positive real entries. We can identify
\[
G\simeq \GL_n(\R)^1
=\{g\in\GL_n(\R)\mid |\det g|=1\}
\]
which is convenient to write down explicit matrices and examples. Hence $G$ can 
also be identified with the $\R$-points of an algebraic $\R$-group.
Let
\[
G_{\C}=\GL_n(\C)^1=\{g\in\GL_n(\C)\mid |\det g|=1\}
\]
with $|\cdot|:\C\longrightarrow\R_{\ge 0}$ the usual absolute value. Let $K=\tO(n)$ be the 
usual maximal compact subgroup of $G$, and $K_{\C}=\U(n)\subseteq G_{\C}$ the usual 
maximal compact subgroup of $G_{\C}$. (Note that $K_{\C}$ is \emph{not} the 
complexification of $K$.) Let $K^{\circ}=\SO(n)\subseteq K$ be the identity component 
of $K$.

Let $T_0$ be the diagonal torus of $G$ and $P_0$ be the Borel subgroup of upper-triangular matrices so that $P_0=T_0U_0$ for $U_0$ the unipotent radical of $P_0$. We call a parabolic subgroup \emph{standard} if it contains $P_0$, and \emph{semi-standard} if it contains $T_0$. Similarly, a Levi subgroup will be called semi-standard (resp.~standard) if it is the Levi component of some semi-standard (resp.~standard) parabolic subgroup. If $M\subseteq G$ is a semi-standard Levi subgroup, we denote by $\CmL(M)$ the set of all Levi subgroups in $G$ containing $M$, by $\CmF(M)$ the set of all parabolic subgroups containing $M$, and by $\CmP(M)$ the set of all parabolic subgroups with Levi component $M$. All these sets are finite. If $P\in \CmF(T_0)$, we denote by $U_P$ the unipotent radical of $P$, and by $M_P$ the Levi subgroup of $P$ containing $T_0$.

Let $W$ denote the Weyl group of the pair $(T_0, G)$. If $H\subseteq G$ is a Levi or parabolic subgroup, and $T\subseteq H$ a split torus, we denote by $\Phi(T,H)$ the set of roots of $T$ on $H$. We write $\Phi=\Phi(T_0, G)$, and  $\Phi^+=\Phi(T_0,U_0)$ for the usual set of positive roots of $T_0$ on $G$. Similarly, if $M\in\CmL(T_0)$, we put $\Phi^M=\Phi(T_0,M)$ and $\Phi^{M,+}=\Phi(T_0, U_0\cap M)$. Let $\Delta_0\subseteq \Phi^+$ be the set of simple roots in $\Phi^+$.

Let $\Fma:=\Fma_0^G=\TLie A_0^G\subseteq\Fmg:=\TLie G $, where $A_0^G\subseteq G$ denotes the subgroup of all diagonal matrices $\Mdiag(a_1, \ldots, a_n)$ with $a_1, \ldots, a_n\in\R_{>0}$ and $a_1\cdot\ldots\cdot a_n=1$. We identify $\Fma$ with the subspace of $\R^n$ consisting of all vectors $X=(X_1, \ldots, X_n)$ with $X_1+\ldots+ X_n=0$.
Let $\|\cdot\|:\R^n\longrightarrow\R$ denote the usual Euclidean norm. This then also defines a $W$-invariant norm on $\Fma$.

If $P=MU\in \CmF(T_0)$, let $A_M:=A_M^G\subseteq M$ be the identity component of the center of $M$, and $\ka_M=\ka_P=\TLie A_M$.
$M$ then equals the direct product $A_M\times M^1$ where $M^1=\bigcap_{\chi}\ker|\chi|$ with $\chi$ running over all unitary characters of $M$. Hence we get a map
\[
H_P: G\longrightarrow \ka_P
\]
characterized by the property that $g=e^{H_P(g)} muk$ with $m\in M^1$, $u\in U$, and $k\in K$. If $P=P_0$, we write $H_0=H_{P_0}$. 

\subsection{Distance functions on $G/K$}
For any $g\in G$, we define
\[
\kL(g):=\log \left( \frac{\Mtr(g^Tg)}{n}\right),
\]
where $g^T$ denotes the transpose of $g$.
 We have the Cartan decomposition $G=KAK$, and for any $g\in G$ we denote by $X(g)$ an element of $\Fma$ such that $g\in Ke^{X(g)}K$. Then $X(g)$ is unique up to Weyl group conjugation, and we can identify $X(g)$ with an element in the quotient $\Fma/W$. To make choices definite we can take $X(g)$ to be an element in the closure $\overline{\ka^+}$ of the positive Weyl chamber $\ka^+=\{X\in\ka\mid\forall\alpha\in\Delta_0:~\alpha(X)>0\}$. Then $X(g)$ is unique.

\begin{remark}
 \begin{enumerate}[(i)]
	 \item The mappings $g\mapsto X(g)\in\overline{\ka^+}$ and $g\mapsto \kL(g)\in \R_+$ are 
	 specific to our choice of maximal compact subgroup $K$ and Cartan 
	 involution $g\mapsto (g^T)^{-1}$.
\item $\kL$ is a bi-$K$-invariant function.
\item For any $g\in G$ we have $\kL(g)\ge0$ or, equivalently, $\Mtr(g^T g)/n\ge1$. This is because $\Mtr (g^T g)/n=(\xi_1+\ldots+\xi_n)/n\ge \det(g^T g)^{1/n}=1$, where $\xi_1, \ldots, \xi_n$ are the eigenvalues of $g^T g$ which are all real and positive. There is equality $\kL(g)=0$ if and only if $g\in K$ (because this happens if and only if $\xi_1=\cdots= \xi_n=1$).

\item $\kL$ has a canonical extension to $G_{\C}$ which satisfies all of the above 
properties (with $K_{\C}=\U(n)$ instead of $K$). Namely, $\kL(g):=\log\left(\Mtr\bar{g}^T 
g /n\right)$ for $\bar{g}$ the complex conjugate of $g$.
 \end{enumerate}
\end{remark}

\begin{lemma}\label{lem:compare:L:X}
If $\cB\subseteq G$ is a bounded set, then for any $g\in \cB$ we have the inequalities
\[
||X(g)||^2 \ll_\cB \kL(g) \le 2 ||X(g)||.
\]
\end{lemma}
\begin{proof}
We use the Cartan decomposition $g=k_1e^{X(g)} k_2$. It yields $g^T g=k_2^{-1}e^{2X(g)}k_2$. Thus taking traces we obtain
\[
\kL(g)=\log \frac{\Mtr(e^{2X(g)})}{n}.
\]
Since $1\le \Mtr(e^{2X(g)})/ n\le e^{2\|X(g)\|}$ the second  upper bound of the 
lemma 
is clear.

For the lower bound we use a sharp version of the arithmetic-geometric mean inequality from~\cite{Alz97}. Write $X(g)=(X_1, \ldots, X_n)\in\R^n$, $X_1+\ldots+X_n=0$. Then
\[
\frac{1}{n} \Mtr(e^{2X(g)}) -1
 \ge \frac{1}{2\max_{i=1, \ldots, n} e^{2X_i}}
\cdot
\frac{1}{n}
\sum_{i=1}^n (e^{2X_i}-1)^2
\gg_\cB    \sum_{i=1}^n X_i^2
=  \|X(g)\|^2,
\]
where the second lower bound holds because the $X_i$ are bounded above and below
(depending on the choice of $\cB$).
 Hence, taking the logarithm, and using again that $\|X(g)\|$ is bounded above, we obtain
$\kL(g)
\gg_\cB \|X(g)\|^2$
which is the first inequality of the lemma.
\end{proof}

\begin{lemma}\label{lem:inequ:L:Iwasawa}
 Let $P=MU$ be any semi-standard parabolic subgroup in $G$. Then for any $g\in G$
\[
\kL(g)\ge \max\left(\kL(m),\kL(\Mid + m(u-\Mid))\right)
\]
for $g=muk$ an Iwasawa decomposition of $g$ with respect to $P$ and $K$.
\end{lemma}

\begin{proof}
 Let $g=muk$ be the Iwasawa decomposition with respect to $P$ as in the lemma. There exists $x\in G$ such that $P'=x^{-1}Px$ is a standard parabolic with Levi $M'=x^{-1} Mx$ and unipotent radical $U'=x^{-1}Ux$. In fact, $x$ can be chosen in the set of representatives of the Weyl group in $K$. Let $m_1=x^{-1}mx\in M'$, and $u_1=x^{-1}ux\in U'$. Then
\begin{align*}
\kL(g) & =\kL(mu)=\kL(m_1u_1)= \kL(m_1 + (m_1u_1-m_1))\\
& = \log\left(\Mtr\frac{m_1^T m_1}{n} + \Mtr\frac{(m_1u_1-m_1)^T(m_1u_1-m_1)}{n} +2\Mtr\frac{m_1^T (m_1u_1-m_1)}{n}\right)\\
& = \log\left(\Mtr\frac{m_1^T m_1}{n} + \Mtr\frac{(m_1u_1-m_1)^T (m_1u_1-m_1)}{n}\right),
\end{align*}
where we used $\Mtr (m_1^T m_1u_1)=\Mtr m_1^T m_1$ for the last equality. Since $\Mtr(m_1^T m_1)\ge n=\Mtr(\Mid^T \Mid)$, and $\Mtr((m_1u_1-m_1)^T (m_1u_1-m_1))\ge0$, we get $\kL(g)\ge \kL(\Mid+m_1(u_1-\Mid))$ as well as $\kL(g)\ge\kL(m_1)$. Using the definition of $m_1$ and $u_1$ and the bi-$K$-invariance of $\kL$, we can replace $m_1$ by $m$ and $u_1$ by $u$ in these inequalities and obtain the assertion.
\end{proof}

\subsection{Weyl discriminant}
Let $D^G$ be the Weyl discriminant, that is, if $g\in G$ is a semisimple element, let
\[
D^G(g) = \det \left(1 -\TAd(g); \Fmg/\Fmg_{g}\right),
\]
where $\Fmg_g$ is the Lie algebra of the centralizer $C_G(g)$. More generally, if $g\in G$ is arbitrary, and $g_s$ is the semisimple part of $g$ in its Jordan decomposition, we also write $D^G(g):=D^G(g_s)$.

The relationship between the Weyl discriminant $D^G(g)$ and $\|X(g)\|$, $\kL(g)$, is less tight in general. There is no hope for a lower bound for $D^G(g)$ if we let $g$ vary over all of $G$. This is due to the fact that the map $g\mapsto D^G(g)$ has singularities if $g$ changes its ``singularity type", that is, if the rank of the semisimple centralizer $C_G(g_s)$ changes.
\begin{lemma}\label{lem:weyl:discr:est}
	For any $g\in G$,
$
|D^G(g)|^{\frac{2}{n(n-1)}}\ll e^{ \kL(g)}$,
where the implied multiplicative constant depends only on $n$.
In particular $|D^G(g)|\ll e^{n(n-1)\|X(g)\|}$.
\end{lemma}

\begin{proof}
We have that the absolute value $|D^G(g)|$ of the Weyl discriminant can be 
expressed as the absolute value of some 
homogeneous polynomial of degree $n(n-1)$ in the (complex) eigenvalues 
$\lambda_1,\ldots,\lambda_n$ of $g\in G$, with the polynomial depending on the singularity 
type of $g$. Indeed, we claim that
\[
 |D^G(g)| = 
\prod_{\substack{(i,j):\, i\neq j, \\ \lambda_i\neq\lambda_j}} \left| \lambda_j -\lambda_i\right|
 \prod_{\substack{(i,j):\, i\neq j, \\ \lambda_i = \lambda_j}} |\lambda_i|.
\]
(For example if $n=3$ and $g$ has repeated eigenvalues $\lambda_1,\lambda_1,\lambda_2$, 
then $|\lambda_1|^2 |\lambda_2|=1$, and $|D^G(g)|=|\lambda_1-\lambda_2|^4|\lambda_1|^2$.)
To establish our claim, write 
\[
   |D^G(g)| = \prod_{\substack{(i,j):\, \lambda_i\neq\lambda_j}} \left| 1-\frac{\lambda_i}{\lambda_j}\right|
   = \prod_{\substack{(i,j):\, \lambda_i\neq\lambda_j}} \left| \lambda_j -\lambda_i\right| 
   \prod_{\substack{(i,j):\, \lambda_i\neq\lambda_j}} |\lambda_j|^{-1}
\]
and clear the denominators in the second product as follows:
\[
 \prod_{\substack{(i,j):\, \lambda_i\neq\lambda_j}} |\lambda_j|^{-1} =
 \prod_{\substack{(i,j):\, i\neq j, \\ \lambda_i = \lambda_j}} |\lambda_i|
   \prod_{(i,j):\, i\neq j} |\lambda_j|^{-1}
= \prod_{\substack{(i,j):\, i\neq j, \\ \lambda_i = \lambda_j}} |\lambda_i|, 
\]  
where we used for the last equality that $\prod_{i} |\lambda_i|=|\det g|=1$. 

The Gershgorin circle theorem implies that the largest eigenvalue of a matrix is 
bounded by the largest sum of absolute values of elements in a row.
Since every element of $g$ is bounded in absolute value by $\Mtr(gg^T)^{\frac12}$, 
the largest eigenvalue of $g$ is bounded by $n^{\frac32} e^{\kL(g)/2}$,
and the first assertion follows.

Lemma~\ref{lem:compare:L:X} then yields the second asserted upper-bound for 
$|D^G(g)|$.
\end{proof}

\begin{example*}
The regular element 
$g=\textrm{diag}(a,2a,\cdots,(n-1)a,\frac1{(n-1)!a^{n-1}})$ is such that 
$|D^G(g)|\asymp a^{n(n-1)}$, and $e^{\kL(g)}\asymp a^2$ as $a\to \infty$.
Hence the exponent $n(n-1)$ in Lemma~\ref{lem:weyl:discr:est} is sharp. 
\end{example*}

\subsection{Norms on groups}\label{sec:norm-groups}
We define a norm on $G$ by setting
\[
|g|= e^{\|X(g)\|}
\]
where $g=k_1e^{X(g)} k_2\in K\exp(\ka) K$ is the Cartan decomposition of $g$ as before.
We extend the norm on $G_{\C}$ analogously to the real case: if $g\in G_{\C}$, we put 
$|g|=e^{\|X\|}$ for $g=k_1e^Xk_2\in K_{\C}\exp(\ka)K_{\C}$ the Cartan decomposition of $g$.
We have the properties: $|g|\ge 1$, $ |g|=|g^{-1}|$, and $|g_1g_2|\le |g_1||g_2|$.
This notation is well-defined if we consider $g\in G$ as an element in $G_{\C}$ 
since $X(g)$ is the same in the Cartan decomposition for $G$ and for $G_{\C}$.

\begin{lemma}\label{lemma:Xg-Lg}
For any $g\in G$, $\|X(g)\|\le \frac{\sqrt{n(n-1)}}{2}(\log n + \kL(g))$.
In particular, $|g|^{\frac{2}{\sqrt{n(n-1)}}}\le n \cdot e^{\kL(g)}$.
\end{lemma}

\begin{proof}
Write $X(g)=(X_1, \ldots, X_n)\in\R^n$, $X_1+\ldots+X_n=0$. Then it is not difficult to 
verify that
\[
 \|X(g)\| = (\sum\nolimits X_i^2)^{\frac12}  \le \sqrt{n(n-1)} \max\nolimits_i X_i,
\]
with equality achieved for the vector $(1,1,\cdots,1,1-n)$.
On the other hand
\[
 \log n + \kL(g)  = \log(\sum\nolimits_i e^{2X_i}) \ge 2 \max\nolimits_i X_i,
\]
which establishes the claim.
\end{proof}

\begin{lemma}\label{lemma:norm:iwasawa}
 There exist constants $c,c_1, c_2>0$, such that if $g=muk\in G_\C$ with $k\in K_\C$ 
 and $mu\in P=MU$ for $P$ a standard parabolic subgroup in $G_\C$, 
 then $|m|\le c |g|^{c_1}$, and $|u|\le c |g|^{c_2}$.

More precisely, we can take $c= n^{\frac{\sqrt{n(n-1)}}{2}} $, $c_1=\sqrt{n(n-1)}$, 
and $c_2=\sqrt{n(n-1)}+1$.
\end{lemma}

\begin{proof}
We can assume without loss of generality that $k=1$.
 Hence
\begin{equation*}
 |m|^{\frac{2}{\sqrt{n(n-1)}}}
\le n\cdot e^{\kL(m)} \le n\cdot  e^{\kL(g)}
\le n\cdot e^{2\|X(g)\|} = n\cdot |g|^2.
\end{equation*}
where the first estimate is Lemma~\ref{lemma:Xg-Lg}, the second inequality 
is Lemma~\ref{lem:inequ:L:Iwasawa}, 
and the third inequality is Lemma~\ref{lem:compare:L:X}.
This establishes the first estimate.
For the second estimate of the lemma, we write
\[
 |u|= |m^{-1} mu|\le |m| |mu|
\le n^{\frac{\sqrt{n(n-1)}}{2}} |g|^{\sqrt{n(n-1)}+1}.
\qedhere
\]
\end{proof}

\begin{lemma}\label{lemma:norm:jordan}
Let $c, c_1, c_2$ be as in Lemma~\ref{lemma:norm:iwasawa}. Suppose $g\in G$ has Jordan decomposition $g=g_sg_u$ with $g_s$ semisimple and $g_u\in C_G(g_s)$ unipotent. Then
\[
|g_s|\le c |g|^{c_1},\;\;\text{ and }\;\;
|g_u|\le c |g|^{c_2}.
\]
\end{lemma}
\begin{proof}
There exists $k\in K_{\C}$ such that $k^{-1}gk$ is upper triangular, more 
precisely, $k^{-1} g_sk$ is diagonal, and $k^{-1}g_uk$ is an upper triangular 
unipotent matrix.
The assertion then follows from the previous lemma and $|g_s|=|k^{-1}g_sk|$.
\end{proof}

We note that there is a $p$-adic analogue of these norms, see
e.g.~\cite{BrTi72}*{(4.4.4)} for the proof of submultiplicativity, 
and~\cite{Art91}*{\S4} for some other properties.

\section{Setting for the main estimate}\label{sec:proof}
This section is to set up the notation and give some preliminaries for the proof of Theorem~\ref{thm:bound:orb:int:real} which will be given in Section~\ref{sec:orbital}.

\subsection{Twisted Levi subgroups}\label{sec:twisted}
For $M\in\CmL:=\CmL(T_0)$, consider the maximal tori contained in $M$, not 
necessarily $\R$-split.
There are only finitely many $M$-conjugacy classes of such maximal tori. We shall
choose a
finite set $\CmT^M_{\max}$ of representatives $T$ for these conjugacy classes such that
the minimal Levi subgroup $L\subseteq G$ containing $T$ is semi-standard, that is,
$L$ also contains the torus $T_0$.

Such a set of representatives can be realized as follows. Let $r_1, 
r_2\in\Z_{\ge0}$ with $r_1+2r_2=n$. Consider the maximal torus
$T_{r_1,r_2}'=\left(\GL_1\right)^{r_1}\times\left(\Res_{\C/\R} \GL_1 \right)^{r_2}$ embedded diagonally in $\GL_n$. Here and in the following we identify $(\Res_{\C/\R} \GL_1)(\R )=\GL_1(\C)$ with $\R^{\times}\SO(2)\subset \GL_2(\R)$. More precisely, $T_{r_1,r_2}'(\R)$ consists of matrices of the form
\begin{equation}\label{eq:max:twisted:tori}
\Mdiag\left(t_1, \ldots, t_{r_1}, \left(\begin{smallmatrix}a_1&b_1\\-b_1&a_1\end{smallmatrix}\right), \ldots, \left(\begin{smallmatrix}a_{r_2}&b_{r_2}\\-b_{r_2}&a_{r_2}\end{smallmatrix}\right)\right)
\end{equation}
with $t_1, \ldots, t_{r_1}\in\R^{\times}$ and $(a_1, b_1), \ldots, (a_{r _2}, b_{r_2})\in\R^2\backslash\{(0,0)\}$.
Let $T_{r_1,r_2}$ be the subset of all elements $t\in T_{r_1, r_2}'$ with $|\det t|=1$.
The minimal Levi subgroup in $\GL_n$ containing $T'_{r_1,r_2}$ is the
diagonally
embedded $(\GL_1)^{r_1}\times (\GL_2)^{r_2}$,  which is standard. The same holds for
$T_{r_1,r_2}\subset G$, and we can take
\[
\CmT^G_{\max}:=\{T_{r_1,r_2}\mid r_1, r_2\in\Z_{\ge0},~r_1+2r_2=n\}.
\]
If $M\in\CmL$ is arbitrary, it is conjugate by a Weyl group  element $w\in W$ to the standard Levi subgroup $(\GL_{n_1}\times\ldots\times\GL_{n_s})\cap G$, for suitable integers $n_1, \ldots, n_s\in\Z_{\ge1}$, $n_1+\ldots+n_s=n$. The set $\CmT^M_{\max}$ can then be chosen to consist of tori which are $w$-conjugates of concatenations of elements of the form~\eqref{eq:max:twisted:tori} with overall $|\det|$ equal to $1$.

A \emph{twisted} Levi subgroup is an $\R$-subgroup $L\subseteq G$ such that $L_\C=L\otimes_{\R}\C$ is a Levi subgroup in $G_{\C}$.  For $T\in\CmT^M_{\max}$ let $\CmL^M_{\text{twist}}(T)$ be the set of all twisted Levi subgroups in $M$ containing $T$ and having the same 
$\R$-rank as $T$.

These can again be described in terms of restriction of
scalars as follows. 
Any $L\in\CmL^G_{\text{twist}}(T_{r_1,r_2})$
is of the form $\left(L^1\times \Res_{\C/\R}L^2\right)\cap G$ for some semi-standard Levi 
subgroups
$L^1\in\CmL^{\GL_{r_1}(\R)}$ and $L^2\in \CmL^{\GL_{r_2}(\C)}$, which both contain the
maximal diagonal torus of $\GL_{r_1}(\R)$ and $\GL_{r_2}(\C)$, respectively. 
(The reader may take this description as a 
definition because it is how these subgroups shall arise in below). 
Both $T_{r_1,r_2}$ and $L$ have $\R$-rank equal to $r_1+r_2-1$.
To be 
precise, we embed here $(\Res_{\C/\R} \GL_{r_2})(\R)=\GL_{r_2}(\C)$ into 
$\GL_{2r_2}(\R)$  
by the map
\[\begin{pmatrix}
a_{11} + i b_{11} & \ldots & a_{1 r_2}+ib_{1 r_2} \\
\vdots & & \vdots \\
a_{r_2 1} + i b_{r_2 1} & \ldots & a_{r_2 r_2} + i b_{r_2 r_2} 
\end{pmatrix}
\mapsto
 \begin{pmatrix}
       a_{11}	& b_{11}		&\ldots	&a_{1 r_2}	&b_{1r_2}\\
       -b_{11}	& a_{11}		&\ldots	&-b_{1r_2}	&a_{1r_2} \\
       \vdots	& \vdots		&	&\vdots		&\vdots \\
        a_{r_21}	& b_{r_21}		&\ldots	&a_{r_2 r_2}	
        &b_{r_2 r_2}\\
       -b_{r_21}	& a_{r_21}		&\ldots	&-b_{r_2r_2}	
       &a_{r_2 r_2} \\
       \end{pmatrix}.
\]
Since there are only finitely many $r_1, r_2\ge0$ with $r_1+2r_2=n$ and since the sets
$\CmL^{\GL_{r_1}(\R)}$ and $\CmL^{\GL_{r_2}(\C)}$ are both finite, it follows that
$\CmL^G_{\text{twist}}(T_{r_1,r_2})$ is finite.
A description of $\CmL^M_{\text{twist}}(T)$ for general $M\in \CmL$ and $T\in 
\CmT^M_{\text{max}}$
can be obtained by $w$-conjugation as above, in particular
 $\CmL^M_{\text{twist}}(T)$ is again finite.

\begin{lemma}\label{lem:centralizer:twisted:levi}
For every $M\in \cL$, and every semisimple element $\sigma$ of $M$, there exist an
$M$-conjugate $\sigma'$ of $\sigma$, a torus $T\in\CmT^M_{\max}$, and twisted Levi
subgroups $L_1\in
\CmL^M_{\text{twist}}(T)$ and $L_2\in\CmL^G_{\text{twist}}(T)$ with $L_1\subseteq L_2$
such that
\begin{equation}\label{gammas-condition}
 \sigma' \in T,\
C_M(\sigma' )=L_1,\
 C_G(\sigma' )=L_2.
\end{equation}
\end{lemma}
\begin{proof}
It is automatic that $C_G(\sigma')$ is a twisted Levi subgroup because 
$C_G(\sigma')\otimes_\R \C = C_{G_\C}(\sigma')$ is a Levi subgroup of $G_\C$.
It follows from the other assertions that $C_G(\sigma')=L_2$ is an element of 
$\CmL^G_{\text{twist}}(T)$. Thus we restrict to constructing $T$ and $L_1$.
Without loss of generality we may then only consider the case $M=G$. 

We have that $\sigma$ is $G$-conjugate to $\sigma':=\Mdiag(g_1, \ldots, g_{r_1}, 
g_{r_1+1}, 
\ldots, 
g_{r_1+r_2})$ with $r_1+2r_2=n$, $g_{r_1+1}, \ldots, g_{r_1+r_2}\in \R^{\times}\SO(2)$ 
with irreducible characteristic polynomial, and $g_{1}, \ldots, 
g_{r_1}\in\R^{\times}$, and we can assume that equal $g_i$ occur in consecutive 
order. 
There exist partitions $(s_1,\ldots,s_a)$ of $r_1$ and $(t_1, \ldots, t_b)$ of 
$r_2$ such that the Levi subgroup $M_1$ in $G$ corresponding to $(s_1, \ldots, 
s_a,2t_1, \ldots, 2t_b)$ is the maximal Levi in which $\sigma$ is elliptic. Let 
$\delta=\Mdiag(1, \ldots, 1,\delta_0, \ldots, \delta_0)$ with $r_1$-many $1$s, and $r_2$-many 
$\delta_0:=\left(\begin{smallmatrix}0&1\\-1&0\end{smallmatrix}\right)$. 
Let $T$ to be the torus $T_{r_1,r_2}$ defined above, and 
let $L_1$ be the 
set of fixed points of $M_1$ under conjugation by $\delta$. 
Then $T$ and $L_1$ satisfy~\eqref{gammas-condition}. Moreover $L_1\in 
\CmL_{\text{twist}}^M(T)$ because $L_1=\left(L^1\times \Res_{\C/\R}L^2\right)\cap G$ where
$L^1\in\CmL^{\GL_{r_1}(\R)}$ is the standard Levi subgroup of type $(s_1,\ldots,s_a)$ 
and $L^2\in \CmL^{\GL_{r_2}(\C)}$ is the standard Levi subgroup of 
type $(t_1, \ldots, t_b)$.
\end{proof}

\subsection{Convention}\label{s_convention}
We fix one of the finitely many triples of subgroups $T\subseteq L_1 \subseteq L_2$ as in
Lemma~\ref{lem:centralizer:twisted:levi}. In particular, $T\in\CmT^M_{\max}$, $L_1\in
\CmL^M_{\text{twist}}(T)$ and $L_2\in\CmL^G_{\text{twist}}(T)$.
Then, we shall denote by $M_2$ the smallest Levi
subgroup in $G$ containing $L_2$. Since $T\subseteq M_2$, we have that $M_2$ is
semi-standard.
In establishing the main estimate in Section~\ref{sec:orbital},
we shall restrict to the semisimple elements $\gamma_s=\sigma'\in M$
satisfying~\eqref{gammas-condition}.
We are allowed to proceed in this way for the proof because there are only finitely many
possible choices of a semi-standard Levi subgroup $M\in \cL$ and of a triple $T,L_1,L_2$.

\begin{remark}
Arthur makes similar choices in his study of weighted orbital integrals,
e.g.~\cite{Ar86}*{p.183},~\cite{Ar88a}*{p.230}. One subtle difference is that he has
the flexibility to choose $K$ within the set of all maximal compact subgroups which are
admissible relative to $M$ (in the sense of~\cite{Ar81}*{p.9}). Our type of test
functions depends by definition on the choice of $K$ (since $X(\cdot)$ does) so that
taking $K$ as the maximal compact subgroup to analyze the weighted orbital integrals is
canonical in our situation.
\end{remark}

\subsection{Measures}\label{sub_measures}
Let $F=\R$ or $F=\C$.
The measure $dx$ on $F$ will be the usual Lebesgue measure if $F=\R$, and twice the usual Lebesgue measure if $F=\C$. On $F^{\times}$ as well as $\R_{>0}$ we use the multiplicative measure $|x|_F^{-1}dx$ where $|x|_{\R}=|x|$ is the usual absolute value and $|x|_{\C}=|x|^2=x\bar{x}$. From this we obtain a measure on the unipotent radical $U$ of any semi-standard parabolic subgroup in $\GL_n(F)$ by identifying $U$ with $F^{\dim U}$ via the usual matrix coordinates. We also obtain a measure on the split maximal torus $T_0'$ of diagonal matrices in $\GL_n(F)$ by identifying $T_0'$ with $(F^{\times})^n$ via the usual matrix coordinates again. On $K=:K_{\R}$ and $K_{\C}$ we normalize the Haar measure such that $\vol(K)=\vol(K_{\C})=1$. The integration formula
\[
\int_{\GL_n(F)} f(g)~dg
=\int_{T_0'}\int_{U_0}\int_{K_F} f(tuk) ~dk~du~dt,
~~~~f\in C_c^{\infty}(\GL_n(F)),
\]
then fixes an invariant Haar measure on $\GL_n(F)$. To fix a measure on $\GL_n(F)^1$, we use the short exact sequence
\[
1\longrightarrow \GL_n(F)^1\longrightarrow \GL_n(F)\longrightarrow\R_{>0}\longrightarrow 1,
\]
where the map $\GL_n(F)\longrightarrow\R_{>0}$ is given by $g\mapsto |\det g|$.
In the same way, this fixes measures on all semi-standard Levi subgroups.

For twisted Levi subgroups, the pullback under the restriction of scalars defines measures on the twisted Levi subgroup as well as on all its (semi-)standard parabolic subgroups and their unipotent and Levi parts.

\subsection{Generalized Cartan and Iwasawa decomposition}\label{sub_gen-Cartan}
To make integration over the quotient $L_2\backslash G$ more explicit we use a mix 
of a generalized Cartan decomposition and the Iwasawa decomposition.  Let $P_2$ 
be a parabolic subgroup with Levi component $M_2$ so that the Iwasawa 
decomposition $G=P_2K=M_2U_2K$ holds. We then apply a generalized Cartan 
decomposition to $L_2\SB M_2$ as in~\cite{FlJe80} as follows. First we need to 
identify $L_2$ as a fixed point set of an involution of $M_2$.

\begin{lemma}
 $L_2$ is the fixed point set in $M_2$ of an involution $\sigma:M_2\longrightarrow M_2$.
Moreover, $\sigma $commutes with the Cartan involution $\theta:M_2\longrightarrow M_2$ defining $K^{M_2}=K\cap M_2$.
Hence $\sigma(K^{M_2})= K^{M_2}$,  $\theta(L_2)=L_2$, and  $K_2:=L_2^{\theta}=L_2\cap K$ is a maximal compact subgroup of $L_2$.
\end{lemma}
\begin{proof}
 It suffices to consider the case $M_2=G$. Then $L_2=G$ or $L_2=(\Res_{\C/\R} 
 \GL_{n/2}) \cap G$. In the first case we take $\sigma$ as the identity. In the 
 second case (which can only happen for even $n$) put $\delta=\Mdiag(\delta_0, \ldots, 
 \delta_0)\in G$ with $\delta_0=\left(\begin{smallmatrix}0&1\\-1&0\end{smallmatrix}\right)$, 
 and define $\sigma(g):=\delta^{-1} g\delta$ for $g\in G$. It is easily checked that 
 $\sigma$ satisfies the asserted properties.
\end{proof}

Let $\km_2$ denote the Lie algebra of $M_2$, $\km_2=\kl_2\oplus\kq_2$ the decomposition of $\km_2$ into the $+1$- and $-1$-eigenspace under $\sigma$, and $\km_2=\kk_2\oplus\kp_2$ the decomposition of $\km_2$ into the $+1$- and $-1$-eigenspace under $\theta$.
Let $\km_2=\km_2^+\oplus\km_2^-$ be the $\pm1$-eigenspace decomposition with respect to $\theta\sigma$.
Hence $\km_2^+=\kl_2\cap \kk_2\oplus\kq_2\cap \kp_2$, and $\km_2^-=\kl_2\cap \kp_2\oplus\kk_2\cap\kq_2$.
Let $\kb_2\subseteq \kq_2\cap\kp_2$ be a maximal abelian subspace.
By construction we may assume that $\kb_2$ is a subspace of $\ka$. 

Let $\Phi_{\kb_2}$ be the set of roots of $(\kb_2,\km_2)$, and fix a subset of positive roots $\Phi_{\kb_2}^+\subseteq \Phi_{\kb_2}$. We denote by $\kb_2^+=\{H\in\kb_2\mid\forall\beta\in\Phi_{\kb_2}^+:~\beta(H)>0\}$ the corresponding positive chamber in $\kb_2$.
For $\beta\in\Phi_{\kb_2}^+$ let $m_{\beta}^{\pm}$ denote the multiplicity of $\beta$ when restricted to $\km_2^{\pm}$.
For $H\in\kb_2$ put
\[
 B_{\kb_2}^{M_2}(H)
=\prod_{\beta\in\Phi_{\kb_2}^+} |\sinh\beta(H)|^{m_{\beta}^+} |\cosh\beta(H)|^{m_{\beta}^-}.
\]

\begin{proposition}\label{prop:integration:formula}
 We have
\[
 \int_{L_2\backslash G} f(g)~dg
= \int_{\kb_2^+}\int_{U_2} \int_{K} f(e^{H} uk) B_{\kb_2}^{M_2}(H)~ dk~du~dH
\]
for every integrable function $f: L_2\backslash G\longrightarrow\C$.
\end{proposition}
\begin{proof}
By definition of our measure on $G$ we have
\[
  \int_{L_2\backslash G} f(g)~dg
=\int_{L_2\backslash M_2}\int_{U_2} \int_{K} f(muk) ~dk~du~dm.
\]
The invariant measure on $L_2\backslash M_2$ is given by~\cite{FlJe80}*{Thm.2.6}, see 
also~\cite{HeSch}*{p.110}, and plugging this into our integral, we get
\[
  \int_{L_2\backslash G} f(g)~dg
=\vol(K\cap M_2) \int_{\kb_2^+}\int_{K^{M_2}} \int_{U_2} \int_{K} f(e^{H} k_M uk) 
B_{\kb_2}^{M_2}(H)~ dk~du~dk_M~dH.
\]
Since $U_2$ is normalized by $K^{M_2}=K\cap M_2$ and $\vol(K^{M_2}) =1$, we can 
absorb the variable $k_M\in K^{M_2}$ into the integration over $k\in K$, which 
gives the assertion.
\end{proof}

\begin{example}
Suppose that $G=\GL_2(\R)^1$. Suppose $L_2=\SO(2)$ so that $M_2=G$, and $\sigma:M_2\longrightarrow M_2$ is given by conjugation with $\delta_0=\left(\begin{smallmatrix}0&1\\-1&0\end{smallmatrix}\right)$. Then $\sigma=\theta$ coincides with the Cartan involution.
Hence the decomposition of $\Fmm_2=\Fmg$ into $\pm1$-eigenspaces of $\sigma$ is the usual Cartan decomposition of $\Fmg$, and $\kb_2=\{H=(H_1,-H_1)\mid H_1\in \R\}$. Therefore,
\[
B_{\kb_2}^{M_2}(H) = |\sinh (2H_1)|.
\]
\end{example}

\begin{example}
Suppose  $G=\GL_4(\R)^1$, $L_1=\{g=\Mdiag(ax,a^{-1} y)\mid x,y\in\SO(2), a\in\R^{\times}\}$ 
(diagonally embedded in $G$), and let $L_2$ be obtained from restriction of 
scalars of $\GL_2(\C)^1$ so that $L_1$ is a maximal torus of $L_2$. Then $M_2=G$, 
and $\sigma:M_2\longrightarrow M_2$ is given by conjugation with $\delta=\Mdiag(\delta_0, 
\delta_0)$ for $\delta_0=\left(\begin{smallmatrix}0&1\\-1&0\end{smallmatrix}\right)$. 
Further,  $\kb_2=\{H=\Mdiag(H_1, -H_1, H_2, -H_2)\}\mid H_1, H_2\in\R\}$, and 
$\Phi_{\kb_2}$ consists of the roots given by mapping $H\in\kb_2$ to $\pm2H_1$, 
$\pm2H_2$, or $\pm H_1 \pm H_2$. We choose $\Phi_{\kb_2}^+$ to consist of the roots 
given by $2H_1$, $2H_2$, and $H_1\pm H_2$. Then
\[
\kb_2^+=\{H=\Mdiag(H_1,-H_1, H_2, -H_2)\mid H_1>H_2>0\},
\]
 and for $H\in\kb_2^+$ we get
\[
 B_{\kb_2}^{M_2}(H)
=\left|\sinh(2H_1)\sinh(2H_2)\sinh(H_1+H_2)\sinh(H_1-H_2)\cosh(H_1+H_2)\cosh(H_1-H_2)\right|.
\]
\end{example}

\section{Bounds for semisimple orbital integrals}\label{sec:HC}

In this section we treat a particular (unweighted) case which will illustrate our strategy in the general weighted case.
Recall that $J_G^G(\gamma,f)$ is the unweighted orbital integral
\[
|D^G(\gamma)|^{1/2}\int_{C_G(\gamma)\backslash G} f(x^{-1}\gamma x)\, dx,
\]
which is defined for any $f\in C_c^\infty(G)$.
The orbital integral without the normalizing factor $|D^G(\gamma)|^{1/2}$ is  denoted by $\CmO_{\gamma}^G(f)$, or $\CmO_{\gamma}(f)$.  We shall be able to directly use results of Harish-Chandra on orbital integrals and his descent formula.

\begin{proposition}\label{prop:ss:orb:bound:HC}
	Let $0<\eta<(n-1)/2$ and $f\in C_c^{\infty}(G)$.
	There is a constant $c(f,\eta)>0$ depending only on $\eta$ and $f$, such that the following holds.
\begin{enumerate}[(i)]
 \item Suppose $n\ge3$. For every regular semisimple $\gamma\in G$,
\[
 \left|J_G^G(\gamma,f\|X(\cdot)\|^{-\eta})\right|
\le c(f,\eta).
\]

\item For every semisimple $\gamma\in G$ which is split and not central, that is, $C_G(\gamma)\subsetneq G$,
\[
 \left|J_G^G(\gamma,f\|X(\cdot)\|^{-\eta})\right|
\le c(f,\eta).
\]
\end{enumerate}
\end{proposition}
We prove this proposition below.

We recall the definition of parabolic descent. Suppose $f:G\longrightarrow\C$ is an integrable function, and $Q=LV$ is an arbitrary semi-standard parabolic subgroup in $G$. The parabolic descent along $Q$,
\[
 C_c^\infty(G)\longrightarrow C_c^\infty(L),~ f\mapsto f^{(Q)},
\]
is then defined by
\begin{equation}\label{def:fQ}
 f^{(Q)}(m):=\delta_Q^{1/2}(m)\int_{V}\int_K f(k^{-1} mv k)~dk~dv, ~ m\in L.
\end{equation}
Note that the dependency of $f^{(Q)}$ on $Q\in\CmP(L)$ can be made explicit as follows: If $Q'$ is another parabolic subgroup with Levi component $L$, there is $w\in W$ with $wQw^{-1}=Q'$ which preserves $L$. Then $f^{(Q)}(w^{-1}mw)=f^{(Q')}(m)$

If $\gamma\in L$ is such that $C_G(\gamma)\subseteq L$, the parabolic descent relates the (normalized) orbital integrals on $G$ and $L$.  More precisely, for any $f\in C_c^\infty(G)$ we have
\begin{equation}\label{eq:parabolic:descent}
 J^G_G(\gamma,f)= J_L^L(\gamma,f^{(Q)})
\end{equation}
for any $Q\in\CmP(L)$ provided that the measures on all involved groups are chosen compatibly. This formula follows from the definitions or from the more general descent formula~\cite{Ar94}*{(1.5)}.

\begin{lemma}\label{l:unidescent}
Suppose $Q=LV\subsetneq G$ is a proper semi-standard parabolic subgroup, and $0 < \eta < (n-1)/2$. Let $f\in C_c^{\infty}(G)$, and define $F_{\eta}:=f\|X(\cdot)\|^{-\eta}$.
\begin{enumerate}[(i)]
\item The parabolic descent $F^{(Q)}_\eta(m)$ converges absolutely for every $m\in L$.
\item\label{eq:bounded:ss:descent} We have $F^{(Q)}_\eta\in C_c(L)$.

\item The descent formula~\eqref{eq:parabolic:descent} holds for $F_\eta$, i.e.,
$J^G_G(\gamma,F_\eta)$ converges absolutely for any $\gamma\in L$ such that $C_G(\gamma)\subseteq L$, and is equal to $J^L_L(\gamma,F_\eta^{(Q)})$.
\end{enumerate}
\end{lemma}

\begin{remark}
 We shall apply the property in \eqref{eq:bounded:ss:descent} to more general weighted orbital integrals later. This will simplify our analysis in several (but not all) cases.
\end{remark}

\begin{proof}[Proof of Lemma \ref{l:unidescent}]
For the proof of the lemma it suffices to show that \eqref{eq:bounded:ss:descent} holds for $f$ replaced by its absolute value $|f|$ in the definition of $F_{\eta}$. We can further assume that $f$ is $K$-conjugation invariant.

Let $\cC\subseteq G$ be a compact set containing the support of $f$.
Note that $F^{(Q)}(m)=0$ unless $m$ is contained in a compact subset $\cC^L\subseteq L$ depending only on $\cC$.
Since $\|X(g)\|^{-\eta}\le 2^{\eta} \CmL(g)^{-\eta}$ for all $g\in \cC$ by Lemma \ref{lem:compare:L:X}, we get
\begin{align*}
\int_{V} |f(mv)| \|X(mv)\|^{-\eta}~dv
&\le 2^{\eta} \int_{V} |f(mv)| \CmL(mv)^{-\eta}~dv\\
&\ll_{\eta,f}  \int_{(m^{-1}\cC) \cap V} \CmL(mv-m+\Mid)^{-\eta}~dv,
\end{align*}
where we used Lemma \ref{lem:inequ:L:Iwasawa} for the second inequality.
To bound this last integral we can clearly assume that $Q$ is a standard parabolic subgroup so that $L$ is of the form $\GL_{n_1}\times\ldots\times\GL_{n_r}$ for some $r$ and $n_1+\ldots+n_r=n$, and we can identify $V$ with $\R^{\dim V}$ by using the matrix entries of elements of $V$.
Changing variables, the integral becomes
\[
\prod_{j=1}^r |\det m_j|^{-(n-n_1-\ldots-n_j)} \int_{\cC'} \left(\log \left(1+\frac{1}{n}\sum_{i=1}^{\dim V} v_i^2\right)\right)^{-\eta} ~d v
\]
where $\cC'=\left(\bigcup_{m\in \cC^L} m^{-1}\cC\right)\cap V \subset \R^{\dim V}$ is a compact subset depending only on $\cC$.
As $m$ is contained in $\cC^M$, the product over the determinants is bounded by a constant depending only on $\cC$. Using polar coordinates one sees that the last integral is finite for any $\eta\in[0,\dim V/2)$. Note that $\dim V\ge n-1$.
In any case, if the last integral is finite, its value depends only on $\cC$ and $\eta$, therefore the claim follows by the dominated convergence theorem.
\end{proof}

\begin{proof}[Proof of Proposition \ref{prop:ss:orb:bound:HC}]
A result of Harish-Chandra~\cite{HC57}*{Thm.2} (cf. also \cite{Art91}*{p.31}, and
\cite{HaChHA}*{Thm.14}) asserts that if $H\subseteq G$ is a Cartan subgroup,
$\cC\subseteq H$ a compact subset, and  $G'$ is the set of regular semisimple
elements in $G$, then for every $f\in C^{\infty}_c(G)$ we have
\[
 \sup_{\gamma\in \cC\cap G'} \left|J_G^G(\gamma,f)\right|<\infty.
\]
Up to $G$-conjugation there are only finitely many Cartan subgroups in $G$, and the centralizer $C_G(\gamma)$ of a regular semisimple element $\gamma\in G'$ is a Cartan subgroup. Hence if $\cC\subseteq G$ is a compact set and $f\in C^{\infty}_c(G)$, then there exists a constant $c=c(f,\cC)>0$ such that
\[
\left|J_G^G(\gamma,f)\right| \le c
\]
for all $\gamma\in \cC\cap G'$.
On the other hand, since $f$ is compactly supported, there exists a compact subset $\cC_f\subseteq G$ such that $J_G^G(\gamma,f)$ vanishes for all $\gamma$ which are not conjugate to some element in $\cC_f$.

To prove the first part of the proposition, first note that $n\ge3$ and the regularity of $\gamma$ imply that the centralizer $C_G(\gamma)$ of $\gamma$ in $G$ is contained in a proper parabolic subgroup $Q=LV$ of $G$. After conjugating $\gamma$ if necessary, we can assume that $Q$ is standard. The parabolic descent \eqref{eq:parabolic:descent} implies that
\[
\left| J_G^G(\gamma,F_\eta)\right|
= \left|J_L^L(\gamma,F_\eta^{(Q)})\right|
\le J_L^L(\gamma, f)
\]
where the last inequality follows from Lemma \ref{l:unidescent} with $f=|F_\eta^{(Q)}|\in C_c(L)$. We then use Harish-Chandra's bound discussed at the beginning of the proof to uniformly bound $J_L^L(\gamma, f)$.

The second part of the proposition is also an easy consequence of Lemma \ref{l:unidescent}. We can assume that $C_G(\gamma)$ equals the Levi component $L$ of a proper standard parabolic subgroup of $G$. Call this standard parabolic $Q$ and define $F_\eta$ as in Lemma \ref{l:unidescent}. By the parabolic descent formula \eqref{eq:parabolic:descent} we get $J_G^G(\gamma,F_\eta)= J_L^L(\gamma,F_\eta^{(Q)})$. Let $f$ be as in the second part of Lemma \ref{l:unidescent}. As above we obtain
$\left| J_G^G(\gamma,F_\eta)\right|
\le J_L^L(\gamma, f)$.
Since now $D^L(\gamma)= D^{C_G(\gamma)}(\gamma)=1$ and $J_L^L(\gamma, f)= \CmO^L_\gamma(f)=\CmO_{\gamma}^{C_G(\gamma)}(f)=f(\gamma)$, the second assertion of the proposition follows.
\end{proof}

\begin{example}\label{ex:split-regular}
Hence if $Q$ is a minimal semi-standard parabolic subgroup, $f^{(Q)}$ is the Abel
transform of $f$, an archimedean analogue of the Satake transform.
It is closely related to the spherical transform, see \cite{Ga71}*{Thm.3.5}, and
also~\cite{book:helg:geometricanalysis}*{IV.\S7},\cite{LM09}*{Thm.1}. We shall
return to this in Example~\ref{ex:orbital-split-regular} below.
\end{example}

The following result shows that the parabolic descent, restricted to bi-$K$-invariant functions, is continuous for the $L^1$-norm. Thus, it follows from general principles that $F^{(Q)}_\eta\in L^1(L)$ because $X(g)$ is bi-$K$-invariant, and thus $f ||X(\cdot)||^{-\eta} \in L^1(K \backslash G /K)$ for any $f\in C_c^\infty(K\backslash G /K)$ and $0< \eta < (n-1)/2$.
If $f$ is right-$K$-invariant, write $f_K(g):=\int_K f(k^{-1} g k)\,dk = \int_K f(kg)\,dk$, which is bi-$K$-invariant.

\begin{lemma}\label{l:L1-descent}
 Let $Q=LV$ be a semi-standard parabolic subgroup of $G$ and let $K^L:=K\cap L$, which is a maximal compact subgroup of $L$.
 If $f\in C_c^\infty(G/K)$, then $f^{(Q)}=f^{(Q)}_K$ is bi-$K^L$-invariant.
 Moreover $f\mapsto f^{(Q)}$ extends from the dense subspace $C^\infty_c(G/K)$ to define a continuous map
 \[
 L^1(G/K) \longrightarrow L^1(K \backslash G/K) \longrightarrow L^1(K^L \backslash L/K^L).
 \]
\end{lemma}
\begin{proof}
The assertion that $f^{(Q)}=f^{(Q)}_K$ is clear, and thus without loss of generality we may assume that $f$ is bi-$K$-invariant.
By the triangle inequality,
 \[
  \int_L |f^{(Q)}(m)|\, dm
  \le  \int_L \int_V \int_K \delta_Q(km)^{1/2} |f(mv)| \,dk \,dv\, dm,
 \]
 where we extend $\delta_Q$ to all of $G$ via Iwasawa decomposition, $\delta_Q(g)=\delta_Q(l)$ if $g=lvk'\in LVK$.
 Note that $\delta_Q= \delta_0 \delta_{L\cap P_0}^{-1}$ with $\delta_{L\cap P_0}^{-1}$
 denoting the modulus function of $T_0$ on $L\cap P_0$. Let $\rho_0^L\in \ka^*$ denote
 the element corresponding to $\delta_{L\cap P_0}$, that is, $\rho_0^L$ is the half
 sum of all roots of $L\cap P_0$ with respect to $T_0$. Similarly, let $\rho_0$ be
 the half sum of all positive roots of $P_0$ with respect to $T_0$. Using the
 bi-$K$-invariance of $f$ we can write the last integral by
 \cite{Knapp:beyond}*{Prop.8.44} as
 \[
  \int_G \phi_{-\rho_0^L}(g) |f(g)|\, dg
 \]
with $\phi_\lambda$ the zonal spherical function with spherical parameter $\lambda$,
see also Section~\ref{sec:spherfcts}. Let $w_L$ denote the longest element in the
Weyl group $W^L$ of $(T_0,L)$. Naturally $W^L$ is a subgroup of $W=W^G$ so that
$w_L\in W$. Then $-\rho_0^L=-\frac{1}{2} (\rho_0- w_L \rho_0)$ so that $-\rho_0^L$ is
contained in the closure of the convex hull of the Weyl group orbit of $\rho_0$.
Hence $\phi_{-\rho_0^L}$ is a bounded function by
\cite{book:helg:geometricanalysis}*{Ch. IV, Thm.8.1}. Let $c>0$ be an upper bound
for $\phi_{-\rho_0^L}$ (it can be shown that $c=1$). Then
\[
 \int_L |f^{(Q)}(m)|\, dm
 \le \int_G \phi_{-\rho_0^L}(g) |f_K(g)|\, dg
 \le c\int_G |f_K(g)|\, dg
 \le c\int_G |f(g)|\, dg <\infty
\]
so that $f^{(Q)}\in L^1(L)$.

Finally, for any $k\in K^L$ and $m\in L$ we have
\[
\delta_Q(mk)^{-1/2}f^{(Q)}(mk)
= \int_{V} f_K(mkv)\, dv
= \int_V f_K(mvk)\, dv
= \int_V f_K(mv)\, dv
= \delta_Q(m)^{-1/2}f^{(Q)}(m)
\]
so that $f^{(Q)}$ is also right $K^L$-invariant. Similarly $f^{(Q)}$ is left $K^L$-invariant.
\end{proof}

\section{Weighted orbital integrals}\label{sec:orbital}
In this section we prove Theorem \ref{thm:bound:orb:int:real}, that is, we are going to
find an upper bound for the weighted orbital integrals $J_M^G(\gamma,
f\|X(\cdot)\|^{-\eta})$. It will be a consequence of Proposition
\ref{prop:bound:weight:orb:non:central} and
Proposition~\ref{prop:weight:orb:int:central}.  We keep the notation as in 
\S\ref{s_convention}, using that $J_M^G$ is 
invariant by $M$-conjugation. In particular, we fix a triple $T\subseteq L_1
\subseteq L_2$, and we restrict to those $\gamma\in M$ such that $\gamma_s\in T$, 
$C_M(\gamma_s)=L_1$ and
$C_G(\gamma_s)=L_2$.

\subsection{Weighted orbital integrals: the semisimple part}\label{sec:weighted:real}
We first provide a detailed description of the distribution $f\mapsto J_M^G(\gamma, f)$.
By results of \cite{Ar88a}, it can be defined as follows. For every $f\in
C_c^{\infty}(G)$ and every $\gamma\in M$, one has~\cite{Ar88a}*{Thm.8.5}
\begin{equation}\label{eq:def:weighted:orb}
J_M^G(\gamma, f):=
 |D^G(\gamma)|^{1/2}\int_{L_2\backslash G}\sum_{R\in \CmF^{L_2}(L_1)} J_{L_1}^{M_R}(\gamma_u, \Phi_{R, y}) ~dy,
\end{equation}
where the function $\Phi_{R, y}: M_R\longrightarrow \C$ is defined by
\[
 \Phi_{R,y}(m)
:=\delta_{R}(m)^{1/2}\int_{K^{L_2}}\int_{N_R} f(y^{-1}\gamma_s k^{-1} m n k y) v_R'(ky) 
~dn~dk.
\]
Here $N_R$ is the unipotent radical of $R$, and $M_R$ the Levi component of $R$ 
containing $L_1$. Note that $M_R\in \CmL^G_{\rm twist}(T)$.
Since $f$ is smooth and compactly supported on $G$, so is $\Phi_{R,y}$  as a 
function on $M_R$. The distribution $\Phi\mapsto J_{L_1}^{M_R}(\gamma_u,\Phi)$ is a 
weighted unipotent orbital integral which we will study in 
\S\ref{subsec:weighted:unipotent} below.

We now turn to the definition of the weight function,
\[
 v_{R}'(z) :=\sum_{\substack{Q\in\CmF(M):\\ C_Q(\gamma_s)=R,  \ka_Q=\ka_R}} v_Q'(z),
\]
where for $Q\in\CmF(M)$, the function
$v_Q':G\longrightarrow \C$ is defined in~\cite{Ar81}*{(6.3)}. It is expressed
in~\cite{Ar86}*{p.200} as
\begin{equation}\label{def:vQ}
 v_Q'(x)=\int_{\ka_Q} \Gamma_Q^G(X, -H_Q(x))\,dX
\end{equation}
with
\[
 \Gamma_Q^G: \ka\times \ka\longrightarrow\C
\]
 given by
\[
 \Gamma_Q^G(X, Y)
:=\sum_{\substack{Q_1\in\CmF(M): \\ Q\subseteq Q_1}} (-1)^{\dim A_{Q_1}^G} \tau_{Q}^{Q_1}(X_Q)\hat{\tau}_{Q_1}((X-Y)_{Q_1}),
\]
where $X_Q$ and $(X-Y)_{Q_1}$ denote the projection of $X$ and $X-Y$ onto $\ka_Q$ and $\ka_{Q_1}$, respectively.
Here $\tau_{Q}^{Q_1}$ and $\hat{\tau}_{Q_1}$ are characteristic functions of certain
cones in $\ka_Q$ and $\ka_{Q_1}$, respectively,  defined in \cite{Ar81}*{p.11}.

For every $Q\in \CmF(M)$,
the function $v_Q'$ is left $M_Q$-invariant and right $K$-invariant.
Hence in particular, $v_R'$ is left $M$-invariant and right $K$-invariant for every $R\in \CmF^{L_2}(L_1)$.

\begin{lemma}\label{lemma:bound:ss:weight}
For every $Q\in\CmF(M)$ and $x\in G$,
\[
 \left|v_Q'(x)\right|
\ll (1+ \log|x|)^{\dim\ka_Q}.
\]
The multiplicative constant depends only on $n$.
\end{lemma}

\begin{proof}
For every $x\in G$, the function $\Gamma_Q^G(X, -H_Q(x))$ has compact support in
$X\in\ka$ by \cite{Ar81}*{Lem.2.1}. More precisely, as explained in the proof of
\cite{Ar81}*{Lem.2.1}, $X\mapsto \Gamma_Q^G(X, -H_Q(x))$ is the characteristic
function of some compact subset contained in a polytope in $\ka$  with sides
given by linear forms in $H_Q(x)$.
Hence by the formula~\eqref{def:vQ} for $v_Q'$, there exists some constant $c>0$, depending only on $n$, such that
\begin{equation}\label{eq:ss:weight:bound}
 \left|v_Q'(x)\right|
\le c (1+ \|H_Q(x)\|)^{\dim\ka_Q}.
\end{equation}
We then show that $\|H_Q(x)\|\le \log|x|$ which will conclude the proof of the lemma.
As explained in \cite{Kot05}*{\S12.1}, $H_Q(x)$ equals the image of $H_0(x)$ under
the orthogonal projection from $\ka$ onto $\ka_{Q}$ so that $\|H_Q(x)\|\le
\|H_0(x)\|$ (in \cite{Kot05}*{\S12.1} the group is assumed to be $p$-adic, but the
arguments are independent of the field).
Recall that $X(x)\in \ka$ is such that $x\in K e^{X(x)}K$. By Kostant's convexity theorem \cite{Ko73}, $H_0(k_1 e^{X(x)})$ lies inside the convex hull in $\ka$ of the Weyl group orbit of the point $X(x)$. Since $\|H_0(e^{X(x)})\|=\|X(x)\|=\log|x|$, the assertion therefore follows from \eqref{eq:ss:weight:bound}.
\end{proof}

\begin{example}
	Suppose $L_1=L_2 \subseteq M$ and $\gamma=\gamma_s$ is semisimple. The expression for the orbital integrals then simplifies:
	\begin{align*}
	J_M^G(\gamma,f)
	& = |D^G(\gamma)|^{1/2}\int_{L_2\backslash G}f(x^{-1}\gamma x) v_M'(x) ~dx\\
	& = |D^G(\gamma)|^{1/2}\int_{\kb_2^+}\int_{U_2}\int_K  f(k^{-1}u^{-1}e^{-H} \gamma e^H uk) v_M'(u) B_{\kb_2}^{M_2}(H)~dk~du~dH.
	\end{align*}
	In this case the weight function $v_M'$ is usually denoted by $v_M$ and
	equals the volume of a certain convex set in $\ka$, see \cite{Ar88a}*{p.224}.
\end{example}

\subsection{Unipotent weighted orbital integrals}\label{subsec:weighted:unipotent}
Let $L_1\subseteq M_R$, and $\gamma_u\in L_1$ be as in~\eqref{eq:def:weighted:orb}. Let 
$\CmO^{L_1}$ be the unipotent conjugacy class in $L_1$ generated by $\gamma_u$.
There is a unique unipotent conjugacy class $\CmO^{M_R}$ in $M_R$, 
induced by 
the conjugacy class $\CmO^{L_1}$ in $L_1$ (\cite{Ar88a}*{p.255}).

For a general unipotent conjugacy class, there exists~\cite{Ra72} an invariant 
measure which can be constructed with a parabolic associated to a 
suitable $\mathfrak{sl}_2$-triple. 
In our situation, $\CmO^{M_R}$ is Richardson, namely there exists a parabolic 
subgroup 
$LV\subseteq M_R$ such that $\CmO^{M_R}$ is induced by the trivial unipotent class 
$\Mid^L$ in $L$. 
Conjugating by an element of $M_R$, we can assume that $LV\in \CmF^{M_R}(T)$, and 
hence that $K^{M_R} LV = M_R$.
 The invariant measure on $\CmO^{M_R}$ admits the following 
expression
\[
 \int_{\CmO^{M_R}} f(x)~dx
=\int_{K^{M_R}}\int_{V} f(k^{-1} v k)~dv~dk,
\quad  f\in C_c^{\infty}(M_R),
\]
according to \cite{Ho74}*{Prop.5} or \cite{LM09}*{Lem.5.3}. It follows that the  
unipotent 
weighted 
orbital integrals 
$J_{L_1}^{M_R}(\gamma_u, f)$ can be 
written as
\[
 J_{L_1}^{M_R}(\gamma_u, f)=J_{L_1}^{M_R}(\CmO^{L_1}, f)
=\int_{K^{M_R}}\int_{V} f(k^{-1}vk) w_{\CmO^{L_1}}^{M_R}(v) ~dv~dk
\]
for a certain weight function $w_{\CmO^{L_1}}^{M_R}:V\longrightarrow\R$,
see~\cite{Ar88a}*{p.256}.

\subsection{Absolute convergence}\label{sec:positive:distr}
We define a modified integral $\tilde{J}_M^G(\gamma, f)$ by setting
\[
\tilde{J}_M^G(\gamma, f)
:= |D^G(\gamma_s)|^{1/2}\int_{\kb_2^+}\int_{U_2}\int_K \sum_{R\in \CmF^{L_2}(L_1)} \tilde{J}_{L_1}^{M_R}(\gamma_u, \tilde{\Phi}_{R, {e^Huk}}) B_{\kb_2}^{M_2}(H)~dk~du~dH,
\]
with
\[
 \tilde{\Phi}_{R, y}(m)
:=\delta_{R}(m)^{1/2}\int_{K^{L_2}}\int_{N_R} \left|f(y^{-1}\gamma_s k^{-1} m n k y)\right|  \sum_{\substack{Q\in\CmF(M):\\ C_Q(\gamma_s)=R,  \ka_Q=\ka_R}}  \left| v_Q'(ky)\right| ~dn~dk
\]
for $y\in G$, and
\[
 \tilde{J}_{L_1}^{M_R}(\gamma_u, \Phi)
:=\int_{K^{M_R}}\int_V \left|\Phi(k_R^{-1}vk_R)\right| \left|w_{\CmO^{L_1}}^{M_R}(v)\right| 
~dv~dk_R
\]
for $\Phi\in C_c^{\infty}(M_R)$.
It follows from Arthur's work \cite{Ar88a}*{\S7} that $\tilde{J}_M^G(\gamma, f)$ is
well-defined and finite for every $f\in C_c^{\infty}(G)$ and $\gamma \in M$. Also
$\tilde{J}_M^G(\gamma, f)$ only depends on the $M$-conjugacy class of $\gamma$.
Clearly, for every $f\in C_c^{\infty}(G)$,
\[
 \big|J_M^G(\gamma, f)\big|
\le \tilde{J}_M^G(\gamma, f)
\]
so that for our purposes it suffices to study $\tilde{J}_M^G(\gamma, f)$.

\subsection{The support of the distributions}\label{sec:support:distr}
For a semisimple element $\sigma_0\in G$, define $\Delta^-(\sigma_0)$ by:
\[
\Delta^-(\sigma_0)
:=\prod_{\substack{\alpha\in\Phi: \\ \alpha(\tilde\sigma)\neq1}} \max\left( 1,|1-\alpha(\tilde\sigma)|^{-1}\right)
\]
where $\tilde\sigma\in G_{\C}$ is a diagonal matrix conjugate to $\sigma_0$ in $G_{\C}$. Hence $\tilde\sigma$ is unique up to Weyl group conjugation and the matrix entries of $\tilde\sigma$ equal the complex eigenvalues of $\sigma_0$.
Note that if $\sigma_0$ is contained in a fixed bounded set $\cC\subset G$, then 
$\Delta^-(\sigma_0)\ll_{\mathcal C} |D^G(\sigma_0)|^{-2}$.

Let $\CmU^{L_2}$ denote the unipotent variety of $L_2=C_G(\sigma)$.

\begin{lemma}\label{lem:bound:supp}
	Let $\cC\subset G$ be a compact set.
 There exists $c \ge 2$ depending only on $n$ and $\cC$ such that the 
 following holds.
For every $H\in\kb_2^+$, $u\in U_2$, and $v\in\CmU^{L_2}$ such that $u^{-1} e^{-H} \gamma_s v e^H u\in \cC$, we have
$|(e^{-H} \gamma_s v e^H)^{-1} u^{-1} (e^{-H} \gamma_s v e^H) u |\le c$, and
\begin{align*}
\|H\|
&\le c+ \sqrt{\frac{r_2}{2}} \cdot \log(\Delta^-(\gamma_s)),				\\
|v|
&\le c \Delta^-(\gamma_s)^{\sqrt{2r_2}},
\end{align*}
where $r_2$ is determined by $L_2$ as in \S\ref{sec:twisted}.
\end{lemma}
\begin{proof}
Throughout the proof we shall denote by $a_i\ge 2$ suitable constants depending 
only on $n$ and $\cC$.
Since the elements $e^{-H} \gamma_s v e^H\in M_2$, $(e^{-H} \gamma_s v e^H)^{-1}u^{-1} 
(e^{-H} \gamma_s v e^H) u \in U_2$, and their 
product belongs to the compact set $\cC$, Lemma \ref{lemma:norm:iwasawa} gives
\begin{equation}\label{eq:2}
 |e^{-H} \gamma_s v e^H|,
~|(e^{-H} \gamma_s v e^H)^{-1} u^{-1} (e^{-H} \gamma_s v e^H)u|
\le \Cst[prf],
\end{equation}
which proves the first asserted inequality.
The  first term has Jordan decomposition $e^{-H} \gamma_s v e^H=(e^{-H} \gamma_s e^H)(e^{-H}v e^H)$. Hence by Lemma \ref{lemma:norm:jordan} and the previous inequality we get
\begin{equation}\label{eq:1}
 |e^{-H} \gamma_s e^H|,
~|e^{-H}v e^H|
\le \Cl[prf]{a} .
\end{equation}
 Conjugating $\gamma_s$ by some Weyl group element if necessary, we can assume that
\[\gamma_s=\Mdiag(\gamma_1, \ldots, \gamma_{r_1}, \gamma_{r_1+1}, \ldots, \gamma_{r_1+r_2})\] with $\gamma_1, \ldots, \gamma_{r_1}\in\R^{\times}$ and $\gamma_{r_1+j}=\left(\begin{smallmatrix} \alpha_j&\beta_j\\-\beta_j&\alpha_j\end{smallmatrix}\right)\in\GL_2(\R)$ with $\beta_j\neq0$. Accordingly,
\[
H=(0, \ldots, 0, H_1, -H_1, \ldots, H_{r_2}, -H_{r_2})
\]
(the first $r_1$ entries are $0$).
Writing $e^{-H} \gamma_s e^H=k_1e^Xk_2$, $X\in\ka$, for the Cartan decomposition, $e^X$ equals (up to permutation of the diagonal entries)
\[
\Mdiag(|\gamma_1|, \ldots, |\gamma_{r_1}|, |\det\gamma_{r_1+1}|^{1/2} A_1, |\det\gamma_{r_1+1}|^{1/2}A_1^{-1}, \ldots,
 |\det\gamma_{r_1+r_2}|^{1/2} A_{r_2}, |\det\gamma_{r_1+r_2}|^{1/2} A_{r_2}^{-1})
\]
with $A_j\ge1$ satisfying
\begin{equation}\label{eq:cosh}
\cosh(\log A_j^2) = \frac{A_j^2+A_j^{-2}}{2}
 = \tilde\alpha_j^2+ \tilde\beta_j^2 \frac{e^{4H_j}+e^{-4H_j}}{2}
 = \tilde\alpha_j^2+ \tilde\beta_j^2 \cosh(4H_j)
\end{equation}
with $\tilde\alpha_j |\det\gamma_{r_1+j}|^{1/2}=\alpha_j$ and $\tilde\beta_j |\det\gamma_{r_1+j}|^{1/2}=\beta_j$.
The first of the inequalities in \eqref{eq:1} implies that
$\left|\log|\gamma_j|\right|\le  \log \Cr{a}$ for $j=1, \ldots, r_1$, and $\left|\log|\det\gamma_{r_1+j}|^{1/2} \pm \log A_{j}\right|\le \log \Cr{a}$ for $j=1,\ldots, r_2$.
Hence
\[
\log|\det\gamma_{r_1+j}|\le 2\log \Cr{a} \;\;\text{ and }\;\;\;
 0\le\log A_{j}\le 2\log \Cr{a}.
\]
Since $|\det\gamma_s|=1$ we immediately get $\log |\det\gamma_{r_1+j}|\ge -2(n-2)\log \Cr{a}$ for every $j=1,\ldots, r_2$, and $\log|\gamma_j|\ge -2(n-1)\log \Cr{a}$ for every $j=1,\ldots, r_1$.

Using $\tilde\alpha_j^2+\tilde\beta_j^2=1$, we deduce from \eqref{eq:cosh} that
\[
 \cosh(\log A_j^2)
 = 1+ 2\tilde\beta_j^2\sinh(2H_j)^2,
\]
hence using the previous bounds on $A_j$ and $\det\gamma_{r_1+j}$ we have
\[
| \sinh(2H_j)|\le a_3 |\beta_j|^{-1}.
\]
Moreover, there exists $\xi\in\Phi^+$ with $\xi(\tilde\gamma_s)\neq1$ and
\[
2|\beta_j|=|\det\gamma_{r_1+j}|^{1/2} |1-\xi(\tilde\gamma_s)|
\ge a_4 |1-\xi(\tilde\gamma_s)|.
\]
We deduce 
\[
|\sinh(2H_j)|\le a_5 |1-\xi(\tilde\gamma_s)|^{-1} \le a_5 \Delta^-(\gamma_s), 
\]
that is, 
$|H_j|\le a_6 + \frac12 \log(\Delta^-(\gamma_s))$.
Therefore,
\[
 \|H\|^2
 =2 \sum_{j=1}^{r_2} H_j^2
\le \frac{r_2}{2} \left( 2 a_6 + \log(\Delta^-(\gamma_s))
\right)^2,
\]
which yields the second estimate.
This implies
\[
 |v|
\le |e^{H}| \cdot |e^{-H}ve^H|\cdot |e^{-H}|
\le a_7 \Delta^-(\gamma_s)^{\sqrt{2r_2}},
\]
which establishes the third estimate.
\end{proof}

\subsection{Weighted orbital integrals for unbounded test functions and non-central $\gamma_s$}\label{sec:non-central}
We now study the weighted orbital integrals for certain unbounded test functions.
More precisely,  let $F_{\eta}:=f \|X(\cdot)\|^{-\eta}:G\longrightarrow\C$ with $\eta\ge0$, and 
$f\in C_c^{\infty}(G)$.
 We show that $\tilde{J}_M^G(\gamma, F_{\eta})$ is finite if $\eta$ is small enough, which then implies that $J_M^G(\gamma,F_{\eta})$ converges absolutely. We further give an upper bound for $\tilde{J}_M^{G} (\gamma, F_{\eta})$ as $\gamma$ varies.

For the rest of \S\ref{sec:non-central} we assume that
$\gamma_s\not\in Z(G)$,
that is, $\gamma_s\neq\pm 1$.
The case that $\gamma_s = \pm 1$ will be treated in \S\ref{sec:central}.

\begin{proposition}\label{prop:bound:weight:orb:non:central}
Let $\eta\in[0,(n-1)/2)$. There exists $c_0>0$ depending on $n$ and $\eta$, such that the 
following holds.
For every $f\in C_c^{\infty}(G)$, there exists $C(f,\eta)>0$ such that for 
every $\gamma=\gamma_s\gamma_u\in M$ with $\gamma_s\neq\pm1$,
\[
\tilde{J}_M^G(\gamma, F_{\eta})
\le C(f,\eta) \Delta^-(\gamma_s)^{c_0},
\]
where $\Delta^-(\gamma_s)$ is defined in \S\ref{sec:support:distr}.
\end{proposition}

We need a few auxiliary estimates for the proof of this proposition.
Recall from \S\ref{s_convention} and \S\ref{sub_gen-Cartan} the semi-standard Levi 
subgroup $M_2$, and parabolic $P_2=M_2U_2$.
\begin{lemma}\label{lem:bound:first:int:non:central}
Fix $c,c_1 > 0$, and define
\begin{equation}\label{def:rgammas}
r(\gamma_s) := c + c_1 \log \left(\Delta^-(\gamma_s)\right).
\end{equation}
 Let $\Mun_{r(\gamma_s)}^{\kb_2^+}: \kb_2^+\longrightarrow\R$, resp.\ 
 $\Mun_{c}^{U_2}:U_2\longrightarrow\R$, be the characteristic function of the set of 
 all $H\in\kb_2^+$ with $\|H\|\le r(\gamma_s)$, resp. $u\in U_2$ with $|u|\le c$.

For every $\eta\in[0,(n-1)/2)$,
$\gamma_s\in M$ with $\gamma_s\neq\pm1$, the integral
\begin{equation}\label{eq:weighted:int:1}
 \int_{\kb_2^+}\int_{U_2} \Mun_{r(\gamma_s)}^{\kb_2^+}(Y) B_{\kb_2}^{M_2}(Y) 
 \Mun_{c}^{U_2}(u)~ \max\left(\kL(e^{-Y}\gamma_s e^Y),\kL(u)\right)^{-\eta}\,du\,dY
\end{equation}
converges. Moreover, there exists $c'\ge 1$ depending only on $c$, $c_1$, $n$, and 
$\eta$ (and not on $\gamma_s$), such that it is bounded from above by
$
c' \Delta^-(\gamma_s)^{c_2}
$,
for some constant $c_2\ge 1$ depending only on $c_1$, $n$, and~$\eta$.
\end{lemma}

\begin{proof}
 Suppose first that $U_2$ is non-trivial, i.e., that $M_2\neq G$. 
 Then~\eqref{eq:weighted:int:1} is bounded by
\[
 \int_{\kb_2^+}\Mun_{r(\gamma_s)}^{\kb_2^+}(Y) B_{\kb_2}^{M_2}(Y)\,dY
\int_{U_2}  \Mun_{c}^{U_2}(u)~ \kL(u)^{-\eta}\,du.
\]
The first integral is obviously bounded by an exponential function in $r(\gamma_s)$, which can be chosen such that it only depends on $n$.

For the second integral note that if we write $u=I_n + (u_{ij})_{i<j}$, then
\[
 \kL(u) = \log\left(1+ \frac{1}{n}\sum_{1\le i<j\le n} u_{ij}^2\right),
\]
and
\[
 \int_{U_2}  \Mun_{c}^{U_2}(u)~ \kL(u)^{-\eta}~du
\le \vol(B^{d_2}(1)) \int_{0}^{\sqrt{n}c}\left(\log \left(1+r^2/n\right) \right)^{-\eta} 
r^{d_2-1} ~dr
\]
where $d_2=\dim U_2\ge n-1$, and $\vol(B^{d_2}(1))$ denotes the volume of the ball 
$B^{d_2}(1)$ 
of radius $1$ around $0$ in $\R^d$.
The last integral is finite if $\eta\in[0,d_2/2)$.
 To be more precise, it is bounded by a constant depending on $n$, $\eta$, and 
 $c$.

Now if $U_2$ is trivial, i.e., if $M_2=G$, then $\gamma_s$ has to be elliptic. 
Since  $|\det \gamma_s|=1$ but $\gamma_s\not\in Z(G)$, we have $\gamma_s=\Mdiag(\delta, \ldots, 
\delta)$ with  $\delta=\left(\begin{smallmatrix} \alpha&\beta\\-\beta 
&\alpha\end{smallmatrix}\right)$, $\alpha^2+\beta^2=1$, $\beta\neq0$. Hence $Y=(Y_1, -Y_1, 
\ldots, Y_{n/2}, -Y_{n/2})\in\kb_2^+$, $Y_1>Y_2>\ldots>Y_{n/2}>0$, and
\[
\kL(e^{-Y}\gamma_s e^Y)
=\log\left(1+\frac{4\beta^2}{n}\sum_{i=1}^{n/2} \sinh^2(2Y_i)\right).
\]
Hence~\eqref{eq:weighted:int:1} equals
\[
 \int_{\kb_2^+}\Mun_{r(\gamma_s)}^{\kb_2^+}(Y) B_{\kb_2}^{M_2}(Y)~\kL(e^{-Y}\gamma_s e^Y)^{-\eta} ~dY
=\int_{\kb_2^+}\Mun_{r(\gamma_s)}^{\kb_2^+}(Y) \frac{\prod_{\beta\in\Phi_{\kb_2}^+}|\sinh\beta(Y)|^{m_{\beta}^+}|\cosh\beta(Y)|^{m_{\beta}^-}}{(\log(1+\frac{4\beta^2}{n}\sum_{i=1}^{n/2} \sinh^2(2Y_i)))^{\eta}} ~dY.
\]
Note that for each $i\in\{1, \ldots, n/2\}$ there is $\beta_i\in \Phi_{\kb_2}^+$ with $\beta_i(Y)=2Y_i$ and $m_{\beta_i}^+=1$.
Hence for $Y$ close to $0$, the integrand in the last integral is bounded by
\[
 \frac{\prod_{i=1}^{n/2} Y_i}{\left(Y_1^2+\ldots + Y_{n/2}^2\right)^\eta}
\]
which is integrable in a neighborhood of $0$ if $\eta\in[0,n/2)$. The 
integral~\eqref{eq:weighted:int:1} can 
therefore be bounded by an exponential function in $r(\gamma_s)$ which can be 
chosen to depend only on $c$, $c_1$, $n$ and $\eta$.
 \end{proof}

\begin{lemma}\label{lem:bound:second}
For $s\ge 1$ define $\Xi_{s}^R:M_R\longrightarrow\R_{\ge 0}$  by
\begin{equation}\label{def:Xis}
 \Xi_{s}^R(m):=\int_{N_R} \Mun_{s}(mn)~dn,
\end{equation}
where $\Mun_s:G\longrightarrow\{0,1\}$ is the characteristic function of all $g\in G$ with 
$|g|\le s$.
 There exist constants $c, c_3>0$ depending only on $n$, such that for every $R\in 
 \CmF^{L_2}(L_1)$,  
 $s\ge 1$, and unipotent $\gamma_u\in L_1$, we have
\[
 \tilde{J}_{L_1}^{M_R}(\gamma_u, \Xi_{s}^R)
\le c s^{c_3}.
\]
\end{lemma}

\begin{proof}
 The set $\CmF^{L_2}(L_1)$ is finite so that it suffices to 
 consider a fixed parabolic $R\in\CmF^{L_2}(L_1)$. Moreover, 
 $\tilde{J}_{L_1}^{M_R}(\gamma_u,\cdot)$ 
 only depends on the $L_1$-conjugacy class of $\gamma_u$ of which there are only 
 finitely many so that it suffices to treat the element $\gamma_u$ as fixed.
 Again, during the proof we denote by $a_i\ge1$ suitable constants depending only 
 on $n$ and the weight function (of which there are of course only finitely 
 many).

Let $\CmO^{L_1}$ be the unipotent conjugacy class in $L_1$ generated by $\gamma_u$, and let $\CmO^{M_R}$ be the unipotent class induced from $\CmO^{L_1}$ to $M_R$. 
Let $LV\in \CmF^{M_R}(T)$ be a Richardson parabolic for $\CmO^{M_R}$.
Then
\[
 \tilde{J}_{L_1}^{M_R}(\gamma_u, \Xi_{s}^R)
=\int_{K^{M_R}}\int_{V} \Xi_s^R(k^{-1} vk) \left|w_{\CmO^{L_1}}^{M_R}(v)\right|~dv~dk
= \int_{V} \Xi_s^R(v) \left|w_{\CmO^{L_1}}^{M_R}(v)\right|~dv.
\]
 By construction of the weight function~\cite{Ar88a}*{\S6}, it suffices to bound 
 integrals of the form
\[
 \int_{V} \Xi_s^R(v) \left|\log |p(v)|\right|^k ~dv
\]
for finitely many integers $k\ge 0$ and polynomial functions $p: V\longrightarrow \R$.
(In \cite{Ar88a}*{\S6}, the polynomial $p$ is vector-valued, however the reduction 
to the scalar case is immediate, see e.g., the proof of \cite{Ar88a}*{Lem.7.1}).

We may assume without loss of generality that $LV$ equals the intersection of 
$M_R$ with a standard parabolic subgroup in $G$.
Let $v= I_n +(v_{ij})_{i<j}\in V$. If $\Xi_s^R(v)\neq0$, then $\|X(v)\|\le 
\Cl[prf]{d}+\Cl[prf]{d} \log s$ by Lemma~\ref{lemma:norm:iwasawa}.
Hence the second inequality of Lemma~\ref{lem:compare:L:X} implies that
$1+ \frac1n\sum_{i<j} v_{ij}^2
\le \Cl[prf]{d} s^2$. Thus $|v_{ij}|\le \Cst[prf] s$ for every $i<j$. Hence for every 
$v\in V$ we have
\[
 \Xi_s^R(v)\neq 0
\Rightarrow
|p(v)|\le \Cl[prf]{d1} s^{\Cl[prf]{d2}}.
\]
Let $\cC_s\subset V$ be the compact subset of all $v=(v_{ij})_{i< j}\in V$ with 
$\sum_{i<j}v_{ij}^2\le n \Cr{d} s^2$.
Note that for any $m\in M_R$ we have
\[
 \left|\Xi_s^R(m)\right|
\ll \vol\left(\{x\in\R\mid x^2\le s^2\}^{\dim N_R}\right)
\ll s^{\dim N_R}.
\]
Hence we are left to estimate
\[
 \int_{\cC_s} \left|\log |p(u)|\right|^k \, du.
\]
We identify $\cC_s$ with the set of all $x\in\R^{\dim V}$ with $\|x\|^2\le n\Cr{d} s^2$, 
where $\|x\|$ denotes the usual euclidean norm. 
Write $\tilde x=x/s$ and let $\tilde \cC:= \{x\in \R^{\dim V}\mid \|x\|^2\le n\Cr{d}\}$ which
is independent of $s$. We can write $p(x)=P^s(\tilde x)$ for some polynomial $P^s$
with $\deg P^s=\deg p$ whose coefficients depend on $s$. More precisely, all
coefficients of $P^s$ are bounded by an absolute multiple of $s^{\deg p}$. Hence
there exists $c>0$ such that for all $\tilde x \in \tilde \cC$ we have $|P^s(\tilde x)|\le c
s^{\deg p}$. Applying \cite{Art88a}*{Lem.7.1} with $\varepsilon=s^{\dim V}$, we find $t>0$ 
such that
\[
\int_{x\in \cC_s\subset \R^{\dim V}} \left|\log |p(x)|\right|^k \,dx
= s^{\dim V}\int_{x\in \tilde \cC} \left|\log |P^s(\tilde x)|\right|^k \,d\tilde x
\ll_{n, p} s^{(t+1)\dim V}.
\]
(See also the proof of Proposition~\ref{prop:bound:weight:orb:non:central} below 
for a 
similar estimate). This completes the proof of the lemma.
\end{proof}

\begin{proof}[Proof of Proposition \ref{prop:bound:weight:orb:non:central}]
Let  $0\le \eta<(n-1)/2$ and $F_\eta = f\|X(\cdot)\|^{-\eta}$. Without loss of generality, 
we assume that $f(k_1gk_2)=f(g)$ for every $k_1,k_2\in K$ and $g\in G$ (in fact we 
shall only use $f(k^{-1}gk)=f(g)$). Recall the 
definition of 
$\tilde{J}_M^G(\gamma, F_{\eta})$ from 
the beginning of \S\ref{sec:positive:distr}, namely $\tilde{J}_M^G(\gamma, F_{\eta})$ is the 
integral 
over $H\in\kb^+_2$ and  $u\in U_2$ and the sum over $R\in\cF^{L_2}(L_1)$ of 
$\tilde{J}^{M_R}_{L_1} (\gamma_u, \tilde{\Phi}_{R, e^Hu}) B^{M_2}_{\kb_2}(H)$, where 
\[
 \tilde{\Phi}_{R, e^Hu} (m) 
 = \delta_R(m)^{1/2} \int_{K^{L_2}}\int_{N_R} |F_{\eta}(u^{-1} e^{-H} \gamma_s k^{-1} m n 
 k e^H u)|\sum_{\substack{Q\in\cF(M)\\ C_Q(\gamma_s) = R, \ka_Q= \ka_R}} |v_Q'(ke^Hu)| \, dn\, 
 dk,
\]
for $m\in M_R$.

The integral $\tilde{J}^{M_R}_{L_1} (\gamma_u, \Phi)$ is defined by 
integrating $|\Phi|$
over $k_R^{-1} v' k_R$, for $v'\in V$ and $k_R\in K^{M_R}$, hence is supported on 
$\mathcal U^{M_R}$.
We may therefore assume that $m\in \mathcal U^{M_R}$.
Then the element $v=k^{-1}m n k$ belongs to $\mathcal U^{L_2}$. We may also assume 
that 
$u^{-1}e^{-H}\gamma_s v e^H 
u$ belongs to the support of $f$.
We then deduce from Lemma~\ref{lem:bound:supp} the following three estimates: 
$\|H\|\le r(\gamma_s)$, where 
$r(\gamma_s)$ is 
defined by~\eqref{def:rgammas} with suitable $c\ge 1$ depending on the support of $f$ 
and $c_1=\sqrt{\frac{r_2}{2}}$, 
also
$|v|\le s$, where $s=
c 
\Delta^-(\gamma_s)^{\sqrt{2r_2}}$, and finally
\begin{equation}\label{pf:m1-conj-u}
 |m_1^{-1} u^{-1} m_1 u| \le c,
\end{equation}
where $m_1=e^{-H}\gamma_s v e^H\in M_2$. Overall we obtain a bound
\[
 \tilde{\Phi}_{R, e^Hu} (m) \le 
\Mun_{r(\gamma_s)}^{\kb_2^+}(H)
 \int_{K^{L_2}}\int_{N_R}
\Mun_{s}(mn) 
\Mun_{c}^{U_2}(m_1^{-1} u^{-1} m_1 u)
\|X(u^{-1} m_1 u)\|^{-\eta}
\]
\[
\sum_{\substack{Q\in\cF(M)\\ C_Q(\gamma_s) = R, \ka_Q= \ka_R}} |v_Q'(ke^Hu)| \, 
dn\, 
 dk.
\]

We next prove that $|u|$ is bounded by some power of $\Delta^-(\gamma_s)$. 
Lemma~\ref{lemma:bound:ss:weight} implies that the weight function 
$|v_Q'(ke^Hu)|$ is
bounded above by a constant times 
$(1+ \|H\| + \log |u|)^n$, which will therefore be bounded by some power of $\log 
\Delta^-(\gamma_s)$.
Since $\|H\|\le r(\gamma_s)$, it suffices to bound $u_1=e^{H}ue^{-H}$ by some power of 
$\Delta^-(\gamma_s)$. 
Conjugating 
conjugating further \eqref{pf:m1-conj-u}  by $e^{H}$, we have
\[
 |(\gamma_s v)^{-1} u_1^{-1} (\gamma_s v)u_1| \le c \Delta^-(\gamma_s)^{c_4}.
\]
Furthermore, in bounding $u_1$, it suffices to work in the complexified $G_\C$. After 
conjugation by some 
element in $M_{2,\C}$ of size bounded by a polynomial in $\Delta^-(\gamma_s)$, we can 
assume $\gamma_s v$ is of the form $tv_1$ with $t$ 
diagonal 
with centralizer in $G_\C$ contained in $M_{2,\C}$, having the eigenvalues of 
$\gamma_s$ as diagonal entries, and  $v_1\in M_{2,\C}\cap U_{0,\C}$. We denote the 
conjugated $u_1$  by $u_2$ which is an element in $U_{2,\C}$, and we obtain that 
$|(tv_1)^{-1}u_2^{-1} (tv_1) u_2|$ is bounded by a power of $\Delta^-(\gamma_s)$.
Since $|v|\le s$, we have that $|v_1|$ is bounded by a power of $\Delta^-(\gamma_s)$.
 Let 
$\Phi^+$ 
denote the set of positive roots corresponding to $(T_{0,\C}, U_{0,\C})$. We have a strict 
partial 
order 
on $\Phi^+$ given by $\beta\prec \alpha$ if and only if $\alpha$ equals the  
sum 
of $\beta$ and other positive roots. 
For each $\alpha\in \Phi^+$, we write $u_\alpha\in\C$ for the value
of the 
matrix entry of $u_2$ at $\alpha$ (since $G_{\C}=
\GL_n(\C)^1$ we make $\alpha$ correspond to matrix entries), which we refer to as the 
$\alpha$-coordinate of $u_2$.

We now prove, inductively on $\alpha$ in this partial order, that $u_\alpha$ is 
bounded 
by some power of 
$\Delta^-(\gamma_s)=\Delta^-(t)$. 
Write $v_1= \MId+ X_1$ with $\MId$ the $n\times n$-identity matrix and 
$X_1$ 
a nilpotent matrix. Similarly, write $v_1^{-1} = \MId + X_2$. The matrix 
entries of both $X_1$ and $X_2$ are bounded by a power of $\Delta^-(t)$.
Then
\begin{equation}\label{eq:almost:weyl}
 - t^{-1}u_2^{-1}tu_2 = (tv_1)^{-1} u_2^{-1} (tv_1)u_2 + X_2 t^{-1}u_2^{-1} 
 tu_2 + t^{-1}u_2^{-1} t X_1 u_2 
 + 
 X_2t^{-1} u_2^{-1} tX_1u_2.
\end{equation}
For each $\alpha\in\Phi^+$, the $\alpha$-coordinate of $t^{-1}u_2^{-1}tu_2$ equals the 
sum 
of  
$(1-\alpha(t))u_\alpha$ and a polynomial in $\beta(t)u_\beta$, $u_\beta$ with $\beta\prec \alpha$. 
On the RHS, the $\alpha$-coordinate of $(tv_1)^{-1} u_2^{-1} (tv_1)u_2$ is bounded 
by a power of $\Delta^-(t)$, 
and the $\alpha$-coordinates of the remaining terms are polynomials in $\beta(t)u_\beta$ 
and $u_\beta$ with $\beta\prec \alpha$, and suitable matrix entries of $X_1$ and $X_2$. 
Hence~\eqref{eq:almost:weyl} implies that $(1-\alpha(t))u_\alpha$ is equal to a  
polynomial in $\beta(t)u_\beta$, $u_\beta$ with $\beta\prec \alpha$ with coefficients 
bounded by a power of $\Delta^-(t)$. It thus 
follows inductively that all $u_\alpha$ for $\alpha\in\Phi^+$ are bounded by some power 
of 
$\Delta^-(t)$.

We next consider $\int_{U_2} \tilde \Phi_{R,e^Hu}(m)du$. 
We deduce from the previous inequality that 
$\tilde{J}_M^G(\gamma, F_{\eta})$ is 
bounded above by the
integral 
over $H\in\kb^+_2$ and the sum over $R\in\cF^{L_2}(L_1)$ of the
 product of
$\tilde{J}^{M_R}_{L_1} (\gamma_u, \tilde{\Psi}_{R, e^H}) \Mun_{r(\gamma_s)}^{\kb_2^+}(H) 
B^{M_2}_{\kb_2}(H)$, where 
\[
 \tilde{\Psi}_{R, e^H} (m) := 
 \int_{K^{L_2}}\int_{N_R}
\Mun_{s}(mn) 
\int_{U_2}
\Mun_{c}^{U_2}(m_1^{-1} u^{-1} m_1 u)
\|X(u^{-1} m_1 u)\|^{-\eta} \,du \, dn\, 
 dk.
\]
Note that we have switched the order of integration, which is justified because 
the integrand is non-negative.

The same argument based on~\eqref{eq:almost:weyl} also implies 
that the inverse of the Jacobian determinant of the 
diffeomorphism $u_2 
\mapsto (tv_1)^{-1} u_2^{-1} (tv_1) u_2$ from $U_{2,\C}$ onto itself is bounded above by 
$\Delta^-(t)$.
Hence the inverse of the Jacobian determinant of the diffeomorphism $u\mapsto 
m_1^{-1}u^{-1} m_1 
u$ from 
$U_2$ 
onto itself is bounded above by a power of $\Delta^-(t)=\Delta^-(\gamma_s)$.

We want to 
replace $\|X(\cdot)\|^{-\eta}$ by a suitable power of $\kL(\cdot)$ in the above 
integral. 
We have $m_1=e^{-H}\gamma_s v e^H\in M_2$, hence
by Lemmas~\ref{lem:compare:L:X} and \ref{lem:inequ:L:Iwasawa}, we find
 \[
2 \|X(u^{-1} m_1 u)\| \ge 
  \kL(u^{-1} m_1 u)
  \ge \max\left(\kL (m_1), \kL (\MId + m_1 (m_1^{-1} u^{-1} m_1 u-\MId))\right),
 \]
since $m_1^{-1}u^{-1}m_1 u \in U_2$.
Using Lemma~\ref{lem:inequ:L:Iwasawa} again, we 
have $\kL(m_1)\ge \kL(e^{-H}\gamma_s e^H)$, since $m_1=e^{-H}\gamma_s e^H e^{-H}ve^H$.

The matrix entries of both $m_1$ and $m_1^{-1}$ are bounded 
by $\Delta^-(\gamma_s)^{c_5}$ for some $c_5>0$.
The map $u\mapsto \MId + m_1 (m_1^{-1} u^{-1} m_1 u-\MId)$ is a diffeomorphism from 
$U_2$ 
onto $U_2$, since it is the composition of the 
diffeomorphism $u\mapsto m_1^{-1} u^{-1} m_1 u$ and the diffeomorphism $u\mapsto \MId + 
m_1(u-\MId)$, and we make this change of variable.
The inverse of the Jacobian determinant is bounded by 
$\Delta^-(\gamma_s)^{c_6}$ for some $c_6>0$, since it is the inverse product of the Jacobian 
determinants of $u\mapsto m_1^{-1} u^{-1} m_1 u$ and $u\mapsto \MId + 
m_1(u-\MId)$.
This yields
\begin{align*}
\tilde \Psi_{R, e^H} (m)
& \le 
\Delta^-(\gamma_s)^{c_6}
\int_{K^{L^2}}
\int_{U_2}
\int_{N_R}
\Mun_{s}(mn) 
\Mun_{c}^{U_2}(\Mid + m_1^{-1}(u-\MId))
\max\left(\kL(e^{-H}\gamma_s e^H),\kL(u)\right)^{-\eta}
\,dn\,du\,dk
\\ & \le 
\Delta^-(\gamma_s)^{c_6}
\cdot
\Xi^R_{s}(m)
\int_{U_2} 
\Mun_{c \Delta^-(\gamma_s)^{c_5}}^{U_2}(u)
\max\left(\kL(e^{-H}\gamma_s e^H),\kL(u)\right)^{-\eta}
\,du,
\end{align*}
where $\Xi_{s}^R$ is as in  \eqref{def:Xis}.

We conclude that 
$\tilde{J}_M^G(\gamma, F_{\eta})$ is 
bounded above by 
\[
\Delta^-(\gamma_s)^{c_6} 
\sum_{R\in\CmF^{L_2}(L_1)} \tilde J_{L_1}^{M_R}(\gamma_u,\Xi_{s}^R)
\]
times
\[
 \int_{\kb_2^+}\int_{U_2} \Mun_{r(\gamma_s)}^{\kb_2^+}(H) B_{\kb_2}^{M_2}(H) 
 \Mun_{c\Delta^-(\gamma_s)^{c_5}}^{U_2}(u) ~\max\left(\kL(e^{-H}\gamma_s 
 e^H),\kL(u)\right)^{-\eta}~du~dH.
\]
These two terms are bounded above by
Lemmas~\ref{lem:bound:second} and~\ref{lem:bound:first:int:non:central}, 
respectively, which concludes the proof of 
Proposition~\ref{prop:bound:weight:orb:non:central}.
\end{proof}

\subsection{Weighted orbital integrals for unbounded test functions and central $\gamma_s$}\label{sec:central}
We now turn to the remaining case that $\gamma_s\in\{\pm1\}$, and take $F_{\eta}=f \|X(\cdot)\|^{-\eta}$ as before.

\begin{proposition}\label{prop:weight:orb:int:central}
There exists $\eta>0$ such that for every $M$, and $\gamma_u \in M$,  with $(M,\gamma_u)\neq(G,1)$,
and every $f\in \cC_c^\infty(G)$,
 the integral $\tilde{J}_M^G(\gamma_u,F_{\eta})$ is finite.
\end{proposition}

\begin{proof}
Let $\CmO^M$ be the unipotent conjugacy class generated by $\gamma_u$, and let  $\CmO^G$ be the class induced from $\CmO^M$ to $G$. 
Let $LV$ be a Richardson parabolic for $\CmO^G$. We can assume that $LV$ is a 
standard parabolic subgroup of $G$.
Without loss of generality we can assume that $f$ is $K$-conjugation invariant.
Then we can write
\[
 \tilde{J}_M^G(\gamma_u, F_{\eta})=\int_V |f(v)| \|X(v)\|^{-\eta} \left|w_{\CmO^M}^G(v)\right|~dv.
\]
By construction of the weight function (see~\cite{Ar88a}*{\S5 and \S7}) it
suffices to consider finitely many integrals of the form
\[
 \int_V \Mun(v) \|X(v)\|^{-\eta} \left| \log|p(v)| \right|^k~dv,
\]
where $p:V\longrightarrow\R$ is a polynomial function, $k\ge0$ a non-negative integer, 
and $\Mun:G\longrightarrow \{0,1\}$ is the characteristic function of a compact subset
$\cC$ in $G$ (depending on the support of $f$).

Suppose first that $p$ does not have a constant term.
For $\varepsilon>0$ sufficiently small let $\Gamma(\varepsilon)=\{v\in V \cap \cC\mid 
|p(v)|<\varepsilon\}$,  and consider the dyadic decomposition
\[
 \Gamma(m, \varepsilon)
=\Gamma(2^{-m}\varepsilon) - \Gamma(2^{-(m+1)}\varepsilon), \;\;m\in\Z_{\ge0}.
\]
By Lemma~\ref{lem:eps} below we can find a constant $c>0$ such that if $v\in V\cap 
\cC$ is such that if $\|X(v)\|<c\varepsilon^2$, then $v\in \Gamma(\varepsilon)$. In 
particular, $\|X(v)\|^{-\eta}$ and $\left|\log|p(v)|\right|$ are both bounded away from 
$0$ on $(V\cap \cC)\backslash \Gamma(\varepsilon)$. Since $\vol(V\cap \cC)<\infty$, the part of the 
above integral over $(V\cap \cC)\backslash\Gamma(\varepsilon)$ is finite.
For the part of the integral corresponding to $\Gamma(\varepsilon)$ we follow 
\cite{Ar88a}*{pp.259--261} and get
\begin{align*}
 \int_{\Gamma(\varepsilon)} \|X(v)\|^{-\eta} |\log|p(v)||^k~dv
& =\sum_{m\ge0}\int_{\Gamma(m,\varepsilon)} \|X(v)\|^{-\eta} |\log|p(v)||^k~dv\\
& \ll_\cC  \sum_{m\ge0}\int_{\Gamma(m,\varepsilon)} (2^{-m}\varepsilon)^{-2\eta} 
|\log(2^{-(m+1)}\varepsilon)|^k~dv\\
& \le \sum_{m\ge0}\vol(\Gamma(m,\varepsilon)) (2^{-m}\varepsilon)^{-2\eta} |\log(2^{-(m+1)}\varepsilon)|^k,
\end{align*}
where the first inequality follows from Lemma \ref{lem:eps}
below. Now by \cite{Ar88a}*{Lem.7.1}, there exist constants $B, t>0$ such that
$\vol(\Gamma(\varepsilon))\le B\varepsilon^t$ for every $\varepsilon <1$. Hence
\[
 \int_{\Gamma(\varepsilon)} \|X(v)\|^{-\eta} |\log|p(v)||^k~dv
\le \Cst[prf] \varepsilon^{t-2\eta} \sum_{m\ge0}2^{m(2\eta-t)} ((m+1)\log2 +\log\varepsilon^{-1})^k ,
\]
and this last sum is finite if $\eta<t/2$. Hence the assertion follows for $p$ without a constant term.

If $p$ has a constant term, and if $\varepsilon$ is sufficiently small, then 
$\|X(v)\|\le \varepsilon$ implies that $c_1\le |p(v)|\le c_2$ for some constants 
$c_1,c_2$ depending on $p$ and $\varepsilon$.  Define $\Gamma'(\varepsilon)=\{v\in V \cap \cC\mid 
\|X(v)\|<\varepsilon\}$, and define $\Gamma'(m, \varepsilon)$ for $m\in\Z_{\ge0}$ similarly as 
before. Then, proceeding as above
\begin{align*}
 \int_{\Gamma'(\varepsilon)} \|X(v)\|^{-\eta} |\log|p(v)||^k~dv
& =\sum_{m\ge0}\int_{\Gamma'(m,\varepsilon)} \|X(v)\|^{-\eta} |\log|p(v)||^k~dv\\
& \ll_{\cC,p,k} \sum_{m\ge0}\int_{\Gamma'(m,\varepsilon)} (2^{-m+1}\varepsilon)^{-\eta} ~dv\\
& \le  \sum_{m\ge0}\vol(\Gamma'(m,\varepsilon)) (2^{-m+1}\varepsilon)^{-\eta}.
\end{align*}
Using the bound $\vol(\Gamma'(\varepsilon))\le B'\varepsilon^{t'}$ for some $B', t'>0$ as before, this sum again converges if we choose $\eta<t'$.
Further let $\Gamma(\varepsilon)=\{v\in V\cap \cC\mid |p(v)|<\varepsilon\}$. The part of the 
integral corresponding to $\Gamma(\varepsilon)$ is bounded similarly as before. On 
$(V\cap \cC)\backslash (\Gamma(\varepsilon)\cup\Gamma'(\varepsilon))$ the functions 
$\|X(\cdot)\|^{-\eta}$ and $|\log|p(v)||$ are bounded away from $0$ so that also the 
integral over $(V\cap \cC)\backslash (\Gamma(\varepsilon)\cup\Gamma'(\varepsilon))$ converges. This 
finishes the assertion for the case that $p$ has a constant term.
\end{proof}

\begin{lemma}\label{lem:eps}
 Let $V$ be the unipotent radical of a semi-standard parabolic subgroup of $G$, 
 and {$p:V\longrightarrow\R$} a polynomial function without constant term. 
 Then there exists a constant $a=a(p)>0$ such that for every $v\in V$ and every 
 $0 \le \varepsilon \le 1$, we have
\[
\|X(v)\|
\le \varepsilon
\implies
|p(v)|\le a\varepsilon^{1/2}.
\]
\end{lemma}
\begin{proof}
 Without loss of generality we can assume that $V$ is the unipotent radical of a 
 standard parabolic subgroup. Let $v=I_n+(v_{ij})_{i<j}\in V$.
By Lemma \ref{lem:compare:L:X}, we have
\[
 \log(1+ \frac{1}{n}\sum_{i<j} v_{ij}^2) =\kL(v) \le 2 \|X(v)\| \le 2 \varepsilon \le  
 \log(1+7\varepsilon).
\]
Hence $
 v_{ij}^2\le 7n \varepsilon$
for every $1\le i<j\le n$, and the assertion follows.
\end{proof}

\section{Spherical functions and archimedean test functions}\label{sec:spherfcts}
The purpose of this section is first to define a certain family of test functions which will be used in Part~\ref{part2} of this paper, and secondly, to prove an upper bound for the weighted orbital integrals $\tilde{J}_M^G(\gamma,\cdot )$ over these test functions in Proposition~\ref{p:J-sph}.
The key step is a new uniform bound for spherical functions in
Proposition~\ref{prop:spher:fct:bound}.

\subsection{Wave packets and a family of test functions}\label{sub:wave-test}
We want to study cuspidal automorphic representations which have trivial $\TO(n)$-type,
that is, which have a $\TO(n)$-fixed vector, or, if $n$ is even, which have $\TO(n)$-type
$\chi_-$. To isolate such representations in the Arthur-Selberg trace formula we need to
use test functions of a specific type. More precisely, to isolate representations of
trivial $\TO(n)$-type, the archimedean part of the test function has to be
bi-$\TO(n)$-invariant. Such functions can be described by the spherical Paley-Wiener
theorem,
see~\cite{Ga71}*{Cor.3.7}. To isolate representations of $\TO(n)$-type $\chi_-$, we
essentially multiply the aforementioned test functions by $\chi_-$.

Let $\lambda\in\ka^*_{\C}$. The zonal spherical function
$\phi_{\lambda}: G\longrightarrow\C$ of spectral parameter $\lambda$ can be defined by
\begin{equation}\label{def:spherical}
 \phi_{\lambda}(g)=\int_{K}e^{\langle\lambda+\rho, H_0(kg)\rangle} ~dk,
\end{equation}
where $\rho$ is the half sum of all positive roots $\Phi^+$, and we recall that
$K=\TO(n)$ and $K^\circ=\SO(n)$.

The two families of test functions $f_{\pm}^{\mu}$, $\mu\in i\ka^*$, are built from the zonal spherical functions by applying the inverse spherical transform to an arbitrary fixed function $h\in C_c^{\infty}(\ka)^W$:
\begin{equation}\label{eq:def:test:function}
 f^{\mu}_+(g)
=\frac{1}{|W|}\int_{i\ka^*} \hat{h}(\lambda-\mu) \phi_{\lambda}(g)
\left|\frac{\Tmc(\rho)}{\Tmc(\lambda)}\right|^{2} ~d\lambda,
\end{equation}
where $\Tmc(\lambda)$ denotes the Harish-Chandra $\Tmc$-function which in our case is given by
\begin{equation}\label{def:c}
\Tmc(\lambda)
= \pi^{|\Phi^+|/2}
\prod_{\alpha\in\Phi^+}\frac{\Gamma(\alpha(\lambda)/2)}{\Gamma((\alpha(\lambda)+1)/2)}.
\end{equation}
We then put $f^{\mu}_-(g)=\chi_-(k) f^{\mu}_+(p)$ where $g=pk$ is the Iwasawa decomposition with $\det p>0$. Note that $f^\mu_- = f^\mu_+$ if $n$ is odd.
Although the functions $f^{\mu}_{\pm}$ depend on $h$, we suppress $h$ from the notation
since we shall fix one particular $h$.  All multiplicative constants occurring in the
sequel will depend on this choice.

The function $f^{\mu}_{\pm}$ satisfies the following properties (see~\cite{Ga71},~\cite{book:helg:geometricanalysis}*{Ch.~IV, \S7}):
\begin{itemize}
\item $f^{\mu}_{\pm}\in C^{\infty}_c(G)$; more precisely, if $h$ is supported in the ball $\{X\in \ka\mid \|X\|\le R\}$ for some $R>0$, then $f^{\mu}_{\pm}$ is supported in the compact set of all $g\in G$ with $\|X(g)\|\le R$;
\item $f^{\mu}_+$ is bi-$K$-invariant;
\item $f^{\mu}_-$ is bi-$K^{\circ}$-invariant and satisfies $f^{\mu}_-(k_1gk_2)=\chi_-(k_1k_2)f^{\mu}(g)$ for all $g\in G$ and $k_1, k_2\in K$.
\end{itemize}
Note further that $\tilde{J}_M^G(\gamma, f^{\mu}_-)=\tilde{J}_M^G(\gamma,f^{\mu}_+)$.

\subsection{An upper bound for the spherical functions}
Our strategy to bound the orbital integrals $\tilde{J}_M^G(\gamma,f^{\mu}_{\pm})$ below is to 
unfold the integrals defining $\tilde{J}_M^G(\gamma,\cdot)$ and  $f^{\mu}_\pm$. We then need a 
good upper bound on the spherical function $\phi_{\lambda}$ at certain points determined by 
the support of the functional $\tilde{J}_M^G(\gamma,\cdot)$ and the support of $f^\mu_\pm$.

Let $\cU(\Fmk)$ be the universal enveloping algebra of the Lie algebra $\Fmk$ of $K$. Any $D\in \cU(\Fmk)$, induces a right differential operator on smooth functions on $K$; in particular we can form $H_0(kg;D)\in \Fma$, that is, $H_0(kg;D)$ denotes the value at $k$ of the function we obtain by applying $D$ to the function $k\mapsto H_0(kg)$. Let $F^\bullet$ be the  filtration by degree on $\cU(\Fmk)$ and let $\cU_0(\Fmk)$ be the subspace of operators without constant term in the splitting $\cU(\Fmk)=\R \oplus \cU_0(\Fmk)$.
\begin{lemma}\label{l:clean}
For any $k\in K$ and  $g\not \in K$, the linear map $F^2\cU_0(\Fmk)\to \Fma$ induced by $D\mapsto H_0(kg;D)$ is surjective.
\end{lemma}
\begin{proof}
This follows from the explicit formula of $H_0(kg;D)$ in~\cite{DKV83}*{Lem.5.1},
and
the determination of the Hessian at a critical point~\cite{DKV83}*{Cor.6.4}. In
fact a
stronger statement holds~\cite{DKV83}*{Lem.5.9}, namely that for any $\lambda\in
\Fma^*$, the function $k\mapsto \langle \lambda,H_0(ak)\rangle$ has clean critical set in
the sense of Bott, i.e., the Hessian is non-degenerate transversely to the 
critical
manifold.
\end{proof}
We establish the following uniform pointwise bound for the zonal spherical function.
\begin{proposition}\label{prop:spher:fct:bound}
Let $\cC\subset G$ be a compact subset and $A  \ge  0$. Then
\begin{equation}\label{eq:decay:spherical}
 \left|\phi_{\lambda}(g)\right|\ll_{A,\cC} (1+\|\mathrm{Im}\,
 \lambda\|\|X(g)\|)^{-\frac12}
\end{equation}
for all $g\in \cC$ and $\lambda\in \Fma^*_\C$ with $|\mathrm{Re}\, \lambda| \le A$.
\end{proposition}

We shall only need this proposition for $\lambda \in  i \mathfrak{a}^*$, i.e., for 
$\mathrm{Re}\, \lambda=0$. In that case, $|\phi_\lambda(g)|\le 1$ for all $g\in G$
with equality obtained for $g\in K$.
So the result is qualitatively sharp in the sense that the upper-bound is uniformly
non-trivial as soon as $X(g)$ is away from zero.

\begin{proof}
We deduce from Lemma~\ref{l:clean} that uniformly for all $\lambda\in \Fma^*_\C$,
 $k\in K$ and $g\in \cC$,
\[
\max_{1\le \deg D \le 2}
|\langle \lambda, H_0(kg;D)\rangle| \gg_{\cC}
\|\lambda \| \|X(g)\|,
\]
where we have fixed a basis of $\Fmk$ and $D$ ranges over monomials in this basis
of degree $1$ and $2$.  Indeed the uniformity in $\lambda$ follows by compactness,
while the uniformity in $X(g)$ follows from the calculation of $H_0(kg;D)$
in~\cite{DKV83}*{(6.4)} which involves only $\sinh(\mathrm{ad} X(g))$.

We are in position to apply a multidimensional van der Corput 
estimate~\cite{Stein}*{\S{}VIII.2.2} to the 
integral~\eqref{def:spherical} and conclude the proof of the proposition.
\end{proof}

Our method of proof should be applicable to other situations where the critical 
set can be more complicated, for example  when combined with 
~\cite{CCW:multidimensional-vandercorput} which achieve multidimensional van der 
Corput estimates with the best uniformity. 
Indeed in such situation one can establish a soft bound by studying higher 
derivatives in $F^k\cU_0(\Fmk)$ in Lemma~\ref{l:clean} and obtain a bound with a 
power saving estimate that depends on $k$ and the dimension.

\subsection{Comparision with previous results}
If at least one of the parameters $X(g)$ or $\lambda$ is uniformly bounded away from the
singular set, we can do better as follows. For $w\in W$, let
$\Sigma_w^+(\cC)=\{\alpha\in\Phi^+\mid \forall g\in \cC:~(w\alpha)(X(g))\neq0 \}$.  Then
by~\cite{DKV83}*{Cor.11.2}
\begin{equation}\label{eq:dec:spher:optimal}
\left| \phi_{\lambda}(g)\right|
\ll_\cC
\sum_{w\in W} \prod_{\alpha\in\Sigma^+_w(\cC)}\left(1+|B(\alpha,\lambda)|\right)^{-\frac12 m(\alpha)}
\end{equation}
for all $\lambda\in i\ka^*$ and all $g\in \cC$. Here $B$ is the Killing form, and $m(\alpha)$ denotes the multiplicity of $\alpha$ (which is $1$ for any $\alpha$ if $G=\GL_n(\R)^1$).
However, this is not strong enough for us, as we need to consider compact sets $\cC$
containing a neighborhood of the identity. If $\cC$ intersects $K$, then
$\Sigma^+_w(\cC)=\emptyset$ for any $w$
 so that~\eqref{eq:dec:spher:optimal} only recovers the trivial bound.

It follows from the asymptotic expansion in \cite{DKV83}*{Thm.9.1 and \S11} that
the exponents in  \eqref{eq:dec:spher:optimal} are sharp if $\cC$ is a compact
subset of $G$ such that the $X(g)$, $g\in\cC$, are equisingular.
 This implies that the  exponent $1/2$ in \eqref{eq:decay:spherical} is optimal in this degree of uniformity, which can be seen as
follows. Fix any simple root $\alpha_1\in\Phi^+$ of $G=\GL_n(\R)^1$. Let $\varpi_1$ be the corresponding fundamental weight so that  $\alpha_1(\varpi_1^\vee)=1$. Suppose that $\lambda\neq 0$ varies in
$i\R\varpi_1$.
Then, for $\alpha\in\Phi$, $B(\alpha,\lambda)\neq0$ only if
$\alpha(\varpi_1^\vee)=\pm1$, that is, if $\alpha_1\prec\alpha$ or $\alpha_1\prec -\alpha$ where $\prec$ denotes the usual ordering on the root lattice.
Suppose $\cC$ is a compact subset of the set of $g\in G$ with $X(g)\in \R_{ > 0}\varpi_1$.
Let $w_1\in W$ be the simple reflection associated to $\alpha_1$. The only root in the set $\Sigma^+_{w_1}(\cC)$ which does not vanish identically on $\lambda\in
i\R \varpi_1$ is $\alpha_1$. Hence the summand in
\eqref{eq:dec:spher:optimal} corresponding to $w=w_1$ equals the single factor
 \[
  \left(1+|B(\alpha_1,\lambda)|\right)^{-\frac12 m(\alpha_1)}.
 \]
 More precisely, the asymptotic expansion in \cite{DKV83}*{Thm.9.1} implies that
 there is a non-zero function $a:\cC\longrightarrow\C$ such that
 \[
  \phi_\lambda(g)
  = a(g) \|\lambda\|^{-\frac12 m(\alpha_1)} + O_{\cC} \left(\|\lambda\|^{-\frac12 m(\alpha_1)-1}\right)
 \]
 for $\lambda\in i \R_{>0}\varpi_1$.
This limits the decay of $\phi_\lambda(g)$ to the rate specified in \eqref{eq:decay:spherical}.

If $\cB\subset i\Fma^*$ is a compact set bounded away from the singular set, then
Marshall~\cite{Mar13}*{Thm.2} showed that
\[
 \left|\phi_{t\lambda}(g)\right|\ll_{\cB, \cC} \prod_{\alpha\in\Phi^+} \left(1+t|\langle \alpha,X(g)\rangle|\right)^{-\frac12 m(\alpha)}
\]
for all $\lambda\in \cB$, $t\ge 1$, and $g\in \cC$. This however does not yield a good upper bound for $f^{\mu}_\pm$ because the constraint $\lambda \in \cB$ prevent us from performing the integration in the definition of $f^{\mu}_\pm$.

If $G$ is a \emph{complex group}, then there is an exact formula for the zonal spherical
function, see e.g.~\cite{CoNe01}*{(2.2)}.
It can be expressed in terms of basic functions, since then the hypergeometric
functions appearing can be given in a closed form.
For any regular $\lambda\in \Fma^*_\C$ and $X(g)\in \Fma$,
\begin{equation}\label{eq:explicit:spherical}
\phi_\lambda(g)
=2^{|\Phi^+|} \prod_{\alpha\in\Phi^+}\frac{B(\alpha, \rho)}{B(\alpha,\lambda)}
\sinh(\langle \alpha,X(g)\rangle )^{-1}  \sum_{w\in W} \operatorname{sgn}(w) e^{\langle w\lambda, X(g) \rangle}.
\end{equation}
This formula allows for the following estimate by treating the sum over $w\in W$ trivially.
If we combine the formula with the spherical Plancherel density $\left|
\Tmc_\C(\lambda)\Tmc_\C(\rho)^{-1}\right|^{-2}$ for $\GL_n(\C)$, one obtains the upper bound that
for every $g$ in a compact set $\cC\subset G$ and every $\lambda\in i\Fma^*$ is given by
\begin{equation}\label{eq:complex:spherical}
\left| \phi_{\lambda}(g)\right|\left| \Tmc_\C(\lambda)\right|^{-2}
\ll_{\cC} (1+\|\lambda\|)^{d_\C-r-(n-1)} \prod_{\alpha\in\Phi^+:\, \langle \alpha,X(g)\rangle\neq0} |\langle \alpha,X(g)\rangle|^{-1}.
\end{equation}
In comparison, one obtains, using the estimate from Proposition \ref{prop:spher:fct:bound} that for every $g$ in a compact set $\cC\subset G$, and every $\lambda \in i\Fma^*$,
\[
 \left| \phi_{\lambda}(g)\right|\left| \Tmc(\lambda)\right|^{-2}
 \ll_{\cC} (1+\|\lambda\|)^{d-r-\frac12} \|X(g)\|^{-\frac12}.
\]
Thus the estimate~\eqref{eq:explicit:spherical} has a better exponent in $\lambda$.
However, it is not sufficient for our purpose, because $\langle\alpha, X(g)\rangle$ can become arbitrary small for some $\alpha \in \Phi^+$, which is an obstacle in estimating $f^\mu_\pm(g)$ and $J^G_M(\gamma,f^\mu_\pm)$.

As mentioned in the introduction, Blomer-Pohl~\cite{BP:sup-norm-Siegel} have obtained the same estimate as Proposition~\ref{prop:spher:fct:bound}, and their proof differs slightly from ours although it builds on the same idea of applying a multidimensional van der Corput estimate. Finally we refer to~\cite{Anker-Ji} for other known properties of zonal spherical functions.

\subsection{An upper bound for  $J_M^G(\gamma,f^{\mu}_{\pm})$}
We will combine the upper bound for the spherical function in Proposition~\ref{prop:spher:fct:bound},
with the results from the previous Section~\ref{sec:orbital} on $\tilde{J}_M^G$ tested against functions of the form $\tilde f\|X(\cdot)\|^{-\eta}$, $\tilde f\in C_c^{\infty}(G)$.

We first bound $|f^\mu_\pm|$ pointwise.
The trivial bound is that for all $g\in G$ and all $\mu\in i\Fma^*$ one has
\begin{equation}\label{fmu_trivial}
\left|f^{\mu}_{\pm}(g)\right| \ll_{n,h} (1+\|\mu\|)^{d-r}.
\end{equation}
This follows immediately from the trivial bound $|\phi_\lambda(g)|\le 1$ which holds for all $g\in G$ and $\lambda\in i\Fma^*$, and by the definition of $f^\mu_\pm$ in \eqref{eq:def:test:function}.
We deduce from Proposition~\ref{prop:spher:fct:bound} the following upper-bound on $f^{\mu}_{\pm}$, which improves on the trivial bound provided that $g$ is bounded away from $K$.
\begin{corollary}\label{p:bound-fmu}
There exists a smooth and compactly supported function $\widetilde f : G\to \R_{ \ge  0}$  depending only on $n$ and $h$ such that
\begin{equation}\label{eq:growth:test:fcts}
\left|f^{\mu}_{\pm}(g)\right|
\le
             (1+\|\mu\|)^{d-r-\frac12 } \tilde f(g) \|X(g)\|^{-\frac12}
\end{equation}
for every $g\in G-K$ and every $\mu\in i\Fma^*$.
\end{corollary}

\begin{proof}
	We need a bound for the Harish-Chandra $\Tmc$-function:
	\begin{equation}\label{eq:decay:plancherel}
	\left|\Tmc(\lambda)\right|^{-2}
	\ll_{n} (1+\|\lambda\|)^{d-r}
	\end{equation}
	for all $\lambda\in i\ka^*$. This follows from~\eqref{def:c}.
	The corollary follows by combining~\eqref{eq:decay:plancherel} and~\eqref{eq:decay:spherical}.
\end{proof}

Combining this with Lemma~\ref{l:unidescent} we can also bound the parabolic descent of $f^{\mu}_{\pm}$:

\begin{corollary}\label{cor:parabolic:descent}
 Assume $n\ge 3$.
 Let $Q=LV\subsetneq G $ be a proper semi-standard parabolic subgroup. Then there exists a compactly supported smooth function $\tilde f:L\longrightarrow\R_{\ge0}$ depending only on $h$ and $Q$, such that for every $\mu\in i\ka^*$ we have
\[
\left|( f^{\mu}_{\pm})^{(Q)}(g)\right|
\le (\left| f^{\mu}_{\pm}\right|)^{(Q)}(g)
\le 	(1+\|\mu\|)^{d-r-\frac12} \tilde f(g),
\]
for every $g\in L$.
\end{corollary}

Our main result in this section is then the following.

\begin{proposition}\label{p:J-sph}
	Assume $n\ge 3$.
	There exist $c_1>0$ depending only on $n$ and $c>0$ depending only on $n$ and $h$, such that the following holds.
\begin{enumerate}[(i)]
\item\label{item:i} 	For every $M\in\CmL$, $\gamma\in M$ such that $\gamma_s\neq\pm 1$, and $\mu\in i\ka^*$ we have
	\[
	\tilde{J}_M^G(\gamma,f^{\mu}_{\pm})
	\le  c \Delta^-(\gamma_s)^{c_1} (1+\|\mu\|)^{d-r-\frac12}.
	\]
	In particular, this inequality also holds if we replace the left hand side by $|J_M^G(\gamma,f^{\mu}_{\pm})|$.

\item\label{item:ii} For every $M\in \CmL$, $\gamma\in M$, $\mu\in i\Fma^*$, and  every
proper semi-standard parabolic subgroup $Q=LV\subsetneq G$ with $M\subseteq L$, we have
	\[
	 \tilde J^L_M(\gamma,(f^\mu_{\pm})^{(Q)})
	 \le c \Delta^-(\gamma_s)^{c_1} (1+\|\mu\|)^{d-r-\frac12}.
	\]
	Again, the left hand side can be replaced by $| \tilde J^L_M(\gamma,(f^\mu_{\pm})^{(Q)})|$.
\end{enumerate}
\end{proposition}

\begin{proof}
To prove assertion \eqref{item:i}, we first
note that
by Corollary~\ref{p:bound-fmu} we have
\[
 \tilde J_M^G(\gamma, f^{\mu}_{\pm})
\le (1+\|\mu\|)^{d-r-\frac12} \tilde{J}_M^G(\gamma,\tilde{f}\|X(\cdot)\|^{-\frac12}),
\]
because the intersection of the support of the functional
$\tilde{J}_M^G(\gamma,\cdot)$
with $K$ has measure $0$ as long as $(M,\gamma_s)\neq (G,\pm1)$.
Hence applying Proposition~\ref{prop:bound:weight:orb:non:central}
finishes the proof of~\eqref{item:i}.

Assertion~\eqref{item:ii} follows from Corollary
\ref{cor:parabolic:descent} and Proposition~\ref{prop:bound:weight:orb:non:central}.
\end{proof}

\begin{example}\label{ex:orbital-split-regular}
Let $M=G$ and $\gamma$ be a \emph{split regular} semisimple element. Without loss of generality, we may assume that $\gamma \in T_0$. Then for any bi-$K$-invariant function $f$,
\[
 J_G^G(\gamma, f)
 = |D^G(\gamma)|^{1/2} \int_{U_0} f(u^{-1}\gamma u)\, du.
\]
Changing variables from $u$ to $v:=\gamma^{-1}u^{-1}\gamma u$ we need to multiply the integral by the Jacobian $|D^G(\gamma)|^{-1/2} \delta_0(\gamma)^{1/2}$ so that we obtain
\[
J^G_G(\gamma,f)= \delta_0(\gamma)^{1/2} \int_{U_0} f(\gamma v)\, dv
 = f^{(P_0)}(\gamma),
\]
compare with Example~\ref{ex:split-regular}.
Specializing to $f=f^\mu_\pm$, this is the inverse transform to~\eqref{eq:def:test:function},
 that is
\[
J^G_G(\gamma,f^\mu_+)
= h(H_0(\gamma)) e^{\langle\mu, H_0(\gamma) \rangle}
\]
and
\[
J^G_G(\gamma,f^\mu_-)
= \operatorname{sgn}(\det\gamma) h(H_0(\gamma)) e^{\langle\mu, H_0(\gamma) \rangle}.
\]
Hence $|J_G^G(\gamma, f^\mu_\pm)|\le ||h||_\infty$,
for every $\mu\in i\Fma^*$, and every split regular semisimple element $\gamma$.
\end{example}

\subsection{Unipotent weighted orbital integrals}
We establish now the estimates for unipotent weighted orbital integrals as well.
\begin{proposition}\label{prop:bound-weighted-unip}
Let $M\in \CmL$ and let $\gamma\in M$ be unipotent such that $(M,\gamma)\neq (G, \pm 1)$. Then
we have the following:
\begin{enumerate}[(i)]
 \item There exists $\delta>0$ depending only on $n$, such that
 \[
   \left|J_M^G(\gamma,f^{\mu}_{\pm})\right|
   \ll (1+\|\mu\|)^{d-r-\delta}.
 \]

 \item  For every $t\ge2$ we have
 \[
  \left|\int_{t\Omega} J_M^G(\gamma, f^\mu_\pm)\, d\mu\right|
  \ll t^{d-1} (\log t)^{\max(3, n)}.
 \]
\end{enumerate}
\end{proposition}
\begin{proof}
 The proof of assertion (i) is the same as for Proposition~\ref{p:J-sph} but we have to use Proposition~\ref{prop:weight:orb:int:central}  instead of Proposition~\ref{prop:bound:weight:orb:non:central}.
  Assertion (ii) is \cite{LM09}*{\S5}.
\end{proof}

\part{Global theory: Weyl's law and equidistribution}\label{part2}

This Part~\ref{part2} of the paper is about proving global results, namely
Theorem~\ref{thm:weyl}.
This will be done by using the Arthur-Selberg trace formula for
$\GL(n)_{/\Q}$.
The left-hand side and right-hand side of the asymptotic
of Theorem~\ref{thm:weyl} will be identified as the main terms on the spectral and
geometric side of the trace formula, respectively.
We take the approach of studying the geometric side of the trace formula by reducing it to local problems.
This involves using Arthur's fine geometric expansion, Arthur's splitting formula for
weighted orbital integrals, a careful study of the properties of the equivalence classes
$\Fmo$ that contribute non-trivially, and a uniform description of the measures on
centralizers that appear locally and globally on the geometric side of the trace formula.
For these purposes, we build up~in Section~\ref{sec:centralizers} some core 
material on
centralizer subgroups which are not easily accessible in the literature. In
Section~\ref{sec:fine} we carefully summarize Arthur's results, adding to them some 
quantitative analysis, we recall the bound of the first-named 
author~\cite{Ma:coeff} on Arthur's global coefficients,
and we formulate our main 
estimate as Theorem~\ref{th:bound-Jgeom}.
Then Section~\ref{sec:p-adic} finishes the proof, namely it establishes a uniform
upper-bound for non-archimedean orbital integrals.
 The method of proof originates from the work of Shin and the second-named 
 author~\cite{ST11cf}. However we rework the whole argument in depth, first because we 
 are treating the more delicate weighted
orbital integrals, and second because we take the opportunity correct an inaccuracy.
Finally, Section~\ref{sec:spectral} handles the spectral side similarly as in the work of 
Lapid--M\"uller~\cite{LM09}.

The writing of the present paper happened simultaneously with revisions of~\cite{Matz}, 
hence there is some overlap in content with this Part~\ref{part2}. After the publication 
of~\cite{Matz}, some improvements have been incorporated in the present paper. On the 
whole, we provide a 
streamlined treatment which can be read largely independently of~\cite{LM09,ST11cf,Matz}.

\subsection*{Notation}
From now on $G$ denotes the group $\GL(n)$ as an algebraic group over $\Q$.
Thus our notation differs slightly from the notation of Part~\ref{partI}, where we worked 
with the group $G(\R)^1$.
 Further $T_0$ denotes the torus of diagonal matrices in $G$, and $P_0$ the minimal parabolic subgroup in $G$ of upper triangular matrices both considered as $\Q$-algebraic groups. Similarly as before, but now in the category of $\Q$-algebraic groups, we define the notions of \emph{(semi-) standard} parabolic and \emph{semi-standard} Levi subgroups, and also define the sets $\CmL(M)$, $\CmF(M)$, and $\CmP(M)$.

 If $v$ is a place of $\Q$, we denote by $|\cdot|_v$ the normalized absolute value on $\Q_v$. Let $\A$ be the adeles of $\Q$, and let $|\cdot|_{\A}$ denote the adelic absolute value on $\A^{\times}$ which is the product of the $|\cdot|_v$. Then $G(\A)^1$ denotes the set of all $g\in G(\A)$ with $|\det g|_{\A}=1$.

\section{Setup for the Arthur-Selberg trace formula for
\texorpdfstring{$\GL(n)$}{GL(n)}}\label{sec:arthur-selberg}

\subsection{Maximal compact subgroups and measures}
If $p$ is a finite prime, we take $K_p=G(\Z_p)$ as the maximal compact subgroup of $G(\Q_p)$, and normalize the Haar measure $dk$ on $K_p$ such that $\vol(K_p)=1$. Similarly, we normalize the Haar measures on $\Q_p$ and $\Q_p^{\times}$ such that $\vol(\Z_p)=1=\vol(\Z_p^{\times})$. We can identify $T_0(\Q_p)$ with $(\Q_p^{\times})^n$ via the usual coordinates which then defines a Haar measure on $T_0(\Q_p)$. Similarly, if $U$ is the unipotent radical of a semi-standard parabolic subgroup, we identify $U(\Q_p)$ with $\Q_p^{\dim U}$ via the usual coordinates which then again defines a Haar measure on $U(\Q_p)$.
From the integration formula
\[
\int_{G(\Q_p)} f(g) ~dg
= \int_{K_p} \int_{T_0(\Q_p)} \int_{U_0(\Q_p)} f(tuk) ~du~dt~dk,
~~f\in L^1(G(\Q_p)),
\]
we obtain a Haar measure on $G(\Q_p)$. The analogue of this integration formula also defines Haar measures on $M(\Q_p)$ for any $M\in\CmL(T_0)$.

At the archimedean place we use the same maximal compact subgroup and the same measures
as in
Part~\ref{partI}. Globally we take the product measures.
On $G(\A)^1$ we define a Haar measure via the exact sequence
\[
1\longrightarrow G(\A)^1\longrightarrow G(\A)\longrightarrow \R_{>0}\longrightarrow 1,
\]
where the map $G(\A)\longrightarrow\R_{>0}$ is given by $g\mapsto |\det g|_{\A}$.

\subsection{Test functions at the non-archimedean places}\label{s:Hecke-operators}
We are going to use elements of the spherical Hecke algebra $\CmH_p:=C_c^{\infty}(G(\Q_p)/\!/K_p)$ as our test function at $p$.
For a tuple $\xi=(\xi_1, \ldots, \xi_n)\in\Z^n$ we denote by $p^{\xi}$ the diagonal matrix $\Mdiag(p^{\xi_1}, \ldots, p^{\xi_n})$, and define
\[
\tau_{p,\xi}:G(\Q_p)\longrightarrow \{0,1\}
\]
as the characteristic function of the double coset $K_p p^{\xi}K_p$. As a convolution algebra, $\CmH_p$ is generated by the functions $\tau_{p,\xi}$ with $\xi$ running over all $\xi\in\Z^n$ with $\xi_1\ge\ldots\ge\xi_n$.

\subsection{Family of global test functions}\label{sub:global-test}
We construct our family of global test functions as follows:
Recall from \S\ref{sub:wave-test} the family of test functions $f^{\mu}_{\pm}\in C_c^\infty(G(\R)^1)$, indexed by $\mu\in i\Fma^*$. We extend
$f_{\pm}^{\mu}$ to a function in $C^{\infty}(G(\R))$, invariant under the subgroup $A_G\simeq \R_{ >  0}$ of scalar diagonal matrices with positive entries.  Each $f^\mu_{\pm}$ is compactly supported modulo center and the support does not change as $\mu$ varies. We then take the global test function
\[
 F^{\mu}_{\pm}:=(f^{\mu}_{\pm}\cdot\tau)_{|G(\A)^1} \in C_c^{\infty}(G(\A)^1)
\]
with $\tau\in C_c^{\infty}(G(\A_f))$ a bi-$K_f$-invariant compactly supported function.
We will consider factorizable $\tau= \prod_{p<\infty}\tau_p$ with each $\tau_p$ running over a set of generators of the spherical Hecke algebra at $p$. More precisely, for each $p$, we are given a tuple of integers $\xi^p=(\xi^p_1,\ldots,\xi^p_n)$ with $\xi^p_1\ge\ldots\ge\xi^p_n$ such that $\xi^p=0$ for all but finitely many $p$. We then take $\tau_p=\tau_{p,\xi^p}$.
Since $F_\mu^\pm$ is obtained by restriction to $G(\A)^1$ and $f^\mu_\pm$ is invariant under $A_G=\R_{>0}$, we can assume without loss of generality that $\xi^p_n = 0$ for all $p$.

This choice of test function is tailored to prove Theorem~\ref{thm:weyl}. Only unramified spherical (respectively, of $K_{\infty}$-type $\chi_-$) representations with infinitesimal character close to $\mu$ will contribute to the cuspidal part of the trace formula if we use the test function $F^{\mu}_+$ (respectively, $F^\mu_-$).

\subsection{The trace formula}
Arthur's trace formula~\cite{Arthur:tr-I}*{(4)} is an identity of distributions
\[
J_{\text{spec}}(f)
=
\Jgeom(f)
\]
of the geometric and spectral side valid for test functions $f\in C_c^{\infty}(G(\A)^1)$.
In Section~\ref{sec:fine} we
provide a detailed analysis of the geometric side.
The strategy then is to consider the integrated trace formula
\[
\int_{t\Omega} J_{\text{spec}}(F^{\mu}_{\pm})~d\mu=
 \int_{t\Omega} \Jgeom(F^{\mu}_{\pm})~d\mu.
\]
We shall identify the main terms as $t\rightarrow\infty$ on both sides as the main terms occurring in Theorem~\ref{thm:weyl}, and estimate the remainder terms
with a power saving in $t$ and a polynomial control in the Hecke operator $\tau$.

\section{Centralizers of semisimple elements}\label{sec:centralizers}
The purpose of this section is to describe the centralizers of semisimple elements in
$G(\Q)$ in a uniform way. This will allow us to formulate uniform estimates for orbital
integrals in the subsequent sections.

\subsection{Conjugacy classes and splitting fields}
Let $F$ be a field.
Let $\CmE=(E_i, m_i)_{i\in I}$ be a tuple consisting of finite field extensions $E_i$ of $F$, and integers $m_i\ge 1$, indexed by a finite set $I$.
We call $n(\CmE):=\sum_{i\in I} m_i [E_i:F]$ the \emph{dimension} of $\CmE$.
For an integer $n\ge 1$, we let $\kR_F^n$ be the set of all such tuples $\CmE$ of
dimension $n=n(\CmE)$, up to isomorphism.
Namely, we identify $\CmE=(E_i, m_i)_{i\in I}$ and $\CmE'=(E'_i, m'_i)_{i\in I'}$ if
there exists a bijection $\pi:I\stackrel{\sim}{\to} I'$ such that $m'_{\pi(i)}=m_i$ and
$E'_{\pi(i)}$ is isomorphic to $E_i$ for every $i\in I$.

We associate the following reductive group over $F$, using the functor of restriction of scalars
 	\[
M_{\CmE} := \prod_{i\in I}
\operatorname{Res}_{E_i/F} \GL_{m_i}.
\]

Choosing an $F$-linear basis of $E_i$, we can define an $F$-embedding of
$\operatorname{Res}_{E_i/F}\GL_{m_i}$ into $\GL_{d_im_i}$, where $d_i=[E_i:F]$.
 Fixing an ordering $I=\{1,\ldots, r\}$, we construct an $F$-embedding of $M_\CmE$ into
 $G=\GL_n$ by embedding $\GL_{d_1m_1}\times\ldots\times\GL_{d_rm_r}$ diagonally.
 The $F$-algebraic group $M_\CmE$ depends only on $\CmE$ up to isomorphism, and
 furthermore the embedding $M_\CmE \hookrightarrow G$ is independent of the choice of
 linear bases, up to conjugation by $G(F)$.
 In particular, the subgroup
  \[
 M_\CmE(F)=\GL_{m_1}(E_1)\times \GL_{m_2}(E_2)\times \ldots \times\GL_{m_r}(E_r)
 \subseteq G(F),
 \]
is well-defined up to $G(F)$-conjugation, which is a variant of the Skolem--Noether theorem.

Let $\sigma\in G(F)$ be semisimple.
 The characteristic polynomial decomposes as a product
 \[
 P_\sigma(X)=\prod_{i=1}^r P_i(X)^{m_i},
 \]
 where $P_i$ are monic irreducible and pairwise distinct, and $m_i\ge 1$.
 Put $E_i:=F[X]/(P_i(X))$.
 We obtain a map
 \begin{equation}\label{eq:ss:root:systems}
 \CmE:\{\text{semisimple conjugacy classes in }G(F)\}\longrightarrow \kR^n_F
 \end{equation}
 associating with the $G(F)$-conjugacy class $\{\sigma\}$ the tuple $\CmE(\sigma):= (E_i, m_i)_{i\in\{1,\ldots,r\}} $.

\begin{lemma}\label{l:centralizer-isomorphism}
	Let $\sigma\in G(F)$ be semisimple.
	The centralizer $C_G(\sigma)$ is connected and reductive.
	The $G(F)$-conjugacy class of $\sigma$ is determined by $P_\sigma$, and we have an isomorphism
	\[
	M_{\CmE(\sigma)} \stackrel{\sim}{\to} C_G(\sigma) \subseteq G.
	\]
\end{lemma}

\begin{proof}
The centralizer $C_G(\sigma)$ is the open subset, defined by non-vanishing of the 
determinant, of the set of the $F$-vector space consisting of matrices $g\in 
M_n(F)$ that commute with $\sigma$, i.e., that satisfy the linear equation 
$g\sigma=\sigma g$. This implies that $C_G(\sigma)$ is connected.

Up to $G(F)$-conjugacy, $\sigma$ can be written in block diagonal form $\Mdiag(\sigma_1, \ldots, \sigma_r)$ with $\sigma_i$ elliptic elements in $\GL_{d_im_i}(F)$ with distinct minimal polynomials $P_i(X)$, and characteristic polynomials $P_i(X)^{m_i}$, and such that $C_G(\sigma)$ is contained in the Levi subgroup $\GL_{d_1m_1}\times\ldots\times\GL_{d_rm_r}$, diagonally embedded in $G$.
One can further conjugate $\sigma_i$ to a block diagonal matrix $\Mdiag(\delta_i,\ldots,\delta_i)$ where $\delta_i$ is a regular elliptic element in $\GL_{d_i}(F)$ with characteristic polynomial $P_i(X)$.
Indeed, one can simply construct $\delta_i$ as the companion matrix of $P_i$, and this also implies that the $G(F)$-conjugacy class of $\sigma$ is determined by $P_\sigma$.

We also deduce that the centralizer $C_{\GL_{d_i}}(\delta_i)$ is the elliptic torus $\operatorname{Res}_{E_i/F} \G_m$, and the centralizer $C_{\GL_{d_im_i}}(\sigma_i)$ is $\operatorname{Res}_{E_i/F} \GL_{d_i}$. The isomorphism follows.
\end{proof}

\begin{example}
 \begin{enumerate}
\item  If $\sigma\in G(F)$ is an $F$-split regular semisimple element, we have $\CmE(\sigma)=(F,1)_{i\in\{1,\ldots, n\}}$, that is a tuple of length $n$ with entry $(F,1)$ everywhere.

\item If $\sigma$ is regular, then $C_G(\sigma)$ is a maximal torus of $G$, and we have $m_i=1$ for all $i$.

\item If $\sigma$ is $F$-elliptic, we have $r=1$, and $\CmE(\sigma)= (E, m)$ with
$E$ a field extension of degree $d$ over $F$ with $md=n$. For example, if $\sigma$ is
regular then $m=1$ and $d=n$, that is, $\CmE(\sigma)=(E, 1)$. On the other extreme, if
$\sigma$
is central then $m=n$ and $d=1$, that is, $\CmE(\sigma) = (F, n)$.

\item Let $F=\R$.
 Lemma~\ref{l:centralizer-isomorphism} could be compared with the similar
 Lemma~\ref{lem:centralizer:twisted:levi} for the Lie group $\GL_n(\R)^1$.
Using the notation from Section~\ref{sec:proof}, we can identify
$\kR^n_{\R}$ with the set of pairs
 $(T,L_2)$ with $T\in\CmT^{G(\R)^1}_{\max}$ and
 $L_2\in\CmL^{G(\R)^1}_{\text{twist}}(T)$ up to conjugacy by the Weyl group of
 $(T,G(\R)^1)$.
 Indeed, the tori $T$ correspond to pairs $(r_1, r_2)$ of non-negative integers
 with $r_1+ 2r_2=n$, and the twisted Levi subgroup $L_2$ correspond to partitions
 $r_1=\sum_{i=1}^s m_i$, $r_2=\sum_{j=1}^t k_j$, and the corresponding tuple in
 $\kR^n_{\R}$
 equals $
 \left((\R,
 m_1),\ldots,(\R,m_s),(\C,k_1),\ldots, (\C,k_t)\right)$.

 \end{enumerate}
\end{example}

\begin{remark}
	If $F$ has characteristic zero, then the primitive element theorem implies that the
	map $\CmE$ in \eqref{eq:ss:root:systems} is surjective.
\end{remark}

\begin{lemma}\label{l:conjugate-integral}
	Suppose that $F$ is the field of fraction of a unique factorization domain $R$, and that $\sigma\in M_n(F)$ is semisimple and its characteristic polynomial $P_\sigma(X)$ is in $R[X]$. Then $\sigma$ is $G(F)$-conjugate to a matrix in $M_n(R)$.
\end{lemma}

\begin{proof}
	By Gauss lemma, $P_i(X)$ is in $R[X]$ for every $i\in \{1,\ldots ,r\}$. Let
	$\delta_i\in M_{d_i}(R)$ be the companion matrix of $P_i$, and let
	$\sigma_i:=\Mdiag(\delta_i,\ldots,\delta_i)$ with multiplicity $m_i$.
	Then $\Mdiag(\sigma_1,\ldots,\sigma_r)\in M_n(R)$ is semisimple, and its characteristic polynomial is equal to $P_\sigma(X)$. Hence it is $G(F)$-conjugate to $\sigma$.
\end{proof}

\subsection{Discriminant bounds}
\begin{lemma}\label{l:discriminant}
	Let $\sigma \in G(\Q)$ be semisimple, $\CmE(\sigma)= (E_i,m_i)_{i\in I}$ and $P_\sigma(X)=\prod\limits_{i\in I} P_i(X)^{m_i}$.

(i) If $P_\sigma\in \Z[X]$, then
\[
\prod_{i\in I}
|D_{E_i}|^{m_i}
\le
\prod_{i\in I}
|\operatorname{Disc}(P_i)|^{m_i}
\le
|\det(\sigma)|^{n-1}
|D^G(\sigma)|.
\]

(ii) If $P_\sigma\in \Z_p[X]$ and $|\det(\sigma)|_p=|D^G(\sigma)|_p = 1$, then $p$ is 
unramified in $E_i$ for every $i\in I$.
\end{lemma}

\begin{proof}
(i)	It follows from Gauss lemma that each $P_i(X)$ has integral coefficients.
	As in the proof of Lemma~\ref{l:centralizer-isomorphism}, we can find a $G(\Q)$-conjugate $\sigma'$ of $\sigma$ of the form
\[
\sigma' = \Mdiag(\delta_1,\ldots,\delta_1,\delta_2,\ldots,\delta_2,\ldots,\delta_r,\ldots,\delta_r),
\]
where $\delta_i\in \GL_{d_i}(\Q)$ is regular elliptic and appears with multiplicity $m_i\ge 1$, and its characteristic polynomial is $P_i(X)$. Using the integrality of $P_i$, we have
\[
\prod_{1\le i\le r} |\det(\delta_i)|^{(d_i-1)m_i} |D^{\GL_{d_i}}(\delta_i)|^{m_i} \le 
|\det(\sigma)|^{n-1} |D^G(\sigma)|.
\]
Since the field $E_i=\Q[X]/P_i(X)$ contains the order $\Z[X]/P_i(X)\simeq \Z[\delta_i]$, we deduce
\[
|D_{E_i}| \le |\operatorname{Disc}(\Z[\delta_i])| = |\operatorname{Disc}(P_i)| =  
|\det(\delta_i)|^{d_i-1} |D^{\GL_{d_i}}(\delta_i)|.
\]

(ii) The same argument yields $|D_{E_i}|_p \ge |\det(\sigma)|_p^{n-1} |D^G(\sigma)|_p=1$.
Hence $|D_{E_i}|_p=1$, establishing the assertion.
\end{proof}

\begin{lemma}\label{lem:length}
Let $p$ be a prime, and $\sigma\in G(\Z_p)$ be regular semisimple. Then
$|D^G(\sigma)|_p \le p^{-2\delta}$, where $\delta$ is the length of the $\Z_p$-module
$\CmO_E/R$, with $E=\Q_p[X]/P_\sigma(X)$ and $R= \Z_p[X]/P_\sigma(X)$.
\end{lemma}

\begin{proof}
We have $|\operatorname{Disc}(R)|_p = |\operatorname{Disc}(P_\sigma)|_p = |D^G(\sigma)|_p$.
On the other hand, since $p^\delta$ is the index of $R$ inside the normalization 
$\CmO_E$, we have
\[
 |\operatorname{Disc}(R)|_p = |\operatorname{Disc}(\CmO_E)|_p p^{-2\delta} \le p^{-2\delta},
\]
which concludes the proof.
\end{proof}

\subsection{Maximal compact subgroups}

Let $p$ be a prime and $\sigma\in G(\Q_p)$ be a semisimple element.
	By Lemma~\ref{l:centralizer-isomorphism}, we have an isomorphism of groups $\prod_{i\in I} \GL_{m_i}(E_i)
	\stackrel{\sim}{\to} C_G(\sigma, \Q_p)$. The maximal compact subgroups of $\prod_{i\in I} \GL_{m_i}(E_i)$ are all conjugate to the standard maximal compact
	$\prod_{i\in I} \GL_{m_i}(\CmO_{E_i})$, and hence the same is true for $C_G(\sigma, \Q_p)$.

\begin{lemma}\label{l:Kottwitz}
	Let $\sigma \in G(\Z_p)$ be semisimple with good reduction $\pmod{p}$, namely 
	$|D^G(\sigma)|_p=1$. Then $C_G(\sigma, \Q_p)\cap G(\Z_p)$ is a 
	maximal compact subgroup of $C_G(\sigma, \Q_p)$.
If $\sigma'\in G(\Z_p)$ is stably conjugate to $\sigma$, then $\sigma'$ is conjugate to 
$\sigma$ under $G(\Z_p)$.
\end{lemma}

\begin{proof}
	This is~\cite{Kot86}*{Prop.7.1} which establishes the same assertion for an
	unramified reductive $p$-adic group $G$.
\end{proof}

\begin{example}
The condition in the lemma is necessary as the following argument shows.
 Let
 \[
 \sigma = \begin{pmatrix}1&1\\0&1+p\end{pmatrix}
 = \begin{pmatrix}1&p^{-1}\\0&1\end{pmatrix}\begin{pmatrix}1&0\\0&1+p\end{pmatrix}\begin{pmatrix}1&-p^{-1}\\0&1\end{pmatrix},
 \]
 which is integral semisimple, but with bad reduction $\pmod{p}$.
 Then
 \begin{align*}
  C_G(\sigma, \Q_p)\cap G(\Z_p)
  &= \left(\begin{pmatrix}1&-p^{-1}\\0&1\end{pmatrix} T_0(\Q_p) \begin{pmatrix}1&p^{-1}\\0&1\end{pmatrix}\right) \cap G(\Z_p) \\
  &= \{\begin{pmatrix}a&p^{-1}(a-b)\\0&b\end{pmatrix}\mid a, b\in\Z_p^\times, \, |a-b|_p\le p^{-1}\}.
 \end{align*}
This is strictly contained in the maximal compact subgroup of $C_G(\sigma, \Q_p)$, which is
\[
 \{\begin{pmatrix}a&p^{-1}(a-b)\\0&b\end{pmatrix}\mid a, b\in\Z_p^\times\}.
\]
\end{example}

We now establish a variant in the case of bad reduction.
The proof is inspired from Eichler theory of optimal embeddings of quadratic orders in quaternion algebras.

\begin{proposition}\label{p:bad-reduction}
\begin{enumerate}[(i)]
 \item	If $\delta\in G(\Q_p)\cap M_n(\Z_p)$ is regular elliptic semisimple, then 
 $C_G(\delta,\Q_p)$ is an elliptic maximal torus, and its
 maximal compact subgroup contains $C_G(\delta, \Q_p)\cap G(\Z_p)$ with index at most
 $|D^G(\delta)|_p^{-1/2}|\det\delta|_p^{-(n-1)/2}$.

\item  If $\sigma= \Mdiag(\delta_1,\ldots, \delta_1,\ldots, \delta_r,\ldots, \delta_r)$ where $\delta_i\in 
\GL_{d_i}(\Q_p)\cap M_{d_i}(\Z_p)$ are regular elliptic semisimple with 
pairwise distinct characteristic polynomials, then there exists a maximal compact 
subgroup of 
$C_G(\sigma, \Q_p)$ which contains 
with index at most 
$|D^G(\sigma)|_p^{-1/2}|\det\sigma|_p^{-(n-1)/2}$ the open compact subgroup $C_G(\sigma, \Q_p)\cap 
G(\Z_p)$.
\end{enumerate}
\end{proposition}

\begin{proof}

	(i) Let $P(X)\in \Z_p[X]$ be the characteristic polynomial of $\delta$. Let 
	$R=\Z_p[\delta]\simeq \Z_p[X]/P$, which is an order in $E=\Q_p[\delta]\simeq \Q_p[X]/P$. 
	We have the natural inclusions
	\[
	R \subseteq E \cap M_n(\Z_p) \subseteq \CmO_E,
	\]
	of orders in $E\subset M_n(\Q_p)$. Moreover $C_G(\delta,\Q_p)=E^\times$, and its maximal compact subgroup is $\CmO_E^\times$. It suffices to bound the index $[\CmO_E^\times:R^\times]$.

	Let $\Fmf\subseteq \CmO_E$ be the conductor of $R$ in $\CmO_E$, that is, the largest
	ideal in $\CmO_E$ which is contained in $R$. Suppose $a, b\in\CmO_E^\times$ are such
	that $a-b\in\Fmf$. Then $a^{-1}b, b^{-1}a\in 1+\Fmf\subseteq R$ and therefore $a^{-1}b
	\in R^\times$, that is, $aR^\times = bR^\times$ in $\CmO_E^\times/ R^\times$. Hence
	$[\CmO_E^\times : R^\times]\le [\CmO_E:\Fmf]$, cf.\ also the proof of
	\cite{book:grouprings02}*{Lem.2.9.5} although the statement there seems to
	contain an inaccuracy, which we have corrected.
	Moreover $|\operatorname{Disc}(P)|_p^{-1} = |D_E| [\CmO_E:\Fmf]^2$. Hence 
	$[\CmO_E^\times : R^\times]\le |\operatorname{Disc}(P)|_p^{-1/2}=|D^G(\delta)|_p^{-1/2} 
	|\det\delta|_p^{- (n-1)/2}$.

	(ii)
	Let $\sigma_i=\Mdiag(\delta_i,\ldots,\delta_i)\in \GL_{d_im_i}(\Q_p)$ which is an $m_i\times 
	m_i$ block diagonal matrix, with $\delta_i\in \GL_{d_i}(\Q_p)$, such that $\sigma = 
	\Mdiag(\sigma_1,\ldots,\sigma_r)\in G(\Q_p)$.
	By Lemma~\ref{l:centralizer-isomorphism}, we have
	\[C_G(\sigma,\Q_p) \simeq C_{\GL_{d_1m_1}}(\sigma_1,\Q_p) \times \cdots \times C_{\GL_{d_rm_r}}(\sigma_r,\Q_p),
		\]
		and similarly for $C_G(\sigma,\Q_p) \cap G(\Z_p)$.
		The integrality of the characteristic polynomial of $\sigma_i$ implies
		\[
		|\det(\sigma)^{n-1} D^G(\sigma)|_p \le \prod_{1\le i\le r}   
		|\det(\sigma_i)^{(d_i-1)m_i} D^{\GL_{d_im_i}}(\sigma_i)|_p,
	\]
	 so it suffices to give a sufficiently good bound for the case $r=1$.

	 Thus let  $\sigma= \Mdiag(\delta,\ldots, \delta)\in G(\Q_p)$ be an $m\times m$ block diagonal matrix, with $\delta\in \GL_{d}(\Q_p)$ regular elliptic semisimple, with integral characteristic polynomial, and $n=dm$.
	We keep notation as in the proof of (i). Using the special form of $\sigma$, 
	we have that $C_G(\sigma, \Q_p)$ is the unit group of the matrix algebra 
	$M_m(E)$ inside $M_n(\Q_p)$, and that $\GL_m(\CmO_E)$, which is the group of 
	units of the ring $M_m(\CmO_E)$, is a maximal compact 
	subgroup of $C_G(\sigma, \Q_p)$. Moreover $\GL_m(\CmO_E)$ contains 
	$C_G(\sigma,\Q_p)\cap 
	G(\Z_p)$.
	 We have the natural inclusion of orders of $M_m(E)$,
	\[
	M_m(R) \subseteq M_m(E) \cap M_n(\Z_p) \subseteq M_m(\CmO_E).
	\]
Hence it  suffices to bound the index of $\GL_m(R)$ in $\GL_{m}(\CmO_E)$.
Let $\Fmf\subseteq \CmO_E$ be the conductor of $R$.
Then
\[
[\CmO_E:\Fmf] \le |\operatorname{Disc}(P)|_p^{-1/2}
 = |D^{\GL_d}(\delta)|_p^{-1/2} |\det\delta|_p^{-(d-1)/2}
= |D^G(\sigma)|_p^{-1/2m^2} |\det\sigma|_p^{-(d-1)/2m}.
\]
Arguing as in the proof of (i) we get
\[
 	[\GL_{m}(\CmO_E) : \GL_{m}(R)] \le
 	 [M_{m}(\CmO_E): M_{m}(\Fmf)].
\]
Now $[M_{m}(\CmO_E): M_{m}(\Fmf)]\le [\CmO_E:\Fmf]^{m^2}$, so
\[
 [\GL_{m}(\CmO_E) : \GL_{m}(R)] \le |D^G(\sigma)|_p^{-1/2} |\det\sigma|_p^{-(d-1)m/2},
\]
 and the assertion then follows.
\end{proof}

\begin{remark}
	The assumption in (ii) could perhaps be relaxed, so as to obtain an
	assertion valid for every integral semisimple $\sigma \in G(\Q_p)$. The
	difficulty however is to find a suitable representative in its
	$G(\Z_p)$-conjugacy class (as opposed to $G(\Q_p)$-conjugacy class). According
	to~\cite{newman:book}*{Thm.III.12}
	 one can find a representative in upper block triangular form, with each diagonal block regular elliptic semisimple.
	However, for example,
	\[
	\sigma=\begin{pmatrix} 1 & 1 \\ 0& 1+p\end{pmatrix}
	= \begin{pmatrix} 1 & p^{-1} \\ 0& 1\end{pmatrix}\begin{pmatrix} 1 & 0 \\ 0& 1+p\end{pmatrix}\begin{pmatrix} 1 & p^{-1} \\ 0& 1\end{pmatrix}^{-1}
	 \]
	 cannot be brought into diagonal form by $G(\Z_p)$-conjugation.
\end{remark}

\subsection{Localization}
Let $F$ be a number field, and $v$ an arbitrary place of $F$. Let $\CmE=(E_i,m_i)_{i\in
I}\in\kR_F^n$.
For each $i\in I$, let $W_{i,v}$ be the set of places of $E_i$ above $v$. Then
$E_{i,v}= E_i\otimes F_v
=\prod_{ w\in W_{i,v}} E_{i,w}$,
where $E_{i,w}$ denotes the completion of $E_i$ at $w$. We write
\[
\CmE_v=(E_{i,w}, m_i)_{i\in I,\, w\in W_{i,v}}
\]
which is an element of $\kR_{F_v}^n$.
We therefore get a map $\kR_F^n\longrightarrow \kR_{F_v}^n$  and the resulting diagram
\[
\xymatrix{
	\{\text{ss. conj.-cl. in }G(F)\} \ar[d]\ar[r] 	& \kR^n_{F} \ar[d]\\
	\{\text{ss. conj.-cl. in }G(F_v)\} \ar[r] 		& \kR^n_{F_v}
}
\]
commutes. Note that the set $\kR_{F_v}^n$ is finite.

\subsection{Choice of measures}\label{sec:measures:p-adic}
In the following sections we need to make a choice of measures on $C_G(\sigma, 
\Q_v)$ for
every semisimple $\sigma\in G(\Q)$ and every place $v$ of $\Q$. Globally on $C_G(\sigma,
\A)$, we then take the product measure.

Let $\CmE= (E_i, m_i)_{i\in I}\in\kR_\Q^n$. Then $M_{\CmE}(\Q_v)$ is a direct product of groups $\GL_{m_i}(E_{i,w})$, for $i\in I$, $w\in W_{i,v}$. We fix
measures on these groups and take the product measure on $M_{\CmE_v}(\Q_v)$.

If $v$ is non-archimedean, we normalize the Haar measure on each 
$\GL_{m_i}(E_{i,w})$ such that $\GL_{m_i}(\CmO_{E_{i,w}})$ has volume $1$ with 
$\CmO_{E_{i,w}}$ the ring of integers of $E_{i,w}$.
At the archimedean place $v=\infty$ we choose an arbitrary Haar measure on 
$M_{\CmE}(\R)$ (for example one could normalize it as in \S\ref{sub_measures}).

This defines measures on the centralizers $C_G(\sigma, \Q_v)$ as well, thanks to Lemma~\ref{l:centralizer-isomorphism} since $C_G(\sigma, \Q_v)$ is conjugate to $M_{\CmE(\sigma)}(\Q_v)$.

Similarly we can define measures on parabolic subgroups and their unipotent radicals in $C_G(\sigma, \Q_v)$ by pulling them back to parabolic subgroups in $M_{\CmE(\sigma)}(\Q_v)$ and defining measures on the parabolic subgroups in $M_{\CmE(\sigma)}$ as usual.

\section{Coarse and fine geometric expansion}\label{sec:fine}
To handle the geometric side of the trace formula, we break it down into independent
local pieces. We shall use the coarse and fine geometric expansions, and Arthur's
splitting formula.

\subsection{The coarse geometric expansion}
Let $\gamma, \gamma' \in G(\Q)$ with Jordan decomposition $\gamma=\gamma_s\gamma_u$, and $\gamma'=\gamma_s'\gamma_u'$ with $\gamma_s, \gamma_s'$ semisimple and
$\gamma_u\in C_G(\gamma_s,\Q)$, $\gamma_u'\in C_G(\gamma_s',\Q)$ unipotent. Then $\gamma$ and $\gamma'$ are called equivalent if $\gamma_s$ and $\gamma'_s$
are conjugate in $G(\Q)$. Let $\CmO$ denote the set of all such coarse equivalence classes in $G(\Q)$. Hence $\CmO$ is in natural bijection with the set of semisimple $G(\Q)$-conjugacy classes.

If $\Fmo\in \CmO$ is a coarse equivalence class, and $\sigma \in \Fmo$ is a semisimple representative, then $\Fmo$ is a finite union of several
$G(\Q)$-conjugacy classes, the number of which equals the number of unipotent conjugacy classes in the subgroup $C_G(\sigma,\Q)$, that is, the number of orbits of the adjoint action of $C_G(\sigma,\Q)$ on $\CmU_\sigma(\Q)$, where $\CmU_\sigma$ denotes the variety of unipotent elements in $C_G(\sigma)$.

\begin{example}
If $\sigma$ is \emph{regular} semisimple, then $\Fmo$ as a set equals the $G(\Q)$-conjugacy class of $\sigma$.
\end{example}

For each $\Fmo\in\CmO$, Arthur~\cite{Arthur:tr-I} constructs a distribution
$J_{\Fmo}:C_c^{\infty}(G(\A)^1)\longrightarrow\C$  such that
\begin{equation}\label{eq:coarse:exp}
 \Jgeom(f)
=\sum_{\Fmo\in\CmO} J_{\Fmo}(f), \quad \forall f\in C_c^{\infty}(G(\A)^1).
\end{equation}
 All but finitely many of the terms $J_{\Fmo}(f)$ vanish due to the fact that $f$ is
 compactly supported.
More precisely, the support of the distribution $J_{\Fmo}(\cdot)$ is contained in
\[
\bigcup_{g\in G(\A)} g^{-1} \sigma \CmU_{\sigma}(\A) g,
\]
where $\sigma\in \Fmo$ is any semisimple representative for the coarse equivalence class $\Fmo$.

\subsection{The fine geometric expansion}
There is a description of the distributions $J_\Fmo$, the fine geometric expansion, which will be more accessible to analysis.
Arthur~\cite{Ar86}*{Thm.9.2} showed that if $S$ is a sufficiently large finite
set of places containing $\infty$, then there exist coefficients $a^M(\gamma, S)\in \R$
for $M\in\CmL$, $\gamma\in M(\Q)$, such that  one has
\begin{equation}\label{eq:fine:exp:o}
J_{\Fmo}(f)
= \sum_{M\in\CmL} \frac{|W^M|}{|W^G|} \sum_{\{\gamma\}} a^M(\gamma, S) J^G_M(\gamma, f),
\end{equation}
for any $f\in C_c^{\infty}(G(\A)^1)$ of the form $(f_S\cdot f^S)_{|G(\A)^1}$ with $f_S\in C^{\infty}(G(\Q_S))$ and $f^S$ the characteristic function of the
standard
maximal compact subgroup $K^S\subset G(\A^S)$. Here $\gamma$ runs over a set of representatives for the $M(\Q)$-conjugacy classes in $M(\Q)\cap \Fmo$.
By~\cite{Ar86}*{Thm.8.2}, $a^M(\gamma, S)=0$ unless $\gamma_s$ is elliptic in $M(\Q)$.
The value of $a^M(\gamma, S)$ depends on the normalization of measures.
We shall quantify in Lemma~\ref{l:sufficient-fine} below how large $S$ needs to be chosen for Arthur's fine expansion~\eqref{eq:fine:exp:o} to hold.

\begin{remark}
 The assertion that the sum over $\gamma$ in \eqref{eq:fine:exp:o} can be taken over $M(\Q)$-conjugacy classes is because $G=\GL(n)$. In general, one has to take
$\gamma$ over a set of representatives for a certain equivalence relation that depends on $S$.
\end{remark}

Since $J^G_M$ 
is given by an absolutely convergent integral~\cite{Art86}, it follows from the 
trivial 
bound~\eqref{fmu_trivial} that
\begin{equation}\label{simple_geom_est}
\Jgeom((f^\mu_{\pm}\cdot \tau_0)|_{G(\A)^1}) \ll_{h} (1+||\mu||)^{d-r},
\end{equation}
where $\tau_0$ is the characteristic function of $K_f$.

\subsection{Arthur's splitting formula for weighted orbital integrals}
We have that $J_M^G(\gamma,f)=0$ unless $\gamma^S\in K^S:=\prod_{p\not\in S} K_p$, in which case we have $J^G_M(\gamma, f)=J^G_M(\gamma_S, f_S)$.
The distributions $J_M^G(\gamma_S,f_S)$ are $S$-adic weighted orbital integrals.
Their value depends only on the $M(\Q_S)$-conjugacy class of $\gamma_S\in M(\Q_S)$.
These are defined for any finite set $S$ by Arthur~\cite{Ar81} as a special value of a certain $(G,M)$-family.

The archimedean weighted orbital integrals studied in Part~\ref{partI} correspond to
$S=\{\infty\}$.
Similarly, by specializing to $S=\{p\}$, for a finite prime $p$, one obtains $p$-adic weighted orbital integrals, that will be studied
in the next Section~\ref{sec:p-adic}.

Arthur established the splitting formula~\eqref{eq:splitting:orb:int} below, by 
which it is enough to understand the $v$-adic distributions for every $v\in S$. In 
other words the $S$-adic distributions are finite sums of factorizable 
distributions. Suppose that $f_S=\prod_{v\in S} f_v$ with $f_v\in C^{\infty}(G(\Q_v))$, 
and is such that the restriction ${f_S}|_{G(\Q_S)^1}$ is compactly supported.
 Then
\begin{equation}\label{eq:splitting:orb:int}
J_M^G(\gamma_S,f_S)
= \sum_{\underline{L}} d_M^G(\underline{L})\prod_{v\in S} J_M^{L_v}(\gamma_v, f_v^{(Q_v)}),
\end{equation}
where the notation is as follows:
\begin{itemize}
\item $\underline{L}:=(L_v)_{v\in S}$ runs over all tuples of Levi subgroups $L_v\in\CmL(M)$, $v\in S$;
\item $d_M^G(\underline{L})\in\R$ are certain coefficients satisfying $d_M^G(\underline{L})=0$ if the natural map $\bigoplus_{v\in
S}\Fma_{M}^{L_v}\longrightarrow\Fma_M^G$ is not an isomorphism; they take values in a finite set that depends only on $G$;
\item $Q_v\in \CmP(L_v)$ is a certain parabolic subgroup,
 and $f_v^{(Q_v)}\in C^{\infty}(L_v(\Q_v))$ is the parabolic descent of $f_v$ along $Q_v$
 (the $p$-adic parabolic descent is defined similarly to the one in the archimedean
case in Section~\ref{sec:HC}); $Q_v$ depends only on the Levi $L_v$ as a $\Q$-group.
\end{itemize}
This splitting formula follows from a general splitting formula for 
$(G,M)$-families~\cite{Ar81}*{\S6}, see also~\cite{Arthur:intro-trace}*{(18.7)}. The 
formula there is stated only in the case that $S$
is split into two non-empty disjoint subsets $S_1, S_2$. The above version follows from repeatedly applying that formula to the subsets $S_1$ and $S_2$ until one
arrives at sets consisting of a single place each.

\begin{lemma}\label{l:underline-L}
The following holds:
\begin{enumerate}[(i)]
\item For any $\underline{L}$ such that
 $d^G_M(\underline{L})\neq 0$, there are at most $\dim \Fma_M^G$ many places $v\in S$ such that $L_v\neq M$.
 \item  The number of $\underline{L}$ for which $d^G_M(\underline{L})\neq 0$
  is bounded by $c|S|^{\dim\Fma_M^G}$ with $c>0$ some constant depending only on $n$.
\end{enumerate}
\end{lemma}

\begin{proof}
For any $\underline{L}$ with $d_M^G(\underline{L})\neq0$, the map $\bigoplus_{v\in
S}\Fma_{M}^{L_v}\longrightarrow\Fma_M^G$ is an isomorphism.
We associate to $\underline{L}$ the multiset $\{L_v,\ v: L_v\neq M\}$.
It is of the form $\{M_1,\ldots, M_r\}$ with $M_1,\ldots, M_r\in \CmL(M)\backslash\{M\}$ such that $\bigoplus^r_{i=1}\Fma_{M}^{M_i}\longrightarrow\Fma_M^G$ is an isomorphism.
Then $r\le \dim \Fma_M^G$, which implies assertion (i).
Furthermore the number of such multisets is finite, thus bounded by some 
constant $c$ depending only on $n$, since the number of possible Levi $M\in \CmL$ is 
finite. Assertion (ii) follows 
by counting 
the number of $\underline{L}$ that give rise to a given multiset.
\end{proof}

If $v=p$ is non-archimedean, the parabolic descent of a Hecke operator in $C_c^\infty(G(\Q_p)// K_p)$ is a compactly supported function on $L(\Q_p)$ that is bi-invariant
under $K_p^{L}=L(\Z_p)$, the standard maximal compact subgroup in $L(\Q_p)$.  Similarly as in \S\ref{s:Hecke-operators}, we let $\tau_{p,\mu}^{L_p}\in C_c^{\infty}(L(\Q_p))$ denote the characteristic function of the double coset $K_p^{L} p^{\mu}K_p^{L}$.
\begin{lemma}\label{l:descent-Satake}
There is a constant $c_1>0$ depending only on $n$ such that the following holds.
Let $L$ be a standard Levi subgroup of $G$ and $Q\in \CmP(L)$.
Let $\xi=(\xi_1,\ldots,\xi_n)$ be a tuple of integers with $\xi_1\ge\ldots\ge\xi_n\ge0$. Then
\[
 \tau_{p,\xi}^{(Q)} =\sum_{\mu:\ \mu_1\le \xi_1} a_{\mu} \tau_{p,\mu}^{L},
\]
where
$\mu=(\mu_1, \ldots,\mu_n)$ runs over all tuples of integers $\mu_1\ge\ldots\ge\mu_n\ge0$ with $\mu_1\le\xi_1$,
and the coefficients $a_{\mu}\in \Q$
satisfy $|a_{\mu}|\le p^{c_1\xi_1}$.
\end{lemma}

\begin{proof}
	This is~\cite{ST11cf}*{p.69}, or~\cite{Matz}*{Lem.7.3}.
\end{proof}

Note the analogy between Lemma~\ref{l:descent-Satake} and Lemma~\ref{l:L1-descent} 
in the
archimedean case.

\begin{example}
If $\xi=0$, that is for $\tau_{p,0}$ is the characteristic function of $K_p$, then 
$\tau_{p,0}^{(Q)}$ equals the characteristic function of $K_p^{L}$.
\end{example}

We can now derive the following consequence of Arthur's splitting formula.
\begin{corollary}\label{c:bound-splitting}
There exist constants $c,c_1>0$ depending only on $n$ such that the following holds.
Suppose that $f_S=\prod_{v\in S} f_v\in C_c^\infty(G(\Q_S))$, and for each finite place $p\in S$, the function $f_p$ equals the Hecke operator $\tau_{p,\xi^p}$ associated to some $\xi^p \in \Z^n$ with $\xi^p_1
\ge \ldots \ge \xi^p_n \ge 0$. Then for every $\gamma_S \in M(\Q_S)$,
\[
|J_M^G(\gamma_S, f_S)|
\le c
\prod_{p\in S\backslash\{\infty\}}
 p^{c_1\xi^p_1}
\cdot
|S|^{\dim\Fma_M^G}
\cdot
\max_{\underline{L}}
\max_{\mu:~\mu^p_1\le\xi^p_1}
 \left( \tilde{J}_M^{L_{\infty}}(\gamma_{\infty},f_{\infty}^{(Q_{\infty})})
\prod_{p\in S\backslash\{\infty\}}   \widetilde J_M^{L_p}(\gamma_p, \tau_{p,\mu^p}^{L_p})\right).
\]
\end{corollary}
For the definition of the functionals $\widetilde J_M^{L_p}$ see 
\S\ref{sec:positive:distr} in the 
archimedean case, and \S\ref{sec:reduction} below in the non-archimedean case.

\begin{proof}
	If follows from the splitting formula~\eqref{eq:splitting:orb:int} and Lemma~\ref{l:underline-L} that
\[
|J_M^G(\gamma_S, f_S)|
\le c
|S|^{\dim\Fma_M^G}
\max_{\underline{L}}
\tilde{J}_M^{L_{\infty}}(\gamma_{\infty},f_{\infty}^{(Q_{\infty})})
\prod_{v\in S\backslash \{\infty\}} \tilde{J}_M^{L_p}(\gamma_p, \tau_{p,\xi^p}^{(Q_p)}).
\]
We deduce from Lemma~\ref{l:descent-Satake} that for every $p\in S\backslash \{\infty\}$,
\begin{equation*}\label{eq:estimate:descent}
\tilde{J}_M^{L_p}(\gamma_p, \tau_{p,\xi^p}^{(Q_p)})
\le p^{c_1\xi^p_1}\sum_{\mu^p:~\mu^p_1\le\xi^p_1} \tilde{J}_M^{L_p}(\gamma_p, \tau_{p,\mu^p}^{L^p})
\le
p^{c_1\xi^p_1} ((\xi^p_1)^n +1) \max_{\mu:~\mu^p_1\le\xi^p_1} \tilde{J}_M^{L_p}(\gamma_p, \tau_{p,\mu^p}^{L_p}). \qedhere
\end{equation*}
\end{proof}

\subsection{Sufficient size of $S$}
To state quantitatively how large the set $S$ has to be for the fine expansion for $J_\Fmo$ to hold, we proceed as follows.

For $\Fmo\in\CmO$ with semisimple representative $\sigma\in \Fmo$, let
\begin{equation}\label{So}
S_{\Fmo}:= \{\text{primes $p$ s.t. } |D^G(\sigma)|_p\neq1\}  \cup \{\infty\}.
\end{equation}
This definition is independent of the choice of $\sigma$ because the Weyl 
discriminant
$|D^G(\cdot)|_p$ is invariant by $G(\Q_p)$-conjugation.

\begin{lemma}\label{l:sufficient-fine}
For every equivalence class $\Fmo\in \CmO$, and every finite set of places $S$ containing $S_{\Fmo}$, Arthur's fine geometric expansion~\eqref{eq:fine:exp:o} holds.
\end{lemma}

\begin{proof}
There are two cases. In the first case, the equivalence class $\Fmo$ does not intersect $K^S = \prod_{p\not \in S} K_p$. Then $J_\Fmo(f)=0$ for any $f=(f_S\cdot f^S)_{G(\A)^1}$ with $f^S$ the characteristic function of $K^S\subset G(\A^S)$. Similarly $J^G_M(\gamma,f)=0$ for any $M\in \CmL$, $\gamma\in M(\Q)\cap \Fmo$. Hence equation~\eqref{eq:fine:exp:o} holds trivially, because both the left-hand side and right-hand side vanish.

In the second case, the equivalence class $\Fmo$ does intersect $K^S$.
The fine geometric expansion of $J_{\Fmo}(f)$ is established in
\cite{Ar86}*{Thm.8.1}, and we need to compare $S$ with the set of places constructed
in~\cite{Ar86}.
Namely, we are going to show that for every $p\not\in S$ the conditions (i)-(iv) of
\cite{Ar86}*{p.203} are satisfied.
Let 
\[
\CmE(\Fmo):=\CmE(\sigma) = (E_i,m_i)_{i\in \{1,\ldots,r\}} \in \kR^n_\Q.
\]
The notation $\CmE(\Fmo)$ is justified because $\CmE(\sigma)$  
is independent of the choice of semisimple
representative $\sigma \in \Fmo$.
Let $d_i=[E_i:\Q]$, and recall that $m_1d_1 + \cdots + m_r d_r=n$. 

Fix a prime $p\not \in S$. Since $p\not \in S_{\Fmo}$, we have $|D^G(\sigma)|_p=1$, and 
since $\Fmo$ intersects $K_p=G(\Z_p)$, we have $|\det(\sigma)|_p=1$. 
It follows from Lemma~\ref{l:discriminant}.(ii) that $p$ is 
unramified in $E_i$ for every $i\in 
\{1,\ldots,r\}$. Condition (i) of \cite{Ar86} is that 
$|D^{G}(\sigma)|_p=1$, which holds by 
construction. Conditions (ii)-(iv) of \cite{Ar86} are more subtle, and depend on a choice of a 
suitable semisimple representative $\sigma \in \Fmo$. Since $P_{\Fmo}\in \Z[1/S][X]$, 
Lemma~\ref{l:conjugate-integral} 
with $F=\Q$ and $R=\Z[1/S]$ the ring of $S$-integers, shows that we can choose 
$\sigma \in \Fmo\cap K^S$. Condition (iii) of \cite{Ar86}, which says that $\sigma K_p \sigma^{-1}=K_p$, 
then holds since $\sigma \in K_p=G(\Z_p)$.

Furthermore, by the proof of Lemma~\ref{l:conjugate-integral}, we can choose the 
semisimple representative $\sigma\in \Fmo\cap K^S$ in block 
diagonal form, that is 
$\sigma=\operatorname{diag}(\delta_1,\ldots,\delta_1,\ldots,\delta_r,\ldots,\delta_r)$, where each 
$\delta_i\in M_{d_i}(\Z[1/S])$ is a regular elliptic matrix in $\GL_{d_i}(\Q)$.
Moreover, $\delta_i\in \GL_{d_i}(\Z_p)$  and 
$|D^{\GL_{d_i}}(\delta_i)|_p=1$ for every $i\in \{1,\ldots,r\}$. 
Let $M_1$ be the standard Levi subgroup 
of type $(d_1,\ldots,d_1,\ldots,d_r,\ldots, d_r)$ such that $\sigma \in M_1(\Q)$ is regular 
elliptic.

Condition (ii) of \cite{Ar86} says that $K_{p,\sigma}:=K_p \cap C_G(\sigma,\Q_p)$ is a maximal 
compact 
subgroup of $C_G(\sigma,\Q_p)$, which holds by Lemma~\ref{l:Kottwitz}, and also that 
it is admissible with respect to $C_{M_1}(\sigma,\Q_p)$, that is, $K_{p,\sigma}$ 
corresponds to a special vertex in the Bruhat-Tits building of $C_G(\sigma,\Q_p)$ 
and this vertex belongs to the apartment associated to the maximal split torus 
of 
$C_{M_1}(\sigma,\Q_p)$, see \cite{Ar81}*{\S1}.

Let $E/\Q_p$ be the splitting field of $\sigma$, which is an unramified extension, 
in fact $E$ is the composite of the completions $E_{i,w}$ for all $i\in \{1,\ldots 
r\}$ and all 
places 
$w|p$.
We have that $\delta_i$ is split regular semisimple in $\GL_{d_i}(\CmO_E)$ with good 
reduction, namely 
$|D^{\GL_{d_i}}(\delta_i)|_E=1$.
Hence the final assertion of Lemma~\ref{l:Kottwitz} implies that there exists 
$k_i\in \GL_{d_i}(\CmO_E)$ such 
that $k_i \delta_i 
k_i^{-1}$ is diagonal.
We deduce that $k = \operatorname{diag}(k_1,\ldots,k_1,\ldots,k_r,\ldots,k_r)\in M_1(\CmO_E)$ is 
such that 
$k \sigma k^{-1}\in T_0(E)$.
This implies $k C_{M_1}(\sigma,E)k^{-1} = T_0(E)$.
The Levi subgroup $C_{G\otimes E}(k \sigma k^{-1})$ is semistandard because it 
contains $T_0$, hence it is of the form $M\otimes E$ for some semistandard Levi 
subgroup 
$M\subseteq G$.

In summary, we have that the pair $(C_{M_1}(\sigma)\otimes E,C_G(\sigma)\otimes E)$ of a torus 
inside a Levi is 
split and 
$k$-conjugate, for the above $k\in M_1(\CmO_E)$, to the pair $(T_0\otimes E,M\otimes E)$ of 
the 
maximal diagonal torus $T_0$, and a semistandard Levi subgroup $M\subseteq G$.

Let  $K_{E,\sigma}:=G(\CmO_E)\cap C_G(\sigma,E)$. Then $k K_{E,\sigma} k^{-1} = G(\CmO_E)\cap 
M(E) = M(\CmO_E)$ because $k\in M_1(\CmO_E) \subset G(\CmO_E)$. 
Since $M(\CmO_E)$ is an admissible maximal compact subgroup in $M(E)$ with respect 
to $T_0(E)$, we deduce by $k$-conjugation that $K_{E,\sigma}$ is an admissible 
maximal 
subgroup of $C_G(\sigma,E)$ with respect to $C_{M_1}(\sigma,E)$.
The corresponding special vertex $x$ of the Bruhat-Tits building of 
$C_G(\sigma,E)$ belongs to the appartment associated to the maximal split torus 
$C_{M_1}(\sigma,E)$. 
Since $E/\Q_p$ is unramified, and by construction of buildings by \'etale 
descent, the $\operatorname{Gal}(E/\Q_p)$-fixed points of the 
Bruhat-Tits building of $C_G(\sigma,E)$ equals the Bruhat-Tits building of 
$C_G(\sigma,\Q_p)$. 
It therefore follows that $x$ is a special vertex in the Bruhat-Tits building of 
$C_G(\sigma,\Q_p)$ and that it belongs to the appartment associated to the maximal 
split torus of $C_{M_1}(\sigma,\Q_p)$.
Since $x$ is fixed by $K_{p,\sigma} = K_{E,\sigma}\cap C_G(\sigma,\Q_p)$, we deduce that 
$K_{p,\sigma}$ is an admissible maximal compact 
subgroup of $C_G(\sigma,\Q_p)$ with 
respect to $C_{M_1}(\sigma,\Q_p)$, as asserted.

Condition (iv) of \cite{Ar86} says that for any $y\in G(\Q_p)$ and unipotent $\nu\in C_G(\sigma, 
\Q_p)$ 
such that $y^{-1}\sigma \nu y\in\sigma K_p$, we have $y\in C_G(\sigma, \Q_p)K_p$. This 
holds because it is a special case of Lemma \ref{lem:minimal:distance} below. 
Namely $\xi=0$ because $\sigma_p\in K_p$, and also $|D^G(\sigma)|_p=1$, thus there 
exists $k\in C_G(\sigma,\Q_p)$ such that $|k y|_{G(\Q_p)}=1$, i.e., $k y\in 
K_p$.
\end{proof}

\begin{remark}
For a given $f\in C_c^\infty(G(\A)^1)$, only finitely many $\Fmo$ contribute to $\Jgeom(f)$ so that the fine expansions for each $J_{\Fmo}(f)$ could be combined.
Assume that $f=\prod_v f_v$ is factorizable, and that $S$ is a sufficiently large finite set with respect to the support of $f$, namely $S$ should contain all
$S_{\Fmo}$ for all $\Fmo$ contributing to $\Jgeom(f)$, and all places $v$ where $f_v$ is 
not the characteristic function of the standard maximal compact $K_v$. Then
\begin{equation*}
\Jgeom(f)
= \sum_{M\in\CmL} \frac{|W^M|}{|W^G|} \sum_{\{\gamma\}} a^M(\gamma, S) J^G_M(\gamma, f),
\end{equation*}
where $\gamma$ now runs over a set of representatives for the $M(\Q)$-conjugacy classes in $M(\Q)$.
For our purposes, it will be however more direct to consider the fine expansion of $J_{\Fmo}(f)$ individually for every contributing $\Fmo$.
\end{remark}

\subsection{Global coefficients}
The global coefficients $a^M(\gamma, S)$ occurring in \eqref{eq:fine:exp:o}
are related to the global geometry of the Hitchin fibration, see~\cite{Hausel-Villegas} and~\cite{Chaudouard:comptage}.
It seems that $a^M(\gamma,S)$ can always be expressed in terms of derivatives of Artin $L$-functions, in which case precise estimates are established in~\cite{ST11cf}*{\S6.6}.
However such formulas are not well-understood, except in some special cases which we now
describe.
If $\gamma$ is semisimple and elliptic in $M(\Q)$, then by \cite{Ar86}*{Thm.8.2}
\[
a^M(\gamma,S)= \vol(C_M(\gamma,\Q)\backslash C_M(\gamma,\A)^1)
\]
which is therefore independent of $S$, and can be expressed in terms of Tamagawa numbers
and special values of Artin $L$-functions. If the semisimple part $\gamma_s$ is not
elliptic in $M(\Q)$, then $a^M(\gamma,S)=0$, again by \cite{Ar86}*{Thm.8.2}.
If the semisimple part $\gamma_s$ is elliptic in $M(\Q)$, then
$a^M(\gamma,S)=a^{M_{\gamma_s}}(\gamma_u,S)$ by
\cite{Ar86}*{(8.1)}.
Chaudouard~\cite{Chaudouard:comptage}
treats certain types of unipotent elements $\gamma_u$.

It is essential for us to treat a general $\gamma$.
We only require an upper bound for $a^M(\gamma,S)$ and the following result
from~\cite{Ma:coeff}  will suffice for our purpose.

\begin{proposition}\label{prop:global:coeff}
There exist $c,c_\glob>0$ depending only on $n$ such that the following holds. If 
$\gamma=\gamma_s\gamma_u\in M(\Q)$ has a characteristic polynomial with integral 
coefficients, then
\begin{equation}\label{eq:bound:coeff}
 |a^M(\gamma,S)|
\le c |\det(\gamma_s)^{n-1} D^G(\gamma_s)|_\R^{c_\glob}  |S|^{n-1} \max_{p\in S\backslash\{\infty\}} (\log p)^{n-1}.
\end{equation}
\end{proposition}

\begin{proof}
	Recall from~Section~\ref{sec:centralizers} the map $\CmE$ from semisimple conjugacy
	classes in $G(\Q)$ to $\kR_\Q^n$.
	Let $\CmE(\gamma_s)=(E_i, m_i)_{i\in I}$.
Let $P(X)$ be the characteristic polynomial of $\gamma$ and let $P(X)=\prod_{i\in
I}P_i(X)^{m_i}$ be its factorization into irreducible polynomials in $\Q[X]$. Thus
$E_i\simeq \Q[X]/P_i(X)$.
By Lemma~\ref{l:discriminant}, we have
\begin{equation}\label{eq:bound:discriminat1}
 \prod_{i\in I} |{\rm Disc}(P_i)|^{m_i}
 \le  |\det(\gamma_s)^{n-1} D^G(\gamma_s)|_\R.
\end{equation}
Let $M_1\subset M$ be the smallest Levi subgroup in which $\gamma_s$ is regular elliptic.
Then  $M_1\simeq \prod_{i\in I} (\GL_{d_i})^{m_i}$. Let $\gamma_{s,i}\in \GL_{d_i}(\Q)$, $i\in 
I$, denote the elliptic elements corresponding to $\gamma_s$ under this  
isomorphism. Then
$\prod_{i\in I} {\rm Disc}(P_i)^{m_i} = D^{M_1}(\gamma_s) \prod_{i\in 
I}\det(\gamma_{s,i})^{m_i(d_i-1)}$, which coincides with the discriminant denoted 
${\rm disc}^{M_1}(\gamma_s)$ in~\cite{Ma:coeff}.

By \cite{Ma:coeff}*{Cor.1.4},
there exist $c,a_{\rm glob}>0$ depending only on $n$ such that
\begin{equation}\label{eq:global:coeff:est1}
 |a^M(\gamma,S)|
\le c  |\det(\gamma_s)^{n-1} D^{G}(\gamma_s)|_\R^{a_{\rm glob}} \sum_{(s_v)_{v\in S}} \prod_{p\in S\backslash\{\infty\}}\left| \frac{\zeta^{(s_p)}_p(1)}{\zeta_p(1)}\right|,
\end{equation}
where the sum runs over all tuples $(s_p)_{p\in S\backslash \{\infty\}}$ of non-negative integers with $\sum s_p\le n-1$, and  $\zeta_p(s)=(1-p^{-s})^{-1}$ denotes the local Riemann zeta function, $\zeta_p^{(s_p)}$ denotes its $s_p$-th derivative.
We note that the normalization of measures in \cite{Ma:coeff} differs from 
our normalization by some power of the absolute discriminants $|D_{E_i}|$. Since 
$|D_{E_i}|\le |\operatorname{Disc}(P_i)|$, these factors have been absorbed in the 
above exponent $a_{\rm glob}$.

The number of the tuples $(s_p)_{p\in S\backslash \{\infty\}}$ in \eqref{eq:global:coeff:est1} is less than $|S|^{n-1}$. Also for each $p\in S\backslash\{\infty\}$, we have $|\zeta^{(s_p)}_p(1) \zeta_p(1)^{-1}|\ll (\log p)^{s_p}$ with implied constant depending only on $n$.  Hence
\[
  \sum_{(s_v)_{v\in S}} \prod_{p\in S\backslash\{\infty\}}\left| \frac{\zeta^{(s_p)}_p(1)}{\zeta_p(1)}\right|
  \ll |S|^{n-1} \max_{p\in S\backslash\{\infty\}} (\log p)^{n-1},
\]
which, together with \eqref{eq:bound:discriminat1} gives the assertion.
\end{proof}

\subsection{Contributing equivalence classes}
We record in this subsection some properties of the classes $\Fmo\in \CmO$ that 
contribute to the geometric side of the trace formula. Since we are only 
interested in the group $G=\GL_n$, we take the opportunity to provide a shorter 
treatment compared to~\cite{ST11cf}*{\S8}, and with improved estimates.

\begin{lemma}\label{l:upper:DGs}
	Let $p$ be a prime, and $\sigma$ be a semisimple element which is $G(\Q_p)$-conjugate
	to an element of $K_p p^\xi K_p$, with $\xi_1\ge\ldots\ge\xi_n$. Then
	\begin{equation}\label{contr-classes-det}
	|D^G(\sigma)|_p\le p^{n(n-1)(\xi_1-\xi_n)},\ \text{and} \quad  p^{-n\xi_1} \le |\det(\sigma)|_p\le p^{-n\xi_n}.
	\end{equation}
\end{lemma}
\begin{proof}
	Let $\overline{\Q}_p$ denote an algebraic closure of $\Q_p$. We extend $|\cdot|_p$ to
	$\overline{\Q}_p$ and denote the extension again by $|\cdot|_p$.  There exists a
	diagonal $\sigma'=\Mdiag(t_1,\ldots, t_n)\in T_0(\overline{\Q_p})$ which is
	$G(\overline{\Q}_p)$-conjugate to $\sigma$. We have
	\[
	D^G(\sigma)
	=D^G(\sigma')
	=\prod_{\alpha\in\Phi:\ \alpha(\sigma')\neq0} (1-\alpha(\sigma'))
	=\prod_{t_i\neq t_j} (1-t_j^{-1} t_i).
	\]
	Without loss of generality we may assume that $\xi_n= 0$, the characteristic polynomial of $\sigma$ has integral coefficients, and all the eigenvalues $t_j$ are integral, hence
	\[
	|D^G(\sigma)|_p
	= \prod_{j=1}^n  |t_j|_p^{-\#\{i: t_i\neq t_j \}} \prod_{i:t_i\neq t_j} |t_j-t_i|_p
	\le \prod_{j=1}^n |t_j|_p^{-\#\{i: t_i\neq t_j \}}.
	\]
	Moreover, $|t_1\cdot\ldots\cdot t_n|_p=|\det\sigma|_p=p^{-(\xi_1+\ldots+ \xi_n)}\ge p^{-n\xi_1}$. Using the integrality of the $t_j$'s again, we obtain
	\[
	|D^G(\sigma)|_p \le \prod_{j=1}^n |t_j|_p^{-(n-1)}  \le p^{n(n-1)\xi_1}. \qedhere
	\]

\end{proof}

Recall the function $X: G(\R)/A_G \to \Fma$ from Section~\ref{sec:notation}.
\begin{lemma}\label{lem:properties:contr:classes}
There is a constant $c_1 >0$ depending only on $n$, and for every $R\ge 1$ a 
constant $c\ge 1$ depending only on $R$ and $n$, such that the following holds.
For each prime $p$, let $\xi^p=(\xi^p_1,\ldots, \xi^p_n)$ be a tuple of integers with $\xi^p_1\ge\ldots\ge \xi^p_n\ge0$, and such that $\xi^p=0$ for all but finitely many $p$. Then the number of equivalence classes $\Fmo\in\CmO$ whose orbit under
$G(\A)$-conjugation intersects the subset
\[
\{g\in G(\R)\mid \|X(g)\|\le R\}\times\prod_{p<\infty} K_p p^{\xi^p} K_p\subset G(\A),
\]
is finite and bounded by $c \prod\limits_{p<\infty} p^{c_1 \xi^p_1}$.
Furthermore, for any such $\Fmo$, with semisimple representative $\sigma \in \Fmo$, we have
\begin{equation}\label{contr-classes-R}
 c^{-1} \prod_{p<\infty} p^{-n(n-1)\xi_1^p} \le |D^G(\sigma)|_{\R} \le c,
\end{equation}
and for every prime $p$,
\begin{equation}\label{contr-classes-p}
c^{-1} \prod_{q\neq p,\infty} q^{-n(n-1)\xi_1^q}  \le |D^G(\sigma)|_p \le p^{n(n-1)\xi_1^p}.
\end{equation}
\end{lemma}

\begin{proof}
The Weyl discriminant is invariant under multiplication by the center, thus $|D^G(g)|_\R$ is invariant under $A_G$. The upper-bound in assertion~\eqref{contr-classes-R} follows from the compactness of the set $\{g\in G(\R)\mid \|X(g)\|\le R\}$. Also the coefficients of the characteristic polynomial $P(T)=T^n+a_1T^{n-1}+\cdots + a_n$ of $\sigma$ satisfy
\[
\max_{1\le j \le n} |a_j|^{1/j} \le c |\det(\sigma)|^{1/n}_\R \le c \prod_{p<\infty} p^{\xi^p_1}.
\]
The assumption $\xi^p_n\ge 0$ for every prime $p$, implies that $a_j \in \Z$ for every $j$. Since $\Fmo$ is uniquely determined by the coefficients of $P(T)$, we deduce the bound on the number of equivalence classes, with $c_1=n(n+1)/2$.

The upper bound of~\eqref{contr-classes-p} is Lemma~\ref{l:upper:DGs}.
	Since $\sigma$ is rational,
\[|D^G(\sigma)|_p = |D^G(\sigma)|^{-1}_\R \prod_{q\neq p} |D^G(\sigma)|^{-1}_q
\]
by the product formula.
The implies the lower-bound of~\eqref{contr-classes-p}, and also the lower-bound of~\eqref{contr-classes-R}.
\end{proof}

We refer to the equivalence classes $\Fmo\in \CmO$ that satisfy the condition of the lemma as \emph{contributing classes}. This depends on a choice of $R$ and $\xi$. As before, we assume that $\xi^p_1\ge\ldots\ge\xi^p_n\ge0$ for all $p$, and $\xi^p=0$ for all but finitely many $p$.
\begin{corollary}
	There is a constant $c\ge 1$ depending only on $R$ and $n$ such that for every contributing class $\Fmo\in \CmO$, we have
\begin{equation}\label{contr-classes-max}
\max_{p\in S_{\Fmo}\setminus\{\infty\}} p 
\le  c \prod_{p<\infty} p^{n(n-1)\xi_1^p},
\end{equation}
and
\begin{equation}\label{prod-contr-classes}
\prod_{p\in S_{\Fmo}\setminus\{\infty\}} p \le 
c \prod_{p<\infty} p^{(n^2-n+1)\xi_1^p}.
\end{equation}
\end{corollary}
\begin{proof}

	Recall that $S_\Fmo$ is defined in~\eqref{So}.
	Let $p\in S_\Fmo\setminus \{\infty\}$. If $\xi^p\neq0$, we 
	trivially have 
	$p\leq p^{n(n-1)\xi_1^p}$, so \eqref{contr-classes-max} holds. If $\xi^p=0$, then 
	by 
	\eqref{contr-classes-p} we have $|D(\sigma)|_p\le 1$, hence $|D(\sigma)|_p\le p^{-1}$ 
	because $p\in S_\Fmo\smallsetminus \{\infty\}$. Hence $p\le |D(\sigma)|_p^{-1}\le 
	c\prod_{q\neq p,\infty} q^{n(n-1)\xi_1^q}$ by \eqref{contr-classes-p} so that 
	\eqref{contr-classes-max} follows.

If $p\not\in S_{\Fmo}$, then Lemma~\ref{l:upper:DGs} yields $1 \le 
|D(\sigma)|_p^{-1} p^{(n^2-n)\xi^p_1}$.
The same argument as above, together with Lemma~\ref{l:upper:DGs} yields the 
inequality 
\[
p \le 
|D(\sigma)|_p^{-1} p^{(n^2-n+1)\xi^p_1},
\]
for every prime $p\in S_{\Fmo}\smallsetminus \{\infty\}$.
Thus the product formula yields
\[
\begin{aligned}
\prod_{p\in S_{\Fmo}\setminus\{\infty\}} p 
&\le
\prod_{p\not\in S_\Fmo}
|D(\sigma)|^{-1}_p p^{(n^2-n)\xi^p_1}
 \prod_{p\in 
S_\Fmo \setminus \{\infty \}} 
|D(\sigma)|^{-1}_p p^{(n^2-n+1)\xi^p_1} 
\\
&\le|D(\sigma)|_\R \prod_{p<\infty} 
p^{(n^2-n+1)\xi^p_1}.
\end{aligned}
\]
This implies~\eqref{prod-contr-classes} in view of \eqref{contr-classes-R}.
\end{proof}

Recall that for $\Fmo\in \CmO$, we write $\CmE(\Fmo):= \CmE(\sigma)$ for some semisimple 
representative $\sigma\in \Fmo$. 
The characteristic polynomial of a contributing class $\Fmo$ has integral coefficients,
since $K_p p^{\xi^p}K_p$ is a subset of $M_n(\Z_p)$ for every prime $p$, because $\xi^n_p
\ge 0$.
Hence, we can always choose a semisimple representative $\sigma \in \Fmo \cap M_n(\Z)$ 
for a contributing class $\Fmo$.
\begin{corollary}
If $\Fmo$ is a contributing class and $\CmE(\Fmo)=(E_i, m_i)_{i\in I}$, then
\[
\prod_{i\in I} |D_{E_i}|^{m_i}
\le c \prod_{p<\infty} p^{n(n-1) \xi^p_1}.
\]
\end{corollary}

\begin{proof}
By Lemma~\ref{l:discriminant}, we have
\[
\prod_{i\in I} D^{m_i}_{E_i} \le |\det(\sigma)|^{n-1}_\R |D^G(\sigma)|_\R.
\]
The assertion then follows from~\eqref{contr-classes-det} and~\eqref{contr-classes-R}.
\end{proof}

\begin{corollary}
	There is a constant $c_2>0$ depending only on $n$, and a constant $c>0$ depending only on $R$ and $n$, such that for any contributing class $\Fmo$ and any $\gamma\in \Fmo$,
\begin{equation}\label{contr-classes-aM}
	|a^M(\gamma,S_{\Fmo}\cup \{p\mid \xi^p\neq0\})|
	\le c \prod_{p<\infty} p^{c_2\xi^p_1}.
\end{equation}
\end{corollary}
\begin{proof}
	This follows from~Proposition~\ref{prop:global:coeff}, combined 
	with~\eqref{contr-classes-det}, \eqref{contr-classes-R}, 
	\eqref{prod-contr-classes} and \eqref{contr-classes-max}.
	We can take $c_2=n(n-1)c_{\rm glob}$.
\end{proof}

\begin{lemma} For any contributing class $\Fmo$, and semisimple representative $\sigma\in \Fmo$
	\begin{equation}\label{contr-classes-Delta}
	\Delta^-(\sigma)\le c \prod_{p} p^{n(n-1)\xi^p_1},
	\end{equation}
	with $\Delta^-$ defined as before Lemma~\ref{lem:bound:supp}.
\end{lemma}
\begin{proof}
	By the product formula, we have
\[
\Delta^-(\sigma)=
\prod_{i\neq j,\ t_i\neq t_j}
\max(1, \prod_{p<\infty} |1-t_j^{-1}t_i|_p ).
\]
The proof is then similar to that of Lemma~\ref{l:upper:DGs}.
\end{proof}

Recall the definition of the archimedean and global test functions $f^{\mu}_{\pm}$, $F^{\mu}_\pm$ from \S\ref{sub:wave-test} and \S\ref{sub:global-test} respectively.
We are interested in the equivalence classes $\Fmo\in\CmO$ that contribute to the coarse expansion of $\Jgeom(F^\pm_\mu)$, that is such that $J_{\Fmo}(F^{\mu}_\pm)\neq0$.
The support of $f_\pm^\mu\in C_c^\infty(G(\R)/A_G)$ is included in $\{g:\ ||X(g)||\le 
R\}$ for some $R\ge 1$, which is independent of $\mu$.
Recall from~\S\ref{sub:global-test}
that $F^{\mu}_\pm=(f^{\mu}_\pm\cdot\prod_p \tau_{p})_{|G(\A)^1}$ satisfies $\tau_p=\tau_{p,\xi^p}$ for some $\xi^p=(\xi^p_1,\dots,\xi^p_n)$ with  $\xi^p_1\ge\ldots\ge\xi^p_n\ge0$.
Hence all equivalence classes contributing to the coarse geometric expansion satisfy the condition of Lemma~\ref{lem:properties:contr:classes}.
All the properties established in this section apply to any contributing class $\Fmo$ such that $J_{\Fmo}(F^\pm_\mu)\neq 0$.

\subsection{Bounding the geometric side of the trace formula}
We may now reduce the bound of the geometric side of the trace formula
to estimating local weighted orbital integrals. For archimedean places, we solved this problem in Part~\ref{partI}, and for non-archimedean places, this
will be established in the next Section~\ref{sec:p-adic}.

The following is the main technical result of the paper.
\begin{theorem}\label{th:bound-Jgeom}
	Assume $n\ge 3$.
	There exists a constant $c_3\ge 0$ depending only on $n$ and a constant $c>0$ 
	depending only on $n$ and the function $h$ used to define $f^{\mu}_\pm$, such 
	that the following holds.
	For every tuple $\xi=(\xi^p)_p$ of integers with $\xi_1^p\ge \ldots \ge \xi^p_n$, and every $\mu\in i\ka^*$,
	\[
	\left|\Jgeom(F^{\mu}_\pm)-\sum_{\text{unip.}\, \Fmo} J_{\Fmo}(F^{\mu}_\pm)\right|
	\le c \prod_p p^{c_3 (\xi^p_1-\xi^p_n)} (1+\|\mu\|)^{d-r-\frac12}
	\]
	where $\Fmo$ runs over the set of unipotent equivalence classes. We recall that $F^\mu_{\pm}:= (f^\mu_\pm\cdot \tau_\xi)|_{G(\A)^1}$.
\end{theorem}

\begin{proof}
Without loss of generality, we assume that $\xi^p_n=0$ for all $p$.
	From the coarse expansion~\eqref{eq:coarse:exp}, we need to give an upper-bound for the sum over non-unipotent classes $\Fmo$.
	The number of contributing classes is bounded by Lemma~\ref{lem:properties:contr:classes}.
	It then follows from the fine expansion~\eqref{eq:fine:exp:o} that the left-hand side is
\[
\le c \prod_{p<\infty} p^{c_2 \xi_1^p}
\max_{\Fmo} \max_M \max_{\{\gamma\}}
|a^M(\gamma,S)|
|J_M^G(\gamma_S,F^\mu_{\pm,S})|.
\]
Note that $S$ depends on the non-unipotent class $\Fmo$. We choose $S=S_\Fmo \cup \{ p | \xi^p\neq 0 \}$ and use the upper-bound~\eqref{contr-classes-aM} for $a^M(\gamma,S)$.
For the weighted orbital integral, we use Corollary~\ref{c:bound-splitting}, and 
also~\eqref{prod-contr-classes} to bound $|S|$, and obtain:
\[
\le c
\prod_{p<\infty} p^{c'\xi_1^p}
\max_{\Fmo,M,\{\gamma\}}
\max_{\underline{L}}
\max_{\mu: \mu_1^p\le \xi_1^p}
\tilde J_M^{L_\infty}(\gamma_\infty,f_{\pm}^{\mu,(Q_\infty)})
\prod_{p\in S_\Fmo,\\\text{or } \xi^p\neq 0}
\tilde J_M^{L_p}(\gamma_p,\tau^{L_p}_{p,\mu^p}).
\]
for some absolute constant $c'>0$.

For the archimedean weighted orbital integral we apply Proposition~\ref{p:J-sph}, together with~\eqref{contr-classes-Delta} to bound $\Delta^-(\gamma_s)$ and deduce
\[
\tilde J_M^{L_\infty}(\gamma_\infty,f_{\pm}^{\mu,(Q_\infty)}) \le
c
\prod_{p<\infty} p^{c_1 n(n-1)\xi_1^p}(1+\|\mu\|)^{d-r-\frac12}.
\]
For the non-archimedean weighted orbital integral, we apply Theorem~\ref{th:weight-orb-int-p-adic} below to deduce
\[
\prod_{p\in S_\Fmo,\\\text{or } \xi^p\neq 0} 
\tilde J_M^{L_p}(\gamma_p,\tau^{L_p}_{p,\mu^p})
\le  \prod_{p\in S_\Fmo} p^{a_\oi}
\prod_{\xi^p\neq 0} p^{a_\oi+b_\oi \xi_1^p} 
\prod_{p\in S_\Fmo,\\\text{or } \xi^p\neq 0} 
|D^G(\sigma)|_p^{-c_\oi},
\]
where $\sigma=\gamma_s$ is a semisimple representative for $\Fmo$.
The first product is bounded by \eqref{prod-contr-classes}.
We may extend the last product over all primes $p$ as for $p\not\in S_\Fmo$ we have 
$|D^G(\sigma)|_p=1$. Hence by the product formula, we 
get
\[
 \prod_{p\in S_\Fmo,\\\text{or } \xi^p\neq 0} \tilde J_M^{L_p}(\gamma_p,\tau^{L_p}_{p,\mu^p})
 \le  c  \prod_{p<\infty} p^{c_2\xi_1^p} \cdot |D^G(\sigma)|_\R^{c_\oi}
\]
which by \eqref{contr-classes-R} is bounded by a constant multiple of 
$\prod_{p<\infty} p^{c_2\xi_1^p}$.
Combining all the previous estimates we conclude the proof of the theorem.
\end{proof}

\subsection{Unipotent equivalence classes}
We say that an equivalence class $\Fmo\in \CmO$ is \emph{unipotent} if a semisimple 
representative $\sigma \in \Fmo$ is central, i.e, if $\sigma \in Z(\Q)$. Clearly $\sigma$ is 
unique, and unipotent classes are parametrized by $Z(\Q)$.
The unipotent class  $\Fmo$ corresponding to $\sigma \in Z(\Q)$ consists of all 
elements $\sigma u$, with $u\in \CmU(\Q)$, where $\CmU$ is the 
unipotent variety in $G$.

\begin{proposition}\label{p:unipotent-classes}
	There exists a constant $c_4\ge 0$ depending only on $n$ such that the 
	following holds.

(i) There exists $\delta>0$ depending only on $n$,
and $c>0$ depending only on $n$ and the function $h$,
 such that for 
every $\mu\in i\Fma^*$, tuple $(\xi^p)_p$, and unipotent equivalence class $\Fmo$ with 
semisimple representative $\sigma \in Z(\Q)$,
\[
\left|
J_\Fmo(F^\mu_{\pm}) -
\vol(G(\Q)\backslash G(\A)^1) f^\mu_+(1) \prod_{p<\infty}\tau_{p,\xi^p}(\sigma)
\right|
\le c (1+\|\mu\|)^{d-r - \delta} \prod_p p^{c_4 (\xi^p_1-\xi^p_n)}.
\]

(ii) If $\Omega\subset i\ka^*$ is a $W$-invariant and bounded measurable set with 
piecewise $C^2$-boundary, and if $h(0)=1$, there exists $c>0$ depending only on 
$n,h,\Omega$ such that for every $t\ge 2$, tuple $(\xi^p)_p$, and unipotent 
equivalence class $\Fmo$ 
 with 
semisimple representative $\sigma \in Z(\Q)$,
\[
\left|
\int_{t\Omega} J_\Fmo(F^\mu_{\pm})\, d\mu -
\tfrac{1}{2}\Lambda_\Omega(t) \prod_{p<\infty} \tau_{p,\xi^p}(\sigma)
\right|
\le c   t^{d-1} (\log t)^{\max(3,n)} \prod_p p^{c_4 (\xi^p_1-\xi^p_n)}.
\]
\end{proposition}

\begin{proof}
Using~\eqref{eq:coarse:exp}, the main term in (i) arises from $M=G$ and $u=1$,
and since we have $\vol(G(\Q)\backslash G(\A)^1)=a^G(\sigma, S)$,
$F^\mu_{\pm}(\sigma)=f^\mu_\pm(\sigma)\cdot\prod_{p<\infty} \tau_{p,\xi_p}(\sigma)$, and 
using~\eqref{eq:def:test:function},
\[
f^\mu_\pm(\sigma)=f^\mu_+(1)=
\frac{1}{|W|}
\int_{i\Fma^*}
\hat{h}(\lambda-\mu)
\left|\frac{\Tmc(\rho)}{\Tmc(\lambda)}\right|^{2} ~d\lambda,
\]
we obtain $\vol(G(\Q)\backslash G(\A)^1) f^\mu_+(1) \prod_{p<\infty}\tau_{p,\xi^p}(\sigma)$ as 
claimed.

Integrating $f^\mu_+(1)$ over $\mu\in t\Omega$, we obtain $\frac{1}{2\vol(G(\Q)\backslash 
G(\A)^1)}\Lambda_\Omega(t)$ as a main term for (ii) if $h(0)=1$. More precisely,
\[
\int_{t\Omega} f^\mu_+(1)\, d\mu - \frac{\Lambda_\Omega(t)}{2\vol(G(\Q)\backslash G(\A)^1)}
\ll t^{d-1}
\]
by \cite{DKV79}*{\S8}, \cite{LM09}*{(4.6)}.
The remainder term for (ii) is the sum of
\[
 \int_{t\Omega} \sum_{\{u\}\neq 1} a^G(\sigma u, S)  J_G^G(\sigma u,F^{\mu}_\pm)~d\mu,
\]
where $\{u\}$ runs over a set of representatives for the \emph{non-trivial}  
$G(\Q)$-conjugacy classes in $\mathcal{U}(\Q)$,
and of
\[
 \sum_{M\in \CmL,\ M\neq G} \frac{|W^M|}{|W^G|} \int_{t\Omega} \sum_{\{u\}} 
 a^M(\sigma u, 
 S)  J_M^G(\sigma u,F^{\mu}_\pm)~d\mu,
\]
where $\{u\}$ runs over a set of representatives of $M(\Q)$-conjugacy classes in 
$\CmU^M(\Q)$, and $S=S_{\Fmo}\cup \{p\mid \xi^p\neq 0\}$.
The global coefficients are bounded by~\eqref{contr-classes-aM}.
We have the 
factorization
\[
J_G^G(\sigma u,F^{\mu}_\pm)
=J_G^G( u,f_\pm^\mu) \prod_{p<\infty} J_G^G(\sigma u,\tau_{p,\xi^p})
\]
for the unweighted orbital integrals. For the weighted orbital integrals 
$J_M^G(\sigma u, F^\mu_\pm)$ we have a similar decomposition into local terms for 
which an upper bound is given in Corollary~\ref{c:bound-splitting}.
To obtain an upper bound for those integrals we argue as in the proof of 
Theorem~\ref{th:bound-Jgeom} but use 
Proposition~\ref{prop:bound-weighted-unip}.(ii) to 
bound the archimedean orbital integrals. For the $p$-adic integrals we can again 
use Theorem~\ref{th:weight-orb-int-p-adic}. This establishes assertion (ii).
The proof of (i) is similar, using Proposition~\ref{prop:bound-weighted-unip}.(i) 
instead.
\end{proof}

\begin{corollary}\label{cor:bound-Jgeom}
	Assume $n\ge 3$ and $h(0)=1$. Let $\delta(\xi^p)=1$ if $\xi^p$ is central and 
	$\delta(\xi^p)=0$ otherwise.
	Let $F^{\mu}_\pm=(f^{\mu}_\pm\cdot\tau_\xi)_{|G(\A)^1}$ and $\Omega\subset i\ka^*$ be as 
	before.
	 There exists $c>0$ depending only on $n,h,\Omega$ such that for every $t\ge1$, 
	 and tuple $(\xi^p)_p$,
	\[
	\left| \int_{t\Omega} \Jgeom(F^{\mu}_\pm)\, d\mu-
	\Lambda_\Omega(t) \prod_{p<\infty}\delta(\xi^p)
	 \right|
	\le c \prod_{p<\infty} p^{c_4 \xi^p_1}
	t^{d-1/2}.
	\]
\end{corollary}

\begin{proof}
	The contributing unipotent equivalence classes correspond to elements $\sigma 
	\in Z(\Q)\cap \prod\limits_{p<\infty} K_p p^{\xi^p} K_p$ (see 
	Lemma~\ref{lem:properties:contr:classes}). Thus there are at most two 
	contributing unipotent equivalence classes, and they differ by $\pm 1$.
	Combining Theorem~\ref{th:bound-Jgeom}, integrating it over $t\Omega$, and 
	Proposition~\ref{p:unipotent-classes}.(ii), we complete the proof.
\end{proof}

\section{Bound for \texorpdfstring{$p$}{p}-adic weighted orbital integrals}\label{sec:p-adic}
To complete the estimate of the global bound on the geometric side of the trace formula in Theorem~\ref{th:bound-Jgeom}, we need an upper bound for the $p$-adic weighted orbital integrals
$
|J_M^G(\gamma,f_p)| $,
for $\gamma\in M(\Q_p)$  and $f_p\in \CmH_p$.
The following is the main result of this section.

\begin{theorem}\label{th:weight-orb-int-p-adic}
There exist effective constants $a_{\oi},b_\oi,c_\oi \ge0$ depending only on $n$
such that the
following holds. For every prime $p$, tuple of integers $\xi=(\xi_1, \ldots,\xi_n)$ with
$\xi_1\ge \ldots \ge\xi_n$, $M\in \CmL$, and $\gamma\in M(\Q_p)$, we have
\[
\tilde J_M^G(\gamma,\tau_{p,\xi})
\le p^{a_\oi+b_\oi(\xi_1-\xi_n)} |D^G(\gamma_s)|_p^{-c_\oi},
\]
where we recall that $\tau_{p,\xi}\in\CmH_p$ is the characteristic function of the double coset $K_p p^\xi K_p$.
The integral is taken with respect to the  measures constructed in \S\ref{sec:measures:p-adic}.

\end{theorem}

We have  $\tilde J_M^G(\gamma,\tau_{p,\xi})= \tilde 
J_M^G(p^{-\xi_n}\gamma,\tau_{p,\xi'})$, if $\xi':=(\xi_1-\xi_n, \xi_2-\xi_n, 
\ldots,\xi_{n-1}-\xi_n, 0)$. Without loss of generality, we may therefore assume 
that $\xi_n\ge 0$, and shall do so whenever convenient.

\subsection{Preliminaries}
As in the archimedean case of Lemma \ref{lem:bound:supp}, the condition
$\tau_{p,\xi}(y^{-1}\sigma  u y)\neq0$ with $y\in G(\Q_p)$ and $u\in C_G(\sigma, \Q_p)$
unipotent
implies that $y$ and $u$ have to be contained in certain subsets.
To quantify this, we write $|g|_{G(\Q_p)}=p^{\lambda_1-\lambda_n}$ if $g\in G(\Q_p)$ with
$g\in K_p p^\lambda K_p$ and $\lambda=(\lambda_1,\ldots, \lambda_n)$,
$\lambda_1\ge\ldots\ge\lambda_n$, compare with \S\ref{sec:norm-groups} in the archimedean
case.

\begin{lemma}\label{lem:minimal:distance}
	There exist $b_1,c_1\ge 0$ depending only on $n$ such that the following holds.
	Suppose that $\sigma\in G(\Q_p)\cap p^{\xi_n} M_{n}(\Z_p)$ is semisimple, $y\in
	G(\Q_p)$ is arbitrary,  and $u\in C_G(\sigma, \Q_p)$ is a unipotent element such that
	$y^{-1}\sigma  u y\in K_p p^\xi K_p$ with $\xi_1 \ge \cdots \ge \xi_n$.
	Then there exists $\delta\in C_G(\sigma,\Q_p)$
	such that
	\[
	|\delta y|_{G(\Q_p)}, \, |\delta u\delta^{-1}|_{G(\Q_p)}
	\le p^{b_1(\xi_1-\xi_n)} |D^G(\sigma)|^{-c_1}_p.
	\]
\end{lemma}

\begin{proof}
	This is \cite{ST11cf}*{Lem.7.9} in the case $u=1$, and \cite{Matz}*{Cor.8.4}
	in general. The difference of notation with~\cite{Matz} is as follows: the norm
	$||\xi_F||_W$ there is dominated by $\xi_1-\xi_n$; the absolute value $|\log_p
	|D^G(\sigma)|_p|$ there was unnecessary because $|D^G(\sigma)|_p\le
	p^{n(n-1)(\xi_1-\xi_n)}$ by Lemma~\ref{lem:properties:contr:classes}; the constant
	$\delta$ there has been absorbed in the constants $b_1,c_1$, because if $\xi_1=\xi_n$
	and $|D^G(\sigma)|_p=1$, then the splitting field of $\sigma$ is tamely ramified
	(Lemma~\ref{l:DG=1} below);
	finally the integrality assumption on $\sigma$ was missing in the formulation
	of~\cite{Matz}*{Cor.8.4}.
\end{proof}

\begin{lemma}\label{l:DG=1}
	For every semsimple $\sigma \in K_p$, with $|D^G(\sigma)|_p=1$,
\begin{enumerate}[(i)]
	\item $\sigma$ splits over an unramified extension of $\Q_p$;
	\item for every $y\in G(\Q_p)$, we have $y^{-1} \sigma y \in K_p$ if and only if $y \in C_G(\sigma,\Q_p)K_p$;
	\item $J_G^G(\sigma,\tau_{p,0})= \CmO_\sigma(\tau_{p,0}) = 1$.
\end{enumerate}
\end{lemma}

\begin{proof}
	(i) Let $\CmE(\sigma)=(E_i,m_i)_{i\in I}$. Then $\sigma$ splits over the composition
	of the fields $E_i$.
	 By Lemma~\ref{l:conjugate-integral}, $\sigma$ is $G(\Q_p)$-conjugate to
	\[
	\Mdiag(\delta_1,\ldots,\delta_1,\delta_2,\ldots,\delta_2,\ldots,\delta_r,\ldots,\delta_r),
	\]
	where $\delta_i\in \GL_{d_i}(\Z_p)$ is regular elliptic.
We have $|\det(\delta_i)|_p=1$, and the characteristic polynomial $P_i$ is in $\Z_p[X]$.
Proceeding as in the proof of Lemma~\ref{l:discriminant}, there is a $\Z_p$-linear
injection of $\Z_p[\delta_i]\simeq \Z_p[X]/P_i$ into $\CmO_{E_i}$, and we deduce
\[
\prod_{1\le i\le r} |\operatorname{Disc}(\CmO_{E_i})|^{m_i}_p \ge
\prod_{1\le i\le r} |\operatorname{Disc}(P_i)|^{m_i}_p \ge |D^G(\sigma)|_p = 1.
\]
Hence each $E_i$ is an unramified extension of $\Q_p$.

	Assertion (ii) is a special case of Lemma~\ref{lem:minimal:distance}, with $u=1$ and $\xi=0$.
	See also~\cite{Kot86}*{Cor.7.3}.

We deduce from (ii) that $\CmO_{\sigma}(\tau_{p,0})=\int_{K_p}\tau_{p,0}(k^{-1}\gamma k)\, dk = 1$, which implies (iii).
\end{proof}

The following is a variant of the previous lemma in the split case.
\begin{lemma}[\cite{Matz}*{Prop.8.1, Cor.8.3}]\label{lem:minimal:distance:split}
	There exist constants $b_1, c_1\ge 0$ depending only on $n$ such that the 
	following holds. Let $p$ be a prime and $E/\Q_p$ be a finite extension, $\CmO_E$ 
	the ring of
	integers in $E$, and $K_E:=G(\CmO_E)$. Let $\sigma\in T_0(E)$ be such that the centralizer $G_\sigma(E)$ is the Levi component $M(E)$ of some standard parabolic subgroup $P(E)=M(E)U(E)$.
	Suppose $\delta\in G(E)$ is such that $\delta^{-1}\sigma \delta\in K_E \varpi_E^{\xi} K_E$ for some $\xi=(\xi_1,\ldots, \xi_n)$, $\xi_1\ge\ldots\ge\xi_n$, where $\varpi_E$ is a uniformizer. Let $\delta= muk$ denote the Iwasawa decomposition, with $m\in M(E)$, $u\in U(E)$, $k\in K_E$. Then
	\[
	|u|_{G(E)} \le p^{b_1 (\xi_1-\xi_n)} |D^G(\sigma)|_E^{-c_1}.
	\]
\end{lemma}

\subsection{Reduction to the semisimple conjugacy
classes}\label{sec:reduction}
The first step to prove
Theorem~\ref{th:weight-orb-int-p-adic} is to reduce the estimate of
$\tilde J_M^G(\gamma,\tau_{p,\xi})$ to the semisimple and unweighted (i.e., $M=G$) case:
	\begin{proposition}\label{lem:reduction}
	 There are constants $c,a_1,b_1,c_1,a_2,b_2,c_2 \ge0$ depending only on $n$ 
	 such that the following holds. For every prime $p$, $M\in \CmL$, $\gamma\in 
	 M(\Q_p)$, and every tuple $\xi=(\xi_1,\ldots,\xi_n)$ with $\xi_1\ge\ldots\ge\xi_n\ge0$, we 
	 have
	 \[
	  \tilde J_M^G(\gamma,\tau_{p,\xi})
	  \le c p^{a_1+b_1\xi_1} |D^G(\gamma_s)|_p^{-c_1} \max_{\mu:\, \mu_1\le\xi_1'}  J^G_G(\gamma_s,\tau_{p,\mu}),
	 \]
         where
$\xi_1':=  a_2 + b_2\xi_1 - c_2 \log_p|D^G(\gamma_s)|_p$,
          and
the maximum is taken over all tuples of integers $\mu=(\mu_1,\ldots , \mu_n)$ satisfying $\xi_1'\ge\mu_1\ge\ldots\ge\mu_n\ge0$.
Further, if $p>n$, we can take $ a_1=a_2=0$.
	\end{proposition}

\begin{proof}
Write $\tau=\tau_{p,\xi}$. If $H(\Q_p)\subseteq G(\Q_p)$ is a subgroup and $\delta\in
H(\Q_p)$, we write $H_\delta(\Q_p)$ for the centralizer of $\delta$ in $H(\Q_p)$ instead
of $C_H(\delta,\Q_p)$. The weighted orbital integral can be written as
\cite{Ar86}*{\S7}
\[
 J_M^G(\gamma,\tau)
 =|D^G(\gamma_s)|_p^{\frac12} \int_{G_{\gamma_s}(\Q_p)\backslash G(\Q_p)} \sum_{R\in \CmF^{G_{\gamma_s}}(M_{\gamma_s})} J_{M_{\gamma_s}}^{M_R}(\gamma_u, \Phi_{R,y})\, dy
\]
where
\[
 \Phi_{R,y}(m)= \delta_R(m)^{\frac12} \int_{K_{p,\gamma_s}}\int_{U_R(\Q_p)} \tau(y^{-1}\gamma_s k^{-1} mn k y) \nu_R'(ky) \, dn\, dk,\;\;\; m\in M_R(\Q_p),
\]
$\CmF^{G_{\gamma_s}}(M_{\gamma_s})$ denotes the set of Levi subgroups in $G_{\gamma_s}$ containing $M_{\gamma_s}$, $R=M_RU_R$ is the Iwasawa decomposition of $R\in \CmF^{G_{\gamma_s}}(M_{\gamma_s})$, and $\delta_R$ denotes the modulus function for $R(\Q_p)$.
$J_{M_{\gamma_s}}^{M_R}(\gamma_u,\cdot)$ denotes the weighted orbital integral inside of $M_R(\Q_p)$ instead of $G(\Q_p)$. Finally, $\nu_R'$ is a certain weight function defined by
\[
 \nu_R'
 = \sum_{\substack{Q\in\CmF(M):\\ Q_{\gamma_s}=R, \, \Fma_Q=\Fma_R}} v_Q'
\]
with $v_Q'$ defined similarly as in the archimedean situation \S\ref{sec:weighted:real}.
The unipotent weighted orbital integral inside the above integral can also be
written as
\[
 J_{M_{\gamma_s}}^{M_R}(\gamma_u, \Phi_{R,y})
 = \int_{V(\Q_p)}\int_{K_{p,\gamma_s}}\int_{U_R(\Q_p)} \tau(y^{-1}\gamma_s k^{-1} v n k y) \nu_R'(ky) w_{\CmO^{M_{\gamma_s}}}^{M_R}(v) \, dn\, dk\, dv
\]
where
\begin{itemize}
 \item $\CmO^{M_{\gamma_s}}\subset M_{\gamma_s}(\Q_p)$ is the $M_{\gamma_s}(\Q_p)$-conjugacy class of $\gamma_u$,
 \item $V$ is the unipotent radical of the parabolic subgroup $Q^R=LV\subseteq M_R$ such
 that the unipotent conjugacy class which is induced from $\CmO^{M_{\gamma_s}}$ to
 $M_R(\Q_p)$ is the Richardson orbit of $Q^R$,
 \item $w_{\CmO^{M_{\gamma_s}}}^{M_R}$ is a certain weight function on $V(\Q_p)$
 similarly as in the archimedean situation in \S \ref{subsec:weighted:unipotent}, see
 \cite{Ar88a}*{p.256}.
\end{itemize}
The functional $ \tilde J_M^G(\gamma,\tau)$ is then defined by replacing $\tau$,
$v_Q'$, and $w_{\CmO^{M_{\gamma_s}}}^{M_R}$ by their absolute values so that in
particular $ \left| J_M^G(\gamma,\tau)\right| \le \tilde J_M^G(\gamma,\tau)$.

The weight function $v_Q'$ satisfies a similar estimate as its archimedean counterpart in
Lemma \ref{lemma:bound:ss:weight}:
According to \cite{Matz}*{Cor.10.9},
\begin{equation}\label{eq:weight:fct:p}
 \left|v_Q'(g)\right|
\ll_n \left(1+ \log_p |g|_{G(\Q_p)}\right)^{n-1}
\end{equation}
for every $Q\in\CmF(M)$ and $g\in G(\Q_p)$.

We can estimate $\tilde J_M^G(\gamma,\tau)$ by using the integral formulas above.
Since $\tau=\tau_{p,\xi}$ and $\xi_n\ge 0$, we have that
$\tilde J_M^G(\gamma,\tau)$ is non-zero only if the characteristic polynomial of $\gamma$ has $p$-integral coefficients.
 Both sides of the inequality in Proposition~\ref{lem:reduction} are invariant
 if we replace $\gamma$ by a $M(\Q_p)$-conjugate, hence by Lemma
 \ref{l:conjugate-integral} we can assume that $\gamma_s\in M_n(\Z_p)$.

 By Lemma~\ref{lem:minimal:distance} we know in a quantitative way that the variable $y$  can not be too far away from the centralizer $C_G(\gamma_s,\Q_p)$, and that the unipotent variables $v,n$ are similarly bounded away from infinity. This allows us to separate the integration into a ``semisimple part" and a ``unipotent part". Moreover, using \ref{eq:weight:fct:p} we can bound each of the weights $\nu_R'(ky)$.
Write $N:= VU_R$. Then $N$ is the unipotent radical of a parabolic subgroup in $G_{\gamma_s}$ such that its Richardson orbit equals the unipotent conjugacy class induced from $\CmO^{M_{\gamma_s}}$ to $G_{\gamma_s}$.
We deduce that we can find constants $c, a'_1,b'_1,c'_1,a_2,b_2,c_2\ge0$
such that
\begin{multline*}
 \tilde J_M^G(\gamma,\tau)
 \le c p^{a_1'+b'_1\xi_1} |D^G(\gamma_s)|_p^{-c_1'}  \max_{\mu: \mu_1\le \xi'_1} \int_{G_{\gamma_s}(\Q_p)\backslash G(\Q_p)} \tau_{p,\mu}(g^{-1}\gamma_s g)\, dg \\
 \cdot \max_{R\in\CmF^{G_{\gamma_s}}(M_{\gamma_s})} \int_{N(\Q_p)\cap B_{\xi'_1}} \left|w_{\CmO^{M_{\gamma_s}}}^{M_R}(n)\right|\, dn,
\end{multline*}
where $\xi'_1$ is as defined in Proposition \ref{lem:reduction},
and $B_{\xi'_1}$ denote the set of matrices $g=(g_{ij})_{i,j}\in M_n(\Q_p)$ such that $|g_{ij}|_p\le p^{\xi'_1}$ for all $i,j$.

Note that in the last integral, we extended $w_{\CmO^{M_{\gamma_s}}}^{M_R}$
trivially to all of $N(\Q_p)$. By \cite{Matz}*{Lem.10.5}, this last integral is
$\le c p^{c_5 \xi'_1}$. This finishes the proof of Proposition \ref{lem:reduction}
with $a_1:=a'_1+c_5 a_2$, $b_1:=b'_1+c_5 b_2$, and $c_1:=c'_1+c_5 b_3$.
\end{proof}

\subsection{A bound for the unweighted semisimple orbital integral}\label{subsec:bound}
In this section we prove the uniform bound for unweighted $p$-adic semisimple orbital integrals $\CmO_\gamma(\tau_{p,\xi})$, that is, for $M=G$ and $\gamma=\gamma_s$, by using a modified version of the argument in \cite{ST11cf}*{\S7}.
We aim to show that there exist constants $a, b, c\ge 0$ depending only on $n$ 
such that the following holds. For every $p$, every semisimple $\gamma\in G(\Q_p)$ 
and every tuple of integers $\xi=(\xi_1, \ldots, \xi_n)$ with  $\xi_1\ge\ldots\ge\xi_n$,
 \begin{equation}\label{prop:p-adic}
|D^G(\gamma)|_p^{1/2} \CmO_\gamma(\tau_{p,\xi})=  J_G^G(\gamma, \tau_{p,\xi})
\le  p^{a+b(\xi_1-\xi_n)} |D^G(\gamma)|_p^{-c}.
 \end{equation}
In view of Proposition \ref{lem:reduction}, this will conclude the proof of
Theorem~\ref{th:weight-orb-int-p-adic}.
	Recall that the measure on $G_\gamma(\Q_p)$, which enters in the definition of
	$\CmO_\gamma$, has been chosen as in \S\ref{sec:measures:p-adic}.

We fix some notation.
We define
 \[
 X(\gamma, \xi)
 :=\{x K_p \in G(\Q_p)/K_p\mid x^{-1}\gamma x\in K_p p^\xi K_p\}.
 \]
Let $E/\Q_p$ be a field extension of smallest possible degree over which $\gamma$ splits, and define similarly
\[
X_E(\gamma, \xi)
:= \{xK_E\in G(E)/K_E\mid x^{-1}\gamma x\in K_E \varpi_E^{\xi} K_E\},
\]
where $\varpi_E\in\CmO_E$ is a uniformizing element, and $K_E=G(\CmO_E)$.
The groups $G_\gamma(\Q_p)$ and $G_\gamma(E)$ act by left multiplication on $X(\gamma,
\xi)$ and $X_E(\gamma, \xi_E)$, respectively.
Let $e_{E/\Q_p}$ be the ramification index of $E$ over $\Q_p$. Write
$\xi_E:=e_{E/\Q_p}\xi$. The inclusion of buildings $G(\Q_p)/K_p\subset G(E)/K_E$ induces
 an embedding $X(\gamma, \xi)\hookrightarrow X_E(\gamma, \xi_E)$.

The first step to prove \eqref{prop:p-adic} is the following, which reduces the estimate
to understanding the $G_\gamma(\Q_p)$-orbits in $X(\gamma,\xi)$.
 This step is incomplete in \cite{ST11cf}*{\S7.3}, because it was assumed that 
 $Z(G)\backslash G_\gamma$ was anisotropic, whereas~\cite{ST11cf}*{\S7.1} only 
 offered a 
 reduction to the case that $\gamma$ was elliptic, which is the weaker condition that 
 $Z(G)\backslash Z(G_\gamma)$ is 
 anisotropic, hence the assertion ``$I_\gamma(F)$ is a compact group" on line~8 
 of page~74 of~\cite{ST11cf} is incorrect.
\begin{lemma}\label{assum}
For every prime $p$, every $\xi=(\xi_1, \ldots, \xi_n)$ with $\xi_1\ge\ldots\ge\xi_n\ge 0$, and every semisimple $\gamma\in G(\Q_p)\cap M_n(\Z_p)$, which is block diagonal of the form $\Mdiag(\delta_1,\ldots,\delta_r)$ with each $\delta_i$ regular elliptic semisimple,
we have
\begin{equation}\label{eq:orb:p-adic}
 \CmO_\gamma(\tau_{p,\xi})
  \le |D^G(\gamma)|_p^{-1/2}|\det\gamma|_p^{-(n-1)/2} \sum_{\bar x \in G_\gamma(\Q_p)\backslash X(\gamma, \xi)} \vol_{G(\Q_p)} \left( K_p x K_p\right),
\end{equation}
where for every $\overline x$, we choose an arbitrary representative $x\in G_\gamma(\Q_p) \bar x 
K_p$ of the corresponding double-coset.
 \end{lemma}

\begin{proof}
By definition of $\CmO_\gamma(\tau_{p,\xi})$ and $X(\gamma, \xi)$ we have
\[
 \CmO_\gamma(\tau_{p,\xi})
 = \sum_{\bar x \in G_\gamma(\Q_p)\backslash X(\gamma, \xi)} \vol_{G_\gamma(\Q_p)\backslash G(\Q_p)} \left( 
 G_\gamma(\Q_p) \overline xK_p\right)
\]
where $\vol_{G_\gamma(\Q_p)\backslash G(\Q_p)}$ denotes the volume with respect to the quotient measure on $G_\gamma(\Q_p)\backslash G(\Q_p)$.

Applying Lemma~\ref{lem:quotient:measures} below with $H:=G_\gamma(\Q_p)$ acting by left translations on $X:=G(\Q_p)$, and $\cC:=H\cap K_p$ and $\cD:=x K_p$, we get
\[
\vol_{G_\gamma(\Q_p)}(\cC)\cdot
\vol_{G_\gamma(\Q_p)\backslash G(\Q_p)} \left( G_\gamma(\Q_p) \overline x K_p\right)
\le \vol_{G(\Q_p)} \left(\cC x K_p\right).
\]
By Proposition~\ref{p:bad-reduction}, $\cC$ is contained in a maximal compact subgroup of $G_\gamma(\Q_p)$ with index at most $|D^G(\gamma)|_p^{-1/2}|\det\gamma|_p^{-(n-1)/2}$.
Thus, our normalization of measures in~\S\ref{sec:measures:p-adic} implies that $\vol_{G_\gamma(\Q_p)}\left(\cC\right)\ge |D^G(\gamma)|_p^{1/2}|\det\gamma|_p^{(n-1)/2}$.
Moreover, we have that $\cC\subseteq K_p$ by construction, hence
\[
 \vol_{G(\Q_p)} \left(\cC x K_p\right)
 \le \vol_{G(\Q_p)} \left(K_p x K_p\right).
\]
(Note that the volume of $K_p x K_p$ depends on the choice of representative $x$, whereas
the double-coset $G_\gamma(\Q_p)\overline xK_p$, and a fortiori its volume, is independent of 
the choice.)
\end{proof}

\begin{lemma}\label{lem:quotient:measures}
	If $H$ is a closed unimodular subgroup with Haar measure $\Mvol_H$ of the 
	unimodular locally compact group $X$ with Haar measure $\Mvol_X$, then for any two 
	measurable subsets $\cC \subseteq 
	H$ and $\cD\subseteq X$, we have
  \[
   \vol_{H}\left( \cC\right) \cdot \vol_{H\backslash X} \left( H \backslash H\cdot \cD \right)
   \le \vol_{X} \left(\cC\cdot \cD\right).
  \]
 \end{lemma}
 \begin{proof}
Let $\chi_{\cC\cD}: X\longrightarrow \{0,1\}$ and $\chi_{H\cD}: H\backslash X\longrightarrow \{0,1\}$  be the characteristic functions of $\cC\cdot \cD\subseteq X$ and $H\backslash H\cdot \cD \subseteq H\backslash X$, respectively. By definition of the quotient measure we have
  \[
   \vol_{X}(\cC\cdot \cD) = \int_{X} \chi_{\cC\cD}(x)\, dx
   = \int_{H\backslash X} \int_H \chi_{\cC\cD}(h\cdot \bar x)\, dh\, d\bar x.
  \]
Since we have
\[
 \int_H \chi_{\cC\cD}(h\cdot \bar x)\, dh
 \ge \vol_H(\cC) \chi_{H\cD}(\bar x),\quad \forall \bar x\in H\backslash X,
\]
the assertion follows.
 \end{proof}
\begin{example}
	If $H$ is a finite group acting freely on a finite set $X$, and we use the counting measures, then $|\cC|\cdot |H\backslash H\cdot \cD|\le |\cC \cdot \cD|$. 
\end{example}

Since $\CmO_\gamma(\tau_{p,\xi})$ depends only on the $G(\Q_p)$-conjugacy class of $\gamma$, we 
can assume in establishing~\eqref{prop:p-adic} that $\gamma \in M_n(\Z_p)$ is block diagonal of 
the form $\Mdiag(\delta_1,\ldots,\delta_r)$ with each $\delta_i$ regular elliptic semisimple, by 
Lemma~\ref{l:conjugate-integral}.
In this way Lemma~\ref{assum} applies.
The right-hand side of~\eqref{eq:orb:p-adic} depends on the choice of a representative $x\in G_\gamma(\Q_p) \bar x K_p$ of the double coset $\bar x$.
We now choose an optimal representative $x_{{\min}}\in G_\gamma(\Q_p) \bar x K_p$ such that 
$|x_{\min}|_{G(\Q_p)} = \min_{x\in G_\gamma(\Q_p) \bar x K_p} |x|_{G(\Q_p)}$. This representative 
$x_{\text{min}}$ is in general not unique.

\begin{lemma}\label{l_vol-min}
	There are constants $c,b_1, c_1\ge 0$, depending only on $n$, such that for 
	every prime $p$,  every $\xi_1\ge \cdots \ge \xi_n$, every semisimple $\gamma\in 
	G(\Q_p)\cap p^{\xi_n}M_n(\Z_p)$, and $\bar x\in G_\gamma(\Q_p)\backslash X(\gamma,\xi)$,
		\[
	\vol_{G(\Q_p)}\left( K_p x_{\min} K_p\right) \le c |x_{\min}|^{n^2/4}_{G(\Q_p)},
	\]
	and
	\begin{equation}\label{eq:min:distance}
	|x_{\min}|_{G(\Q_p)}
	\le c p^{b_1 (\xi_1-\xi_n)} |D^G(\gamma)|_p^{-c_1}.
	\end{equation}
\end{lemma}

\begin{proof}
	Inequality~\eqref{eq:min:distance} follows from Lemma~\ref{lem:minimal:distance},
	with the same constants $b_1,c_1\ge 0$.
	Indeed, starting with an arbitrary representative $x\in G_\gamma(\Q_p) \bar x K_p$, there exists $\delta\in G_\gamma(\Q_p)$ such that $|\delta x|_{G(\Q_p)}$ satisfies the inequality, and we have $|x_{\min}|_{G(\Q_p)} \le |\delta x|_{G(\Q_p)}$ by construction of $x_{\min}$.

	Let $\nu=(\nu_1,\ldots,\nu_n)$ with  $\nu_1\ge\ldots\ge \nu_n$ be such that $x_{\min} \in K_p p^\nu K_p$. Then $p^{\nu_1-\nu_n}= |x_{\min}|_{G(\Q_p)}$ and $\vol_{G(\Q_p)}\left( K_p x_{\min} K_p\right)
	=\vol_{G(\Q_p)}\left( K_p p^\nu K_p\right)$ so that Lemma~\ref{lem:degree} below gives the first assertion.
\end{proof}

Recall that if $g\in G(\Q_p)$ with $g\in K_p p^\lambda K_p$, $\lambda=(\lambda_1,\ldots, 
\lambda_n)\in\Z^n$, we write $|g|_{G(\Q_p)} = p^{\max_k \lambda_k - \min_k
	\lambda_k}$. If $g\in G(E)$, we define $|g|_{G(E)}$ similarly, namely if $g\in K_E 
	\varpi_E^\lambda K_E$, then $|g|_{G(E)}=|\varpi_E|_E^{-(\max_k \lambda_k - \min_k 
	\lambda_k)}= p^{\frac{1}{e_{E/\Q_p}}(\max_k \lambda_k - \min_k \lambda_k)}$. Note that 
	$|g|_{G(E)}=|g|_{G(\Q_p)}$ if $g\in G(\Q_p)$.
We have that for every $g=(g_{ij})\in G(E)$,
\begin{equation}\label{gij}
|g_{ij}|_E
\le
p^{-\frac{1}{e_{E/\Q_p}} \min_k \lambda_k}
\le |\det(g)|^{\frac1n}_E \cdot |g|^{1-\frac1n}_{G(E)},
\quad \forall i,j=1,\ldots, n.
\end{equation}

 \begin{lemma}\label{l_number-terms}
There exist $b_1, c_1\ge 0$ depending only on $n$ such that for  every $\xi_1\ge \cdots 
\ge
\xi_n$ and for every semisimple $\gamma\in G(\Q_p)\cap p^{\xi_n} M_n(\Z_p)$ with
splitting field $E/\Q_p$, we have
 \begin{equation}\label{eq:ineq1}
 \# \left( G_\gamma(\Q_p)\backslash X(\gamma, \xi) \right)
 \le \# \{ u\in U(E)/U(E)\cap K_E\mid |u|_{G(E)}\le  p^{b_1(\xi_1-\xi_n)}|D^G(\gamma)|_p^{-c_1}\}.
  \end{equation}

  \end{lemma}

\begin{proof}

We can find $\sigma\in T_0(E)$ and $y\in G(E)$ such that $y^{-1}\sigma y=\gamma$.
By changing $y$ if necessary, we can assume that
	$G_\sigma(E)=M(E)$ with $M$ the Levi component of some standard parabolic subgroup $P=MU\subseteq G$.
We get an injective map $X(\gamma, \xi)\longrightarrow X_E(\sigma, \xi_E)$ given by $xK_p\mapsto yx K_E$. It is therefore equivalent to estimate the number of points in $M(E)\backslash X_E(\sigma, \xi_E)$.

An element $\delta \in M(E)\backslash X_E(\sigma, \xi_E)$ is uniquely determined by $u\in U(E)/U(E)\cap K_E$, which appears in its Iwasawa decomposition $\delta=muk$.
By Lemma~\ref{lem:minimal:distance:split}, we have
\[|u|_{G(E)}  \le p^{b_1  (\xi_1-\xi_n)} |D^G(\sigma)|_E^{-c_1}.
\]
Since $|D^G(\sigma)|_E = |D^G(\gamma)|_p$, this proves the assertion.
\end{proof}

\begin{proof}[Proof of~\eqref{prop:p-adic}]
Without loss of generality we may assume that $\xi_n=0$.
Lemma~\ref{assum} implies
\[
J^G_G(\gamma,\tau_{p,\xi})
\le |\det(\gamma)|^{-(n-1)/2}_p 
\sum_{\bar x \in G_\gamma(\Q_p)\backslash X(\gamma, \xi)} \vol_{G(\Q_p)} \left( K_p x_{\mathrm{min}} 
K_p\right).
\]
By Lemma~\ref{l:upper:DGs}, we have $|\det(\gamma)|^{-1}_p \le p^{n\xi_1}$.
Lemma~\ref{l_vol-min} bounds the volume terms in the inner sum, so that 
it remains to estimate the number of elements in the quotient
	$G_\gamma(\Q_p)\backslash X(\gamma, \xi)$. 
By Lemma~\ref{l_number-terms}, it suffices to estimate the right hand side 
of~\eqref{eq:ineq1}.
Since $u\in U(E)$ satisfies $\det(u)=1$ and $|u|_{G(E)}\le p^{b_1\xi_1}|D^G(\gamma)|_p^{-c_1}$, and
the number of $x\in E/\CmO_E$ with $|x|_E\le R$ is bounded by a constant multiple of $R$,
the asserted estimate follows from~\eqref{gij} and counting all the possible matrix
entries of $u$.
\end{proof}

\subsection{Example: regular semisimple orbital integrals}
For unweighted regular semisimple orbital integrals one can give precise estimates as follows.

For $G=\GL(2)$ the orbital integrals can be computed explicitly for a general semisimple element, see~\cite{book:Langlands:bc,Kot05}.
\begin{lemma}
For every prime $p$, every $\xi_1\ge \xi_2$, and every regular semisimple $\gamma \in \GL(2,\Q_p)$,
\[
J_G^G(\gamma,\tau_{p,\xi})=
|D^G(\gamma)|_p^{1/2} \CmO_{\gamma}(\tau_{p,\xi})
\le 4p^{\xi_1-\xi_2}.
\]
\end{lemma}
\begin{proof}
	Recall that $\tau_{p,\xi}\in\CmH_p$ is the characteristic function of $K_p\Mdiag(p^{\xi_1}, p^{\xi_2})K_p$. Without loss of generality, we may assume that $\gamma$ is $\GL(2,\Q_p)$-conjugate to an element of $K_p p^\xi K_p$, since otherwise  $J_G^G(\gamma,\tau_{p,\xi})=0$.

	If $\gamma$ splits over $\Q_p$, that is, if its eigenvalues are elements of $\Q_p$, then by~\cite{Kot05}*{(5.8.4), (5.8.5)}
	\[
	J_G^G(\gamma, \tau_{p,\xi})=|D^G(\gamma)|_p^{1/2} \CmO_{\gamma}(\tau_{p,\xi})
	=
	\begin{cases}
	1
	&\text{if }\xi_1=\xi_2,\\
	p^{\xi_1-\xi_2}(1-p^{-1})
	&\text{if } \xi_1>\xi_2.
	\end{cases}
	\]
	If $\gamma$ does not split over $\Q_p$, there is a quadratic extension $E/\Q_p$ over which $\gamma$ splits. Then 
	\[
	 J_G^G(\gamma, \tau_{p,\xi})=|D^G(\gamma)|_p^{1/2} \CmO_{\gamma}(\tau_{p,\xi})
	 = \vol(Z(\Q_p)\backslash G_\gamma(\Q_p))^{-1} |D^G(\gamma)|_p^{1/2} \int_{Z(\Q_p)\backslash G(\Q_p)} \tau_{p,\xi}(g^{-1}\gamma g)\, dg,
	\]
and the latter integral equals by \cite{Kot05}*{\S5.9}
	\[
	\begin{cases}
	1+ 2\frac{1-|D^G(\gamma)|_p^{1/2}}{p-1}
	&\text{if }E/\Q_p\text{ is unramified and }\xi_1=\xi_2,\\
	(1+p^{-1})p^{\xi_1-\xi_2}
	&\text{if }E/\Q_p\text{ is unramified and } \xi_1> \xi_2,\\
	2+ 2\frac{1-|D^G(\gamma)|_p^{1/2}}{p-1}
	&\text{if }E/\Q_p\text{ is ramified and }\xi_1=\xi_2,\\
	2p^{\xi_1 - \xi_2}
	&\text{if }E/\Q_p\text{ is ramified and }\xi_1 > \xi_2.
	\end{cases}
	\]
	The inverse volume $\vol(Z(\Q_p)\backslash G_\gamma(\Q_p))^{-1}$ equals the 
	discriminant 
	$|\operatorname{Disc}(\CmO_E)|_p \le1$ 
	of $E/\Q_p$. 
	The lemma follows.
	Also the constant $4$ is sharp, since it is achieved for $\xi_1=\xi_2=0$, $E/\Q_p$ ramified quadratic extension, and $\gamma \in \CmO_E^\times \subset K_p$ with $|D^G(\gamma)|_p\to 0$.
\end{proof}

\begin{example}
	If $|D^G(\gamma)|_p=1$ and $\gamma\in K_p$ is semisimple, then $E$ is either $\Q_p\times \Q_p$ or an unramified quadratic extension of $\Q_p$ and $\CmO_\gamma(\tau_{p,0})=1$ by Lemma~\ref{l:DG=1}. This is consistent with the formulas given in the proof of the above lemma.
	\end{example}

For general $G=\GL(n)$, and $\xi=0$, we deduce the following from results of 
Yun~\cite{Yun13}.
\begin{proposition}\label{prop:reg:ss:orb:int}
	For every $\epsilon >0$, there exists a constant $c(n,\epsilon)>0$ depending only on $n$ and $\epsilon$ such that for every prime $p$ and every $\gamma\in G(\Q_p)$ semisimple, which is either regular or splits over $\Q_p$,
\[
 J_G^G(\gamma,\tau_{p,0})
=
|D^G(\gamma)|_p^{1/2} \CmO_{\gamma}(\tau_{p,0})
\le c(n,\epsilon) |D^G(\gamma)|_p^{-\epsilon},
\]
where $\tau_{p,0}\in \CmH_p$ denotes the characteristic function of $K_p$.
\end{proposition}

\begin{proof}
Let $M\subseteq G$ be the smallest $\Q_p$-split Levi subgroup such that $C_G(\gamma,\Q_p)\subseteq M(\Q_p)$. Conjugating $\gamma$ by an element of $K_p$ if necessary, we can assume that $M$ is a standard Levi subgroup. By parabolic descent we have
\[
 J_G^G(\gamma,\tau_{p,0})
=J_M^M(\gamma,\tau_{p,0}^{(P)})
\]
for any $P=MU\in\CmP(M)$. Now for any $m\in M(\Q_p)$ we have
\begin{align*}
\tau_{p,0}^{(P)}(m)
& = \delta_P(m)^{1/2} \int_{U(\Q_p)} \int_{K_p} \tau_{p,0}(k^{-1} m u k)~dk~du
  = \delta_P(m)^{1/2} \int_{U(\Q_p)} \tau_{p,0}( m u )~du\\
& =\begin{cases}
    1						&\text{if } m\in K_p\cap M(\Q_p)=:K_p^M,\\
    0						&\text{else},
   \end{cases}
\end{align*}
that is, $\tau_{p,0}^{(P)}=\tau_{p,0}^M$ is the characteristic function of $K^M_p$. Here for the last equality we used that for any $u\in U(\Q_p)$ we have $\tau_{p,0}( m u )=0$ unless $m\in K_p\cap M(\Q_p)$.

By assumption, $\gamma$ is either regular elliptic or central in $M(\Q_p)$.  In the latter case, we trivially have
\[
 J_G^G(\gamma,\tau_{p,0})
= J_M^M(\gamma, \tau_{p,0}^M)
= \tau_{p,0}^M(\gamma)
= \tau_{p,0}(\gamma) \le 1.
\]

In the former case that $\gamma$ is regular elliptic in $M(\Q_p)$, we apply Yun's
estimate on unweighted regular semisimple orbital
integrals~\cite{Yun13}*{Thm.1.5}.

Without loss of generality, we may assume that $\gamma\in K_p$ is regular elliptic, 
so $M=G$.
Let $R=\Z_p[X]/P_\gamma(X)$, and $\delta$ denote the length of the $\Z_p$-module 
$\CmO_{E}/\Z_p[\gamma]$.
We have $p^\delta \le |D^G(\gamma)|_p^{-1/2}$ by Lemma~\ref{lem:length}.
In the notation of~\cite{Yun13}*{\S1.4}, we have
\[
O_\gamma(\tau_{p,0})
\le  p^{-d \delta} M_{\delta,r}(p^{d}),
\]
where $d,r\in \Z_{\ge 1}$ are certain invariants of $P_\gamma(X)$, and $M_{\delta,r}$ is 
a polynomial of degree $\delta$. We find
\[
 M_{\delta,r}(x) \le M_{\delta,\delta+1}(x)
 \le 2 \delta p(\delta) x^{\delta}
\]
for every $x\ge 1$, where $p(\delta)$ denotes the number of partitions of the 
integer $\delta$. Since $p(\delta)\ll \exp(\pi \sqrt{2\delta/3})\le c(\eps) \exp(\eps \delta)$ for 
every $\eps>0$ and $\delta\in \Z_{\ge 1}$ by the Hardy--Ramanujan asymptotic, we deduce
\[
J_G^G(\gamma,\tau_{p,0})
 \le 2 \delta p(\delta)
 \le 2 c(\eps)^n \exp(\eps  \delta)
 \le  2c(\eps)^n |D(\gamma)|^{-\eps/2}_p.
\]
This concludes the proof of the proposition.
\end{proof}

\section{Spectral side and conclusion of proof of Theorem \ref{thm:weyl}}\label{sec:spectral}

\subsection{Spectral side}
Let $\Pi_{\text{disc}}(G(\A)^1)$ (resp.\ $\Pi_{\text{cusp}}(G(\A)^1)$) denote the set of 
irreducible unitary representations occurring in the discrete (resp.\ cuspidal) 
part of $L^2(G(\Q)\backslash G(\A)^1)$. 
For $\pi\in\Pi_{\text{disc}}(G(\A)^1)$ with $\pi_\infty$ spherical we denote by 
$\lambda_{\pi_\infty}\in i\ka^*$ the infinitesimal character of $\pi_{\infty}$.

\begin{proposition}
Fix $h\in C_c^\infty(\mathfrak a)^W$. For every $\mu\in i\mathfrak a^*$,
\begin{equation}\label{bound-Jspec}
J_{\rm spec}((f^\mu_{\pm}\cdot \tau_0)|_{G(\A)^1}) =  \Jgeom((f^\mu_{\pm}\cdot 
\tau_0)|_{G(\A)^1}) 
\ll_{h} (1+||\mu||)^{d-r},
\end{equation}
where $f^\mu_{\pm}$ is defined in \S\ref{sub:global-test}, and moreover
\[
 J_{\rm disc}((f^\mu_{\pm}\cdot \tau_0)|_{G(\A)^1}) \ll_{h} (1+||\mu||)^{d-r}.
\]
\end{proposition}

\begin{proof}
The first estimate is~\eqref{simple_geom_est}, which is an improvement on 
\cite{LM09}*{Prop.4.2} in that we have removed the requirement on the small 
support of 
$h$.
The second estimate follows from the first estimate~\eqref{bound-Jspec} in the 
same way as for the proof~\cite{LM09}*{Prop.4.3}, via
an inductive procedure on the non-discrete spectrum which is established in 
\cite{LM09}*{\S6}, and \cite{Muller07}*{\S5}, see also the proof of 
Proposition~\eqref{prop:spectral} below.
\end{proof}

\begin{proposition}\label{lem:local_upper}
 For $R\in \R_{\ge1}$ and $\mu\in i\ka^*$, the number of 
 spherical, everywhere unramified $\pi\in\Pi_{\text{disc}}(G(\A)^1)$    with 
 $\|\lambda_{\pi_\infty}-\mu\|\le R$ satisfies
 \[
 \left|
\{
\pi\in \Pi_{\text{disc}}(G(\A)^1),\
\|\lambda_{\pi_\infty}-\mu\|\le R,\
\pi_\infty^{K^\circ_\infty} \neq 0,\
\pi_f^{K_f}\neq 0
\}\right|
 \ll 
R^{r}
\prod_{\alpha>0} 
(R+|\langle\alpha,\mu\rangle|)
 \]
where the implied constant is independent of $R$ and  $\mu$ (it depends only on 
$n$), and the product runs 
over all positive roots with respect to $(T_0,U_0)$.
\end{proposition}

\begin{proof}
We can directly use the results of~\cite{LM09} because we are free to pass to a 
principal congruence subgroup $K_f(3)\subset K_f$ (since $\pi_f^{K_f}\neq 0 \implies 
\pi_f^{K_f(3)}\neq 0$), for which \cite{LM09}*{Prop.4.5} applies.
Alternatively, one can verify 
that the proof of~\cite{LM09}*{Prop.4.5} only relies upon the 
above bound~\eqref{bound-Jspec} for $J_{\rm spec}$, hence we could also repeat their 
proof in our present setting.
\end{proof}

We need to know how the spectral side of the trace formula behaves for our 
family of test
functions $f^\mu_\pm \cdot \tau$ from \S\ref{sub:global-test}.  We assume from now on
that the fixed function $h$ used to construct $f_\pm^\mu$ in 
\eqref{eq:def:test:function}
satisfies $h(0)=1$.

\begin{proposition}\label{prop:spectral}
	There exists $c>0$ depending only on $n$ and $h$ such that for every $\tau\in\CmH$ with $|\tau|\le 1$, and every $\mu\in i\ka^*$,
	\[
	\left| J_{\text{spec}}\left((f^{\mu}_\pm\cdot\tau)_{|G(\A)^1}\right)
	- \sum_{\pi\in\Pi_{\text{cusp}}(G(\A)^1)}  \Mtr \pi(f^\mu_\pm \cdot \tau)\right|
	\le c \|\tau\|_{L^1(G(\A))} (1+\|\mu\|)^{d-r-1} (\log_+ \|\mu\|)^{\max(3, n)},
	\]
	 where $\log_+(x) :=\max(1,\log x)$.
\end{proposition}

\begin{proof}
The method is to reduce to the (local) estimation of the operator norm of $f^{\mu}_\pm\cdot\tau$ acting on certain induced representations, and the (global) estimation of the discrete non-cuspidal spectrum for the test function $(f^{\mu}_\pm\cdot \tau_0)_{|G(\A)^1}$  for $\tau_0$ the characteristic function of the maximal compact subgroup $K_f$.
Up to some changes we shall follow the proof of~\cite{LM09}*{Prop.4.3}, 
and of~\cite{Matz}*{Prop.15.1 and Lem.16.2}.

For a test function $f\in C_c^\infty(G(\A)^1)$ write 
$J_{\text{disc}}(f)=\sum_{\pi\in\Pi_{\text{disc}}(G(\A)^1)} \Mtr \pi(f)$ for the 
contribution of the discrete spectrum to the trace formula. By~\cite{MS:gln} the 
difference $J_{\text{spec}}(f)-J_{\text{disc}}(f)$ can be written as a finite linear 
combination of distributions $J_{M,s}(f)$ with $M\in\CmL(T_0)$ running over all 
semistandard Levi subgroups $\neq G$, and $s\in M\backslash {\rm Norm}_G(M)$ with ${\rm 
Norm}_G(M)$ the  normalizer of $M$ in $G$. We first show that for each such $M$ 
and $s$ the absolute value $\left| J_{M,s}(f^{\mu}_\pm\cdot\tau)\right|$
is bounded by $\|\tau\|_{L^1(G(\A))} (1+\|\mu\|)^{d-r-1} (\log_+ \|\mu\|)^{\max(3,
n)}$. If $\tau=\tau_0$, this was proven in~\cite{LM09}*{\S6} and we generalize their
argument. The absolute value $\left| J_{M,s}(f^{\mu}_\pm\cdot\tau)\right|$ is bounded from
above by
\begin{equation}\label{eq:spectral1}
\sum_{\pi\in\Pi_{\text{disc}}(M(\A)^1)} \int_{i\ka_L} \left|  \widehat{h}(\lambda+\lambda_{\pi_\infty}-\mu)\Mtr\left(\FmM_L(P,\pi,\lambda) M_{P|P}(0,s, \pi) \rho(\pi,\tau) \right) \right| \, d\lambda
\end{equation}
where $L\in\CmL(M)$ is a certain Levi group (determined by $M$ and $s$), $P=MU\in\CmP(M)$, $\FmM_L(P,\pi,\lambda)$ and  $M_{P|P}(0,s, \pi)$ are certain intertwining operators, and $\rho(\pi,\cdot)$ denotes the right regular representations on the $K_\infty$-fixed part of the $\pi$-isotypical component in the space of automorphic forms $\CmA_\pi(P)$ on $U(\A)M(\Q)\backslash G(\A)$.  Note that the image of $\rho(\pi,\tau)$ consists only of  $K$-invariant vectors, i.e, $ \rho(\pi,\tau)= \Mtr \pi^{K_\infty}(\tau^{(P)}) (\Pi_{K})_{|\CmA_\pi(P)}$ with $\Pi_K$ the projection onto $K$-invariant vectors and $\tau^{(P)}$ the parabolic descent of $\tau$ along $P$. Hence~\eqref{eq:spectral1} is bounded from above by
\[
\sum_{\pi\in\Pi_{\text{disc}}(M(\A)^1)} \left|\Mtr\pi^{K_\infty}(\tau^{(P)})\right| \int_{i\ka_L} \left|  \widehat{h}(\lambda+\lambda_{\pi_\infty}-\mu)\Mtr\left(\FmM_L(P,\pi,\lambda) M_{P|P}(0,s, \pi) (\Pi_K)_{|\pi} \right) \right| \, d\lambda.
\]
Since $|\Mtr\pi^{K_\infty}(\tau^{(P)})|\leq \|\tau\|_{L^1(G(\A_f))}$,
and the integral over $d\lambda$ is bounded in~\cite{LM09}*{(6.3)},
 the asserted bound for $\left| J_{M,s}(f^{\mu}_\pm\cdot\tau)\right|$ follows.

To finish the proof we still need to show that
\[
J_{\text{disc}}(f^{\mu}_\pm\cdot\tau) -\sum_{\pi\in \Pi_{\text{cusp}}(G(\A)^1)} \Mtr \pi(f^\mu_\pm \cdot \tau)
=\sum_{\pi\in \Pi_{\text{disc}}(G(\A)^1)\backslash \Pi_{\text{cusp}}(G(\A)^1)} \Mtr \pi(f^\mu_\pm \cdot \tau)
\]
is bounded from above by the right hand side given in the proposition.

In view of Proposition~\ref{prop:local_weyl_bound} below, we
have 
\[
\sum_{\pi\in\Pi_{\text{disc}}(G(\A)^1)\smallsetminus \Pi_{\text{cusp}}(G(\A)^1)}\!\!\!\!\!\! 
\left|\Mtr\pi^{K_\infty}(\tau)\right| |\widehat{h}(\lambda_{\pi_\infty}-\mu)|
\ll \|\tau\|_{L^1(G(\A_f))} (1+\|\mu\|)^{d-r-1},
\]
because for any integer $N \ge 0$ we have 
$|\widehat{h}(\nu)|\ll_N 
(1+\|\nu\|)^{-N}$. 
 This proves that the 
discrete but non--cuspidal automorphic representations only contribute to the 
error term given on the right hand side of the displayed inequality in the 
proposition, and it therefore finishes the proof of the proposition.
\end{proof}

\begin{corollary}\label{cor:spectral}
	There exists a constant $c>0$ depending only on $n$, $h$, and $\Omega$ such that the following holds. For every $\tau\in\CmH$ with $|\tau|\le1$, and every $t\ge 1$,
	\begin{multline*}
	\left|\int_{t\Omega} J_{\text{spec}}\left((f^{\mu}_\pm\cdot\tau)_{|G(\A)^1}\right)~d\mu
	- \sum_{\substack{\pi\in\Pi_{\text{cusp}}(G(\A)^1) \\ \lambda_{\pi}\in t\Omega}} \dim (\pi_{\infty}\otimes\chi_\pm)^{K_{\infty}} \Mtr \pi_f(\tau)\right|\\
	\le c \|\tau\|_{L^1(G(\A))} t^{d-1}(\log t)^{\max(3,n)}.
	\end{multline*}
\end{corollary}
\begin{proof}
	We have that $
	\int_{t\Omega} \Mtr \pi_\infty(f^\mu_\pm) d\mu$
	approximates
	$\dim(\pi_\infty \otimes \chi_\pm)^{K_\infty}$ if $\lambda_\pi\in t\Omega$ by~\cite{LM09}*{(4.6)} so that the corollary follows from Proposition~\ref{prop:spectral} by integrating its inequality over $\mu\in t\Omega$.
\end{proof}

\begin{remark}
	On the spectral side we have a better control over the dependence on $\tau$ compared to the geometric side. This is due to the fact that the proof of Proposition~\ref{prop:spectral} reduces to the (local) estimation of the operator norm of $\rho((f^{\mu}_\pm\cdot\tau) _{|G(\A)^1})$ on certain induced representations (the Jacquet-Shalika bound).   We could establish an even better bound saving a power of $\|\tau\|_{L^1(G(\A_f))}$ by writing down fully the Hecke eigenvalue of the block of the induced representations and applying bounds towards Ramanujan. For example for $\GL(2)$ the eigenvalues of Eisenstein series are units which indeed is often useful in applications of the Selberg trace formula. Ultimately the best possible bound could be deduced by induction on the Sato-Tate equidistribution for smaller groups. These improvements are more complicated to implement  so we have favored the present argument for Proposition~\ref{prop:spectral} which provides a straightforward separation of variables.
\end{remark}

\subsection{Bound on the number of residual representations}
If $\pi\in \Pi_{\text{disc}}(G(\A)^1)$ is spherical and non-cuspidal, it is 
known, for example by the 
classification of the discrete spectrum of $\GL(n)$ by Moeglin--Waldspurger, 
that $\pi_\infty$ is non--tempered, 
that 
is, $\lambda_{\pi_\infty}\in \ka_\C^*\smallsetminus i\ka^*$. 
Moreover $\Im \lambda_{\pi_\infty}\in (\ka_M^G)^*$ for some semistandard Levi subgroup 
$M\neq T_0$ by 
\cite[\S3]{LM09}.
We deduce the following estimate for 
the number of residual
representations in balls of varying radius:
\begin{proposition}\label{prop:local_weyl_bound}
 For $R\in \R_{\ge1}$ and $\mu\in i\ka^*$, the number of non--cuspidal, 
 spherical, everywhere unramified $\pi\in\Pi_{\text{disc}}(G(\A)^1)$    with 
 $\|\lambda_{\pi_\infty}-\mu\|\le R$ satisfies
 \[
 \left|
\{
\pi\in \Pi_{\text{disc}}(G(\A)^1)\setminus \Pi_{\text{cusp}}(G(\A)^1),\
\|\lambda_{\pi_\infty}-\mu\|\le R,\
\pi_\infty^{K^\circ_\infty} \neq 0,\
\pi_f^{K_f}\neq 0
\}\right|
 \ll (R+\|\mu\|)^{d-r-1} R^{r-1}
 \]
where the implied constant is independent of $R$ and  $\mu$ (it depends only on 
$n$).
\end{proposition}
\begin{proof}
For each semistandard Levi subgroup $M$, $R \ge1$, and $\xi\in \ka^*$ let 
$S_R^M(\xi)$ 
denote the set of all non--tempered, spherical, everywhere unramified $\pi\in 
\Pi_{\text{disc}}(G(\A)^1)$ with $\Im(\lambda_{\pi_\infty})\in (\ka_M^G)^*$ and 
$\|\Im(\lambda_{\pi_\infty})-\xi\|\le R$. 
Since $\Im \lambda_{\pi_\infty}\in (\ka_M^G)^*$ for some $M\neq T_0$, we have
\[
\left|
\{
\pi\in \Pi_{\text{disc}}(G(\A)^1)\setminus \Pi_{\text{cusp}}(G(\A)^1),\
\|\lambda_{\pi_\infty}-\mu\|\le R,\
\pi_\infty^{K^\circ_\infty} \neq 0,\
\pi_f^{K_f}\neq 0
\}\right|
\le
\sum_{M\neq T_0} \# S^M_R(\Im \mu),
\]
where the sum runs over all semistandard Levi subgroups $M\neq T_0$. 

Fix such a 
Levi subgroup $M\neq T_0$. 
We can find 
points $\xi_1,\ldots, \xi_{K}\in (\ka_M^G)^*$ with $K\ll R^{\dim\ka_M^G}\le R^{r-1}$ and 
$\|\xi_j\|\le \|\mu\|+R$ such that 
\[
 S^M_R(\Im \mu)
 \subseteq \bigcup_{j=1}^{K} S^M_1(\xi_j).
\]
We bound the number of elements in $S^M_1(\xi_j)$ from above by the number of 
all spherical, everywhere unramified $\pi\in \Pi_{\text{disc}}(G(\A)^1)$ with $\|\Im 
\lambda_{\pi_\infty}-\xi_j\|\le 1$ using Proposition~\ref{lem:local_upper} (note that the 
real part of $\lambda_{\pi_\infty}$ is uniformly bounded because $\pi_\infty$ 
is unitary). 

Since $\xi_j\in (\ka_M^G)^*$, 
we have $|\langle\alpha,\xi_j\rangle|=0$ for at least one positive root $\alpha$. Hence
\[
 \# S^M_1(\xi_j)
 \ll (1+\|\xi_j\|)^{d-r-1} 
 \ll (R+\|\mu\|)^{d-r-1}.
\]
Summing up, we obtain that 
\[
\# S^M_R(\Im \mu)
\ll K 
(R+\|\mu\|)^{d-r-1} 
\ll 
(R+\|\mu\|)^{d-r-1} R^{r-1},
\]
which concludes the proof of the proposition.
\end{proof}

\subsection{Hecke operators and a volume estimate}
The $L^1$-norm of a Hecke operator is equal to the volume of a $K_p$ double-coset which
can be estimated as follows.
\begin{lemma}\label{lem:degree}
	There exists a constant $c>0$ depending only on $n$, such that for every prime $p$ and every $\xi=(\xi_1,\ldots, \xi_n)$ with $\xi_1\ge\ldots\ge\xi_n$,
	\[
	p^{\xi_1-\xi_n}
	\le \vol_{G(\Q_p)} \left(K_p p^\xi K_p\right) =
	\|\tau_{p,\xi}\|_{L^1(G(\Q_p))}
	\le c p^{n^2 (\xi_1-\xi_n)/4}.
\]
\end{lemma}
\begin{proof}
	It follows from \cite{Gross:satake}*{Prop.7.4} that
	\[
	\vol_{G(\Q_p)} \left(K_p p^\xi K_p\right) =
	\frac{\#(G/P_\xi)(\F_p)}{p^{\dim (G/P_\xi)}} p^{\langle\xi,2\rho\rangle},
	\]
	with $\rho$ the half-sum of all positive roots of $T_0$ acting on the Borel 
	$P_0$.

	By a trivial bound on the dimension of cohomology, we deduce that
	the number of $\F_p$-points on the partial flag variety $G/P_\xi$ satisfies
	\[
	1
	\le \frac{\#(G/P_\xi)(\F_p)}{p^{\dim (G/P_\xi)}}
	\le \dim(H^*(G/P_\xi)) \le |W| = n!
	\]
	More precise upper-bounds can be obtained as follows. The above number of
	$\F_p$-points on $G/P_\xi$ is a $p$-multinomial coefficient,
	by writing
	\[
	\#G(\F_p) = p^{n^2} \prod_{1\le i \le n} (1-p^{-i}),
	\]
	and similarly for $\#P_\xi(\F_p)$.
	We find
	\[
	\frac{\#(G/P_\xi)(\F_p)}{p^{\dim (G/P_\xi)}} \le
	\frac{\#(G/P_0)(\F_p)}{p^{\dim (G/P_0)}}
	=(1-p^{-1})^{1-n} \prod_{2\le i \le n} (1-p^{-i})
	\le
	(1-p^{-1})^{-n} \le 2^n.
	\]

	Since
	$
	\langle\xi, 2\rho\rangle = \sum\limits_{1\le i < j \le n} (\xi_i-\xi_j) \ge \xi_1-\xi_n$, the lower-bound follows. Conversely, we have
	\[
	\sum_{1\le i < j \le n} (\xi_i-\xi_j) = \sum_{1\le i \le n-1} i(n-i)(\xi_i-\xi_{i+1}) 
	\le \left\lfloor \frac{n}{2} \right\rfloor 
\left\lceil \frac{n}{2} \right\rceil
 (\xi_1-\xi_n).
	\]
Since $\left\lfloor n/2 \right\rfloor 
\left\lceil n/2 \right\rceil \le n^2/4$, the upper-bound follows.
\end{proof}

\subsection{Conclusion of the proof of Theorem~\ref{thm:weyl}}
If $\tau$ is the characteristic function of a bi-$K_f$-invariant subset as in the theorem, then
$\tau=\sum_{\xi\in\Xi} \tau_{\xi}$ for some finite set $\Xi$ of tuples $\xi=(\xi^p)_p$ with $\xi^p=(\xi_1^p,\ldots,\xi^p_n)$ such that $\xi^p_1\ge\ldots\ge\xi^p_n$, $\xi^p=0$ for all but finitely many $p$, and $\tau_{\xi}=\prod_p \tau_{p, \xi^p}$.
By linearity it suffices to establish the estimate for a single $\tau_\xi$.

Every term appearing in  Theorem~\ref{thm:weyl} is invariant if we replace $\tau_\xi$ by
$\tau_{\tilde\xi}$ under the condition that for every $p$, $\xi^p-\tilde\xi^p=(a_p,\ldots, 
a_p)$ for
some integer $a_p\in\Z$, and thus we may assume without loss of generality that
$\xi_n^p=0$ for every $p$. Indeed, this invariance holds by combining the following:
$\Mtr\pi_p(\tau_{\xi^p})=\Mtr\pi_p(\tau_{\tilde\xi^p})$ for every spherical unramified
$\pi\in\Pi_{\text{cusp}}(G(\A)^1)$, since the central character of $\pi$ is then trivial;
  $\|\tau_\xi\|_{L^1(G(\A_f))} = \|\tau_{\tilde \xi}\|_{L^1(G(\A_f))} $ because of the
  invariance of Haar measures by translation; finally, $\sum_{\gamma\in
Z(\Q)/\{\pm1\}} \tau_\xi(\gamma)=\sum_{\gamma\in Z(\Q)/\{\pm1\}}
\tau_{\tilde\xi}(a^{-1}\gamma)=\sum_{\gamma\in Z(\Q)/\{\pm1\}} \tau_{\tilde\xi}(\gamma)$
with $a=\prod_p p^{a_p}\in \Q^\times \simeq Z(\Q)$.

Globally we take the test function $(f^{\mu}_+\cdot\tau_\xi)_{|G(\A)^1}$ or
$(f^{\mu}_-\cdot\tau_\xi)_{|G(\A)^1}$, depending on whether $\chi=\chi_+$ or
$\chi=\chi_-$.
Corollary~\ref{cor:spectral} relates the first term of Theorem~\ref{thm:weyl} and
$\int_{t\Omega} J_{\rm spec}((f^{\mu}_\pm\cdot\tau_\xi)_{|G(\A)^1})$, up to a remainder
term which is admissible.

Arthur's trace formula is $J_{\rm spec}=J_{\rm geom}$. We then apply
Corollary~\ref{cor:bound-Jgeom}, which relates
$\int_{t\Omega} J_{\rm geom}((f^{\mu}_\pm\cdot\tau_\xi)_{|G(\A)^1})$ to
$\Lambda_\Omega(t)\prod_{p<\infty} \delta(\xi^p)$,
together with Lemma~\ref{lem:degree} to show that the remainder term is admissible.
Since
\[
\prod_{p<\infty} \delta(\xi^p) = \sum_{\gamma \in Z(\Q)/\{\pm 1\}} \tau_\xi(\gamma),
\]
we recover the second term of Theorem~\ref{thm:weyl}, which concludes the proof.
\qed

\subsection{Local Weyl's law}
We record the following variant of Theorem~\ref{thm:weyl}.
\begin{theorem}
	There exist $\delta>0$ and $A<\infty$, depending only on $n$, and $c>0$ depending only on $n$ and $h$ such that for every $\tau$ the characteristic function of a compact bi-$K_f$-invariant set, and every $\mu\in i\Fma^*$,
\begin{equation*}
\left|\sum_{\pi\in\Pi_{\text{cusp}}(G(\A)^1)}  \Mtr \pi(f^\mu_\pm \cdot \tau)
-
2 \vol(G(\Q)\backslash G(\A)^1) f^\mu_+(1) \prod_{p<\infty}\delta(\xi^p)\right|
\le c  (1+\|\mu\|)^{d-r-\delta} ||\tau||^A_{L^1},
\end{equation*}
where $\delta(\xi^p)=1$ if  $\xi^p_1=\ldots=\xi_n^p$, and
$\delta(\xi^p)=0$ otherwise.
\end{theorem}

\begin{proof}
	Similarly as before, we reduce to the case $\tau=\tau_\xi$, and $\xi_n^p=0$.
	We apply Proposition~\ref{prop:spectral} for the spectral side.
For the geometric side, we apply Theorem~\ref{th:bound-Jgeom}, Proposition~\ref{p:unipotent-classes}.(i), together with Lemma~\ref{lem:degree}, and finally note that if $z=\Mdiag(a,\ldots,a) \in G(\Q)$ is central, $\left(f^\mu_\pm\cdot\tau\right)(z)=0$ unless $\xi^p_1=\ldots=\xi_n^p$ for every $p$. In the latter case, $\left(f^\mu_\pm\cdot\tau\right)(z)=0$ unless $|a|_p=p^{\xi^p_1}$ for every $p$ in which case $\left(f^\mu_\pm\cdot\tau\right)(z)
=f^\mu_+(1)$.
\end{proof}

\subsection*{Acknowledgments.}
 The authors would like to thank Erez Lapid and Peter Sarnak for various discussions
 along the years about this project.
NT thanks Sigurdur Helgason for an helpful discussion and confirming that Proposition~\ref{prop:spher:fct:bound} was new.
We thank Farrell Brumley, Simon Marshall, and Sug Woo Shin for helpful 
comments. We thank the referee for a careful reading.
Parts of this paper have been written while the authors were visiting the 
Israel Institute for Advanced Studies, Princeton University and the Hausdorff 
Institute in Bonn. We thank all of these institutions for their hospitality.
JM acknowledges support from the NSF under grant No. DMS-0932078 while in
residence at the MSRI, from the NSF under agreement No. DMS-1128155
while in residence at the IAS, and from Israel Science Foundation grant no.\ 1676/17. NT acknowledges support from the NSF under
agreement No. DMS-1200684.

\bibliographystyle{amsalpha}
\bibliography{glbib,allmac}

\end{document}